\documentclass[10pt, oneside]{article}
\usepackage{mathrsfs}
\usepackage{amsfonts}
\usepackage{amssymb} 
\usepackage{amsmath,amsfonts,amssymb,amsthm}
\usepackage{caption}
\usepackage{subfigure}
\usepackage{cite}
\usepackage{color}
\usepackage{graphicx}
\usepackage{subfigure}
\usepackage{float}
\usepackage{paralist}
\usepackage{indentfirst}
\usepackage{authblk}
\usepackage[colorlinks,
linkcolor=red,
citecolor=blue
]{hyperref}

\usepackage[T1]{fontenc}
\def\div{ \hbox{\rm div}\,  }

\newtheorem{theorem}{Theorem}[section]
\newtheorem{lemma}{Lemma}[section]
\newtheorem{prop}{Proposition}[section]

\newtheorem{remark}{Remark}[section]

\newtheorem{defn}{Definition}[section]

\def\var{\varepsilon}

\def\bma#1\ema{{\allowdisplaybreaks\begin{aligned}#1\end{aligned}}}

\numberwithin{equation}{section}

\begin{document}
\title{{\LARGE \textbf{Global existence and optimal time-decay rates of the compressible Navier-Stokes-Euler system}}}

\author[b,c]{Hai-Liang  Li  \thanks{
E-mail:		hailiang.li.math@gmail.com (H.-L. Li).}}

\author[b,c]{Ling-Yun Shou \thanks{E-mail: shoulingyun$\underline{~}$math@163.com(L.-Y Shou).}}
    \affil[b]{School of Mathematical Sciences,
	Capital Normal University, Beijing 100048, P.R. China}
\affil[c]{Academy for Multidisciplinary Studies, Capital Normal University, Beijing 100048, P.R. China}
\date{}

\renewcommand*{\Affilfont}{\small\it}
\maketitle
\begin{abstract}
In this paper, we consider the Cauchy problem of the multi-dimensional compressible Navier-Stokes-Euler system for two-phase flow motion, which consists of the isentropic compressible Navier-Stokes equations and the isothermal compressible Euler equations coupled with each other through a relaxation drag force. We first establish the local existence and uniqueness of the strong solution for general initial data in a critical homogeneous Besov space, and then prove the global existence of the solution if the initial data is a small perturbation of the equilibrium state. Moreover, under the additional condition that the low-frequency part of the initial perturbation also belongs to another Besov space with lower regularity, we obtain the optimal time-decay rates of the global solution toward the equilibrium state. These results imply that the relaxation drag force and the viscosity dissipation affect regularity properties and long time behaviors of solutions for the compressible Navier-Stokes-Euler system.

\end{abstract}
\section{Introduction}
%\thispagestyle{empty}

%Two-phase flow models arise in many applied scientific areas including manufacturing, biomedicine, power, nuclear, chemical process, petroleum, space, and cryogenic, cf. \cite{brennen1,bresch3,desv1,evje3,gidaspow1,ishii2,pros1,zuber1}. 

We consider the coupled compressible Navier-Stokes-Euler (NS-Euler) system for two-phase flow motion in $\mathbb{R}^{d}$ ($d\geq2$) as follows:
\begin{equation}\label{m1}
\left\{
\begin{aligned}
& \partial_{t}\rho+\div (\rho u)=0,\\
& \partial_{t}(\rho u)+\div (\rho u\otimes u)+\nabla P(\rho)=\mu \Delta u+(\mu+\lambda)\nabla \div u-\kappa n(u-w),\\
&\partial_{t}n+\div (nw)=0,\\
&\partial_{t}(n w)+\div ( n w\otimes w)+\nabla n=\kappa n(u-w),\quad\quad x\in\mathbb{R}^{d},\quad t>0,\\
\end{aligned}
\right.
\end{equation}
 with the initial data
 \begin{align}
 &(\rho,u,n,w)(x,0)=(\rho_{0},u_{0},n_{0},w_{0})(x)\rightarrow (\bar{\rho}, 0, \bar{n}, 0), \quad\quad |x|\rightarrow\infty,\label{d}
 \end{align}
where $(\bar{\rho}, 0, \bar{n}, 0)$ $(\bar{\rho},\bar{n}>0$) is the constant state. The unknowns are the densities $\rho=\rho(x,t)\geq0, n=n(x,t)\geq0$ and the velocities $u=u(x,t)\in\mathbb{R}^{d}, w=w(x,t)\in\mathbb{R}^{d}$. Furthermore, the pressure function $P(\rho)\in C^{\infty}(\mathbb{R}_{+})$ satisfies $P'(\rho)>0$ for $\rho>0$, the drag force coefficient $\kappa>0$ is a constant, and the viscosity coefficients $\mu$ and $\lambda$ satisfy 
\begin{equation}\nonumber
\begin{aligned}
\mu>0,\quad \quad 2\mu+\lambda>0.
\end{aligned}
\end{equation}
The NS-Euler system $(\ref{m1})$ was derived in \cite{choi3,choi1} as a hydrodynamic limit of the compressible Navier-Stokes-Vlasov-Fokker-Planck model describing the dynamics of small particles dispersed in a fluid, such as the sedimentation of suspensions, sprays, combustion \cite{jabin1,o1,will1}, and so on.

%The two-phase fluid system $(\ref{m1})$ was derived in \cite{choi3} as a fluid-dynamical limit of the following compressible Navier-Stokes-Vlasov-Fokker-Planck model subject to a local alignment force:
%\begin{equation}\label{NSVFP}
%\left\{
%\begin{aligned}
%&\partial_{t}\rho^{\var}+\div_{x}(\rho^{\var} u^{\var})=0,\\
%&\partial_{t}(\rho^{\var} u^{\var})+\div_{x} (\rho^{\var}  u^{\var}\otimes u^{\var})+\nabla_{x} P(\rho)=\mathcal{A} u^{\var}+\kappa\int_{\mathbb{R}^{d}}(\xi-u^{\var}) f^{\var} d\xi,\\
%&\partial_{t}f^{\var}+\xi\cdot \nabla_{x} f^{\var}=\div_{\xi}\big((\xi-u^{\var})f^{\var} \big)+\frac{1}{\var} \div_{\xi} \big( (\xi-w^{\var})+\nabla_{\xi}f^{\var} \big),\quad (x,\xi)\in\mathbb{R}^{d}\times\mathbb{R}^{d},\quad t>0,
%\end{aligned}
%\right.
%\end{equation}
%where $\rho^{\var}=\rho^{\var}(x,t)\in\mathbb{R}_{+}$ and $u^{\var}=u^{\var}(x,t)\in\mathbb{R}^{d}$, respectively, denote the density and velocity of compressible Navier-Stokes equations $(\ref{NSVFP})_{1}$-$(\ref{NSVFP})_{2}$, and $f^{\var}=f^{\var}(x,v,t)\in \mathbb{R}_{+}$ stands for the distribution function associated with the Vlasov-Fokker-Planck equation $(\ref{NSVFP})_{3}$. Such fluid-particle interaction equations arise in describing the dynamics of small particles dispersed in a fluid, such as the sedimentation of suspensions, sprays, combustion \cite{jabin1,o1,will1}, etc. 

There has been much important progress made recently on the analysis of two-phase flow models, refer to \cite{wu1,choi1,mellet1,mellet2,choi2,choi3,choi4,hcc1,lf1,wen0,evje6,bresch4,bresch5,lhl1,lhl2,lhl3,breschhand,bresch3,novotny2} and references therein. However, for the NS-Euler system \eqref{m1}, there are few important results to our knowledge on the well-posedness and asymptotic behaviors of solutions, cf. \cite{choi1,tang1,wu1,choi3}. If the initial data is a small perturbation of the equilibrium state in the three-dimensional Sobolev space $H^{s}$ ($s\geq 3$), Choi \cite{choi1} established the global existence and uniqueness of strong solutions  for \eqref{m1} in both the whole space and the periodic domain, and further got the exponential time stability in the periodic case. Choi and Jung \cite{choi3} showed the global existence and uniqueness of solutions to \eqref{m1} near the equilibrium state in a three-dimensional bounded domain. In addition, Wu, Zhang and Zou \cite{wu1} and Tang and Zhang \cite{tang1} obtained the optimal algebraic time-decay rates of global solutions for \eqref{m1} with Sobolev regularity in the three-dimensional whole space if the initial data further belongs to $L^1$. For the pressureless NS-Euler system (i.e., without the pressure term $\nabla n$ in $(\ref{m1})_{4}$), the global dynamics of strong solutions near the equilibrium state was studied in \cite{choi4}, and the finite-time blow-up phenomena of classical solutions was investigated in \cite{choi2}. When the density-dependent viscosity term $\div(n\mathbb{D} w)$ in $(\ref{m1})_{4}$ is taken into account, which can be derived from the Champman-Enskog expansion of compressible Navier-Stokes-Vlasov-Fokker-Planck model in \cite{lhl1}, the existence and nonlinear stability of steady-states to the inflow/outflow problem in the half line are proved in \cite{lhl2,lhl3}.

However, there are no any results about the NS-Euler system (\ref{m1}) in critical spaces. The main purpose in this paper is to study the well-posedness and optimal time-decay rates of solutions to the multi-dimensional Cauchy problem (\ref{m1})-\eqref{d} in the critical regularity framework. To be more precise, it is proved that for general initial data in the $L^2$-type critical Besov space, a unique strong solution to the Cauchy problem (\ref{m1})-\eqref{d} exists locally in time, which is shown to be a global one when the initial data is near the constant equilibrium state. In addition, the optimal algebra time-decay rates of the global solution to its equilibrium state are obtained under the additional mild assumption that the low-frequency part of the initial data is either bounded or small in another Besov space with lower regularity.

\vspace{2ex}

%In the present paper, the solution $(\rho,u,n,w)$ is considered to tend to the equilibrium state $(\bar{\rho},\overline{u},\bar{n},\overline{w})$ with the constants $\bar{\rho},\bar{n}>0$ and $\overline{u}, \overline{w}\in \mathbb{R}^{d}$.

 Without loss of generality, set 
 $$
  P'(\bar{\rho})=\kappa=\bar{\rho}=\bar{n}=\mu=1,\quad\quad \lambda=-1.
  $$
   Define the perturbation 
$$
a:=\rho-1,\quad \quad a_{0}:=\rho_{0}-1,\quad\quad b:=\log{n},\quad \quad b_{0}=\log{n_{0}}.
$$
Then, the Cauchy problem \eqref{m1}-\eqref{d} can be reformulated into
\begin{equation}\label{m1n}
\left\{
\begin{aligned}
&\partial_{t}a+\div u=-\div (au),\\
&\partial_{t}u+ \nabla a- \Delta u+u-w=-u\cdot\nabla  u+G,\\
&\partial_{t}b+\div w=-w\cdot\nabla b,\\
&\partial_{t}w+\nabla b+w-u=-w\cdot \nabla w,\quad\quad  x\in\mathbb{R}^{d},\quad t>0,\\
&(a,u,b,w)(x,0)=(a_{0},u_{0},b_{0},w_{0})(x)\rightarrow(0,0,0,0),\quad |x|\rightarrow\infty,
\end{aligned}
\right.
\end{equation}
where $G$ is the nonlinear term
\begin{equation}\label{G}
\begin{aligned}
G:=g(a)\nabla a+f(a) \Delta u+h(a,b) (u-w),
\end{aligned}
\end{equation}
with 
$$
g(a):=-\frac{P'(1+a)}{1+a}+ 1,\quad f(a):=-\frac{a}{a+1},\quad h(a,b):= (e^{b}-1)\frac{a}{a+1}+\frac{a}{a+1}-e^{b}+1.
$$

\vspace{2ex}

First, we have the local existence and uniqueness of the strong solution to the Cauchy problem (\ref{m1n}) for general initial data in the critical Besov space as follows:

\begin{theorem}\label{theorem11} 
Assume that the initial data $(a_{0},u_{0},b_{0},w_{0})$ satisfies
\begin{equation}\label{a1}
\begin{aligned}
 a_{0}\in\dot{B}^{\frac{d}{2}}_{2,1},\quad~ \inf_{x\in\mathbb{R}^{d}}(1+a_{0})(x)>0,\quad~ u_{0}\in\dot{B}^{\frac{d}{2}-1}_{2,1},\quad~ (b_{0},w_{0})\in \dot{B}^{\frac{d}{2}-1}_{2,1}\cap \dot{B}^{\frac{d}{2}+1}_{2,1}.
\end{aligned}
\end{equation}
Then, there exists a time $T>0$ such that the Cauchy problem $(\ref{m1n})$ admits a unique strong solution  $(a,u,b,w)$ satisfying for $t\in [0,T]$ that
\begin{equation}
\left\{
\begin{aligned}
& a\in \mathcal{C}([0,T];\dot{B}^{\frac{d}{2}}_{2,1}),\quad \inf_{(x,t)\in\mathbb{R}^{d}\times[0,T]}(1+a)(x,t)>0\\
&  u\in \mathcal{C}([0,T];\dot{B}_{2,1}^{\frac{d}{2}-1})\cap L^1(0,T;\dot{B}^{\frac{d}{2}+1}_{2,1}),\\
& (b,w)\in \mathcal{C}([0,T];\dot{B}^{\frac{d}{2}-1}_{2,1}\cap\dot{B}^{\frac{d}{2}+1}_{2,1}) .\label{r1}
\end{aligned}
\right.
\end{equation}
In addition to $(\ref{a1})$, if assume $a_{0}\in \dot{B}^{\frac{d}{2}-1}_{2,1}$, then $a\in \mathcal{C}([0,T];\dot{B}^{\frac{d}{2}-1}_{2,1})$ holds.
\end{theorem}

%\begin{remark}
%We would like to moention  that the homogeneous Besov space $\dot{B}^{\frac{d}{2}-1}_{2,1}\cap\dot{B}^{\frac{d}{2}+1}_{2,1}$ is embedded in $\mathcal{C}^1$ and weaker than usual inhomogeneous Besov space $B^{\frac{d}{2}+1}_{2,1}$ for the well-posedness of classical solutions to compressible Euler equations, cf. {\rm{\cite{xu0,xu1}}}.
%\end{remark}

 \vspace{2ex}

Then, we establish the global existence of the strong solution to the Cauchy problem $(\ref{m1n})$ for the initial data close to the equilibrium state below:

\begin{theorem}\label{theorem12}
For any $d\geq2$, there exists a constant $\varepsilon_{0}>0$ such that if the initial data $(a_{0},u_{0},b_{0},w_{0})$ satisfies $a_{0}\in \dot{B}^{\frac{d}{2}-1}_{2,1}\cap\dot{B}^{\frac{d}{2}}_{2,1}$, $u_{0}\in\dot{B}^{\frac{d}{2}-1}_{2,1}$, $(b_{0},w_{0})\in\dot{B}^{\frac{d}{2}-1}_{2,1}\cap\dot{B}^{\frac{d}{2}+1}_{2,1}$ and
\begin{equation}\label{a2}
\begin{aligned}
\mathcal{X}_{0}:= \|(a_{0},u_{0},b_{0},w_{0})^{\ell}\|_{\dot{B}^{\frac{d}{2}-1}_{2,1}}+\|a_{0}^{h}\|_{\dot{B}^{\frac{d}{2}}_{2,1}}+\|u_{0}^{h}\|_{\dot{B}^{\frac{d}{2}-1}_{2,1}}+\|(b_{0},w_{0})^{h}\|_{\dot{B}_{2,1}^{\frac{d}{2}+1}}\leq \varepsilon_{0},
\end{aligned}
\end{equation}
then the Cauchy problem $(\ref{m1n})$ admits a unique global strong solution $(a,u,b,w)$, which satisfies
\begin{equation}\label{rglobal}
\left\{
\begin{aligned}
&a^{\ell}\in \mathcal{C}_{b}(\mathbb{R}^{+};\dot{B}^{\frac{d}{2}-1}_{2,1})\cap L^1(\mathbb{R}^{+};\dot{B}^{\frac{d}{2}+1}_{2,1}),\quad\quad a^{h}\in \mathcal{C}_{b}(\mathbb{R}^{+};\dot{B}^{\frac{d}{2}}_{2,1})\cap L^1(\mathbb{R}^{+};\dot{B}^{\frac{d}{2}}_{2,1}),\\
&u^{\ell}\in \mathcal{C}_{b}(\mathbb{R}^{+};\dot{B}^{\frac{d}{2}-1}_{2,1})\cap L^1(\mathbb{R}^{+};\dot{B}^{\frac{d}{2}+1}_{2,1}),\quad\quad u^{h}\in \mathcal{C}_{b}(\mathbb{R}^{+};\dot{B}^{\frac{d}{2}-1}_{2,1})\cap L^1(\mathbb{R}^{+};\dot{B}^{\frac{d}{2}+1}_{2,1}),\\
&b^{\ell}\in \mathcal{C}_{b}(\mathbb{R}^{+};\dot{B}^{\frac{d}{2}-1}_{2,1})\cap L^1(\mathbb{R}^{+};\dot{B}^{\frac{d}{2}+1}_{2,1}),\quad\quad ~b^{h}\in \mathcal{C}_{b}(\mathbb{R}^{+};\dot{B}^{\frac{d}{2}+1}_{2,1})\cap L^1(\mathbb{R}^{+};\dot{B}^{\frac{d}{2}+1}_{2,1}),\\
&w^{\ell}\in \mathcal{C}_{b}(\mathbb{R}^{+};\dot{B}^{\frac{d}{2}-1}_{2,1})\cap L^1(\mathbb{R}^{+};\dot{B}^{\frac{d}{2}+1}_{2,1}),\quad\quad w^{h}\in \mathcal{C}_{b}(\mathbb{R}^{+};\dot{B}^{\frac{d}{2}+1}_{2,1})\cap L^1(\mathbb{R}^{+};\dot{B}^{\frac{d}{2}+1}_{2,1}),\\
& (u-w)^{\ell}\in  L^1(\mathbb{R}_{+};\dot{B}^{\frac{d}{2}}_{2,1})\cap L^2(\mathbb{R}^{+};\dot{B}^{\frac{d}{2}-1}_{2,1}),
\end{aligned}
\right.
\end{equation}
and
\begin{equation}\label{XX0}
\begin{aligned}
&\|(a,u,b,w)\|_{L^{\infty}_{t}(\dot{B}^{\frac{d}{2}-1}_{2,1})}^{\ell}+\|(a,u,b,w)\|_{L^1_{t}(\dot{B}^{\frac{d}{2}+1}_{2,1})}^{\ell}+\|u-w\|_{L^1_{t}(\dot{B}^{\frac{d}{2}}_{2,1})}^{\ell}+\|u-w\|_{L^2_{t}(\dot{B}^{\frac{d}{2}-1}_{2,1})}^{\ell}\\
&\quad\quad+\| a\|_{L^{\infty}_{t}(\dot{B}^{\frac{d}{2}}_{2,1})}^{h}+\|u\|_{L^{\infty}_{t}(\dot{B}^{\frac{d}{2}-1}_{2,1})}^{h}+\|(b,w)\|_{L^{\infty}_{t}(\dot{B}^{\frac{d}{2}+1}_{2,1})}^{h}+\|a\|_{L^1_{t}(\dot{B}^{\frac{d}{2}}_{2,1})}^{h}+\|(u,b,w)\|_{L^1_{t}(\dot{B}^{\frac{d}{2}+1}_{2,1})}^{h}\\
&\quad\leq C\mathcal{X}_{0},\quad\quad t>0,
\end{aligned}
\end{equation}
for $C>0$ a constant independent of time.
\end{theorem}

%\begin{equation}
%\begin{aligned}
%&\|(a,u,b,w)\|_{L^{\infty}_{t}(\dot{B}^{\frac{d}{2}-1}_{2,1})}^{\ell}+\| a\|_{L^{\infty}_{t}(\dot{B}^{\frac{d}{2}}_{2,1})}^{h}+\|u\|_{L^{\infty}_{t}(\dot{B}^{\frac{d}{2}-1}_{2,1})}^{h}+\|(b,w)\|_{L^{\infty}_{t}(\dot{B}^{\frac{d}{2}+1}_{2,1})}^{h}\\
%&\quad\quad+\|(a,u,b,w)\|_{L^1_{t}(\dot{B}^{\frac{d}{2}+1}_{2,1})}^{\ell}+\|u-w\|_{L^1_{t}(\dot{B}^{\frac{d}{2}}_{2,1})}^{\ell}+\|u-w\|_{L^2_{t}(\dot{B}^{\frac{d}{2}-1}_{2,1})}^{\ell}\\
%&\quad\quad+\|(b,w)\|_{L^{\infty}_{t}(\dot{B}^{\frac{d}{2}+1}_{2,1})}^{h}+\|a\|_{L^1_{t}(\dot{B}^{\frac{d}{2}}_{2,1})}^{h}+\|(u,b,w)\|_{L^1_{t}(\dot{B}^{\frac{d}{2}+1}_{2,1})}^{h}\\
%&\quad\leq C\mathcal{X}_{0},\quad\quad t>0,
%\end{aligned}
%\end{equation}

For Besov spaces, the readers can refer to Definitions \ref{defnbesov}-\ref{defntimespace} in Appendix. 

%&a^{\ell}\in \mathcal{C}_{b}(\mathbb{R}_{+};\dot{B}^{\frac{d}{2}-1}_{2,1})\cap L^1(\mathbb{R}_{+};\dot{B}^{\frac{d}{2}+1}_{2,1}),\quad\quad a^{h}\in \mathcal{C}_{b}(\mathbb{R}_{+};\dot{B}^{\frac{d}{2}}_{2,1})\cap L^1(\mathbb{R}_{+};\dot{B}^{\frac{d}{2}}_{2,1}),\\
%&u^{\ell}\in \mathcal{C}_{b}(\mathbb{R}_{+};\dot{B}^{\frac{d}{2}-1}_{2,1})\cap L^1(\mathbb{R}_{+};\dot{B}^{\frac{d}{2}+1}_{2,1}),\quad\quad u^{h}\in \mathcal{C}_{b}(\mathbb{R}_{+};\dot{B}^{\frac{d}{2}-1}_{2,1})\cap L^1(\mathbb{R}_{+};\dot{B}^{\frac{d}{2}+1}_{2,1}),\\
%&b^{\ell}\in \mathcal{C}_{b}(\mathbb{R}_{+};\dot{B}^{\frac{d}{2}-1}_{2,1})\cap L^1(\mathbb{R}_{+};\dot{B}^{\frac{d}{2}+1}_{2,1}),\quad\quad~ b^{h}\in \mathcal{C}_{b}(\mathbb{R}_{+};\dot{B}^{\frac{d}{2}+1}_{2,1})\cap L^1(\mathbb{R}_{+};\dot{B}^{\frac{d}{2}+1}_{2,1}),\\
%&w^{\ell}\in \mathcal{C}_{b}(\mathbb{R}_{+};\dot{B}^{\frac{d}{2}-1}_{2,1})\cap L^1(\mathbb{R}_{+};\dot{B}^{\frac{d}{2}+1}_{2,1}),\quad\quad w^{h}\in \mathcal{C}_{b}(\mathbb{R}_{+};\dot{B}^{\frac{d}{2}+1}_{2,1})\cap L^1(\mathbb{R}_{+};\dot{B}^{\frac{d}{2}+1}_{2,1}),\\
%& (u-w)^{\ell}\in \widetilde{L}^2(\mathbb{R}_{+};\dot{B}^{\frac{d}{2}-1}_{2,1}),

%\begin{remark}
%For the pure compressible Euler equations, the singularities may develop in finite time even for small and smooth initial data {\rm\cite{john1,liu1,lax1}}. On the other hand, the damping term of the velocity for compressible Euler equations can prevent the development of shock waves in finite time {\rm{\cite{nas1,hsiao1,si1,wang1}}}. 
%\end{remark}

\begin{remark}
The regularity $L^1(\mathbb{R}_{+};\dot{B}^{\frac{d}{2}+1}_{2,1})$ of the velocity $w$ in \eqref{rglobal} comes essentially from the coupling of the relaxation drag force term on the relative velocity $u-w$ and the viscosity dissipation on the velocity $u$. Furthermore, due to the influences of the relaxation drag force term, the regularity $L^1(\mathbb{R}_{+},\dot{B}^{\frac{d}{2}}_{2,1})\cap L^2(\mathbb{R}_{+}, \dot{B}^{\frac{d}{2}-1}_{2,1})$ of the relative velocity $u-w$ in \eqref{rglobal} is stronger than the regularity $L^1(\mathbb{R}_{+},\dot{B}^{\frac{d}{2}+1}_{2,1})\cap L^2(\mathbb{R}_{+}, \dot{B}^{\frac{d}{2}}_{2,1})$ of the solution $(a,u,b,w)$ for low frequencies.
\end{remark}

%\begin{remark}
%The critical initial assumption $\dot{B}^{\frac{d}{2}-1}_{2,1}\cap\dot{B}^{\frac{d}{2}}_{2,1}\times\dot{B}^{\frac{d}{2}-1}_{2,1}$ in $(\ref{a1})$ of the compressible Navier-Stokes equations $(\ref{m1n})_{1}$-$(\ref{m1n})_{2}$ are very different from the critical initial assumption $\dot{B}^{\frac{d}{2}-1}_{2,1}\cap\dot{B}^{\frac{d}{2}+1}_{2,1}$ in $(\ref{a1})$ of the compressible Euler equations $(\ref{m1n})_{3}$-$(\ref{m1n})_{4}$. To balance such two kinds of regularities for initial data, we observe and prove hat the low-order estimates of the solution $(a,u,b,w)$ to the two-phase flow model~{\rm\eqref{m1n}} can give the parabolic regularity $u\in L^1_t(\dot{B}^{\frac{d}{2}+1}_{2,1})$, which is enough to make the high-order estimates of $(b,w)$ in terms of the damped compressible Euler equations $(\ref{m1n})_{3}$-$(\ref{m1n})_{4}$.
%\end{remark}

Moreover, we have the optimal time-decay rates of the global solution to the Cauchy problem \eqref{m1n} if the low-frequency part of the initial data is further bounded in $\dot{B}^{\sigma_{0}}_{2,\infty}$ for $\sigma_{0}\in [-\frac{d}{2},\frac{d}{2}-1)$ as follows:
\begin{theorem}\label{theorem13}
For any $d\geq2$, let the assumptions of Theorem \ref{theorem12} hold, and $(a,u,b,w)$ be the corresponding global strong solution to the Cauchy problem $(\ref{m1n})$ given by Theorem \ref{theorem12}. If the initial data $(a_{0}, u_{0}, b_{0}, w_{0})$ further satisfies its low-frequency part 
\begin{equation}\label{a3}
\begin{aligned}
(a_{0}, u_{0}, b_{0}, w_{0})^{\ell}\in \dot{B}^{\sigma_{0}}_{2,\infty}\quad\quad\text{for}\quad\sigma_{0}\in [-\frac{d}{2},\frac{d}{2}-1),
\end{aligned}
\end{equation}
then it holds for any $t\geq1$ that
\begin{equation}\label{decay1}
\left\{
\begin{aligned}
&\|(a,u,b,w)(t)\|_{\dot{B}^{\sigma}_{2,1}}^{\ell}\leq C\delta_{0}(1+ t)^{-\frac{1}{2}(\sigma-\sigma_{0})},\quad\quad\quad\sigma\in(\sigma_{0},\frac{d}{2}-1],\\
&\|a(t)\|_{\dot{B}^{\frac{d}{2}}_{2,1}}^{h}+\|(u,b,w)(t)\|_{\dot{B}^{\frac{d}{2}+1}_{2,1}}^{h}\leq C\delta_{0}(1+ t)^{-\frac{1}{2}(\frac{d}{2}-1-\sigma_{0})},\\
&\|(u-w)(t)\|_{\dot{B}^{\sigma_{0}}_{2,\infty}}^{\ell}\leq C\delta_{0}  (1+t)^{-\sigma_{*}},
\end{aligned}
\right.
\end{equation}
with a constant $C>0$ independent of time, $\sigma_{*}:= \min\{\frac{1}{2},\frac{1}{2}(\frac{d}{2}-1-\sigma_{0})\}>0$ and
\begin{equation}\label{D0}
\begin{aligned}
&\delta_{0}:=\|(a_{0},u_{0},b_{0},w_{0})^{\ell}\|_{\dot{B}^{\sigma_{0}}_{2,\infty}}+\|a_{0}^{h}\|_{\dot{B}^{\frac{d}{2}}_{2,1}}+\|u_{0}^{h}\|_{\dot{B}^{\frac{d}{2}-1}_{2,1}}+\|(b_{0},w_{0})^{h}\|_{\dot{B}_{2,1}^{\frac{d}{2}+1}}.
\end{aligned}
\end{equation}
Furthermore, for $d\geq3$ and $\sigma_{0}\in [-\frac{d}{2},\frac{d}{2}-2)$, the relative velocity $u-w$ satisfies
\begin{equation}\label{decay11}
\begin{aligned}
&\|(u-w)(t)\|_{\dot{B}^{\sigma}_{2,1}}^{\ell}\leq C\delta_{0} (1+t)^{-\frac{1}{2}(1+\sigma-\sigma_{0})},\quad\quad \sigma\in(\sigma_{0},\frac{d}{2}-2].
\end{aligned}
\end{equation}
\end{theorem}

%\begin{remark}
%It should be noted that for $\sigma_{0}=-\frac{d}{2}$, the time-decay rates \eqref{decay1} of the global solution $(a,u,b,w)$ to the Cauchy problem \eqref{m1n} are optimal. Indeed, due to the embeddings $L^1\hookrightarrow\dot{B}^{-\frac{d}{2}}_{2,\infty}$ and $\dot{B}^{s}_{2,1}\hookrightarrow \dot{H}^{s}$, it was shown in {\rm{\cite{wu1}}} that the solution $(a,u,b,w)$ in $\dot{B}^{\sigma}_{2,1}$ has lower bounds of the time-decay rates $(1+t)^{-\frac{1}{2}(\sigma-\frac{d}{2})}$ for some initial data.
%\end{remark}

\begin{remark}
Theorem \ref{theorem13} implies that the solution $(a,u,b,w)$ to the Cauchy problem \eqref{m1n} decays at the same rate $(1+ t)^{-\frac{1}{2}(\sigma-\sigma_{0})}$ in $\dot{B}^{\sigma}_{2,1}$ as the solution of the heat equation with initial data in $\dot{B}^{\sigma_{0}}_{2,\infty}${\rm;} however, due to the dissipation effect of the relaxation drag force, the relative velocity $u-w$ decays at the faster rate $(1+ t)^{-\frac{1}{2}(1+\sigma-\sigma_{0})}$ in $\dot{B}^{\sigma}_{2,1}$.
\end{remark}

\vspace{1ex}

When the low-frequency part of the initial data is suitably small in $\dot{B}^{\sigma_{0}}_{2,\infty}$ for $\sigma_{0}\in [-\frac{d}{2},\frac{d}{2}-1)$, we can also prove the optimal time-decay rates of the global solution to the Cauchy problem \eqref{m1n}.

\begin{theorem}\label{theorem14}
For any $d\geq2$, let the assumptions of Theorem \ref{theorem12} hold, and $(a,u,b,w)$ be the global strong solution to the Cauchy problem $(\ref{m1n})$ given by Theorem \ref{theorem12}. There exists a constant $\varepsilon_{1}>0$ such that if the initial data $(a_{0}, u_{0}, b_{0}, w_{0})$ further satisfies
\begin{equation}\label{a4}
\begin{aligned}
\|(a_{0},u_{0},b_{0},w_{0})^{\ell}\|_{\dot{B}^{\sigma_{0}}_{2,\infty}} \leq \varepsilon_{1}\quad\quad\text{for}\quad\sigma_{0}\in [-\frac{d}{2},\frac{d}{2}-1),
\end{aligned}
\end{equation}
then it holds for any $t\geq1$ that
\begin{equation}\label{r4}
\left\{
\begin{aligned}
&\|(a,u,b,w)(t)\|_{\dot{B}^{\sigma}_{2,1}}^{\ell}\leq C \delta_{0} (1+t)^{-\frac{1}{2}(\sigma-\sigma_{0})},\quad\quad \sigma\in (\sigma_{0},\frac{d}{2}+1],\\
&\| a(t)\|_{\dot{B}^{\frac{d}{2}}_{2,1}}^{h}+\|(u,b,w)(t)\|_{\dot{B}^{\frac{d}{2}+1}_{2,1}}^{h}\leq C\delta_{0}(1+ t)^{-\frac{1}{2}(d+1-2\sigma_{0}-2\varepsilon)},\\
&\|(u-w)(t)\|_{\dot{B}^{\sigma_{0}}_{2,\infty}}^{\ell}\leq  C \delta_{0} (1+t)^{-\frac{1}{2}},\\
&\|(u-w)(t)\|_{\dot{B}^{\sigma}_{2,1}}^{\ell}\leq C \delta_{0} (1+t)^{-\frac{1}{2}(1+\sigma-\sigma_{0})},\quad\quad  \sigma\in (\sigma_{0},\frac{d}{2}],
\end{aligned}
\right.
\end{equation}
where $\delta_{0}$ is denoted by \eqref{D0}, $C>0$ is a constant independent of time, and $\varepsilon\in(0,1]$ is any small constant.
\end{theorem}

\begin{remark}
Compared with the time-decay rates of the solution $(a,u,b,w)$ to the Cauchy problem $(\ref{m1n})$ in Theorem \ref{theorem13}, it is shown in Theorem \ref{theorem14} that the low-frequency part $(a,u,b,w)^{\ell}$ has the optimal time-decay rate $(1+t)^{-\frac{1}{2}(\sigma-\sigma_{0})}$ in $\dot{B}^{\sigma}_{2,1}$ for higher regularity indexes $\sigma\in(\frac{d}{2}-1,\frac{d}{2}+1]$, the high-frequency part $(a,u,b,w)^{h}$ decays at the faster rate $(1+t)^{-\frac{1}{2}(d+1-2\sigma_{0}-2\varepsilon)}$ in the same Besov space, and furthermore the relative velocity $u-w$ decays at the rate $(1+t)^{-\frac{1}{2}(1+\sigma-\sigma_{0})}$ in $\dot{B}^{\sigma}_{2,1}$ for higher regularity indexes $\sigma\in(\frac{d}{2}-2,\frac{d}{2}]$ without the restrictions $d\geq3$ and $\sigma_{0}<\frac{d}{2}-2$.
\end{remark}

\begin{remark}
By $(\ref{r4})$ and interpolation arguments, for $\Lambda:=(-\Delta)^{\frac{1}{2}}$,  $p\geq2$ and $t\geq 1$, the following optimal $L^{p}$ time-decay rates hold$:$
\begin{equation}\label{decay21}
\left\{
\begin{aligned}
&\|\Lambda^{\sigma}a(t)\|_{L^{p}}\lesssim (1+t)^{-\frac{1}{2}(\sigma+\frac{d}{2}-\frac{d}{p}-\sigma_{0})},\quad\quad\quad\quad ~\quad\sigma+\frac{d}{2}-\frac{d}{p}\in(\sigma_{0},\frac{d}{2}],\\
&\|\Lambda^{\sigma}(u,b,w)(t)\|_{L^{p}}\lesssim (1+t)^{-\frac{1}{2}(\sigma+\frac{d}{2}-\frac{d}{p}-\sigma_{0})},\quad\quad~~\sigma+\frac{d}{2}-\frac{d}{p}\in(\sigma_{0},\frac{d}{2}+1],\\
&\|\Lambda^{\sigma}(u-w)(t)\|_{L^{p}}\lesssim (1+t)^{-\frac{1}{2}(1+\sigma+\frac{d}{2}-\frac{d}{p}-\sigma_{0})},\quad\quad\sigma+\frac{d}{2}-\frac{d}{p}\in(\sigma_{0},\frac{d}{2}].
\end{aligned}
\right.
\end{equation}
\end{remark}

%Different from the compressible Navier-Stokes equations \cite{charve1,chen1,danchin1,hospot1} or the one-velocity two-phase flow equations \cite{hcc1,zhang1}, the viscous part in $(\ref{m1n})_{1}$-$(\ref{m1n})_{2}$ and $(\ref{m1n})_{3}$-$(\ref{m1n})_{4}$ are not symmetric, and we can't control the two fluids separately due to the linear terms $\kappa(u-w)$ in $(\ref{m1n})_{2}$ and  $\kappa(w-u)$ in $(\ref{m1n})_{4}$. In addition, the nonlinear term $ h(a,b)(u-w)$ lacks one derivative compared with usual nonlinear terms. 

\vspace{2ex}

We would like to mention that the important progress has been obtained about the well-posedness and  optimal time-decay rates of solutions to the Cauchy problem for the isentropic compressible Navier-Stokes equations in the $L^2$-type or $L^p$-type critical Besov spaces, refer to  \cite{danchin1,danchin2,okita1,charve1,chen1,haspot1,danchin5,xin1,xu3} and references therein. Complete overviews on Fourier analysis methods for the compressible Navier-Stokes equations are presented in \cite{bahouri1,danchin4}.

%with critical regularity. In the $L^2$-type critical Besov spaces, Danchin established the global existence theory for small initial perturbations in \cite{danchin1} and investigated the local well-posedness issue for general initial data in  \cite{danchin2}. Subsequently, the results in \cite{danchin1,danchin2} has been extended to the $L^p$-type critical Besov spaces by the works  \cite{charve1,chen1,fang1,haspot1}. There are some interesting progress on the optimal time-decay rates of solutions for the isentropic compressible Navier-Stokes equations in the critical Besov spaces, refer to \cite{okita1,danchin5,xin1,xu3} and references therein.

Meanwhile, for the Cauchy problem of  the compressible Euler equations with damping,  the global existence and optimal time-decay rates of small classical solutions with critical regularity were investigated in either the inhomogeneous Besov spaces \cite{xu0,xu1,xu2} or the homogeneous setting \cite{CBD1,c2}.

We explain the main ideas to prove above Theorems \ref{theorem12}-\ref{theorem14} about the global existence and optimal time-decay rates of the strong solution to the Cauchy problem \eqref{m1n} in the framework of critical Besov spaces. To promise the critical control of the convective terms in \eqref{m1n},  the key challenge arises from the gain of $L^1$ time integrability of the Lipschitz norms for $u$ and $w$, i.e.,
\begin{equation}
\begin{aligned}
&\int_{0}^{t}\|(\nabla u,\nabla w)\|_{L^{\infty}}dt\lesssim \int_{0}^{t}\|(\nabla u,\nabla w)\|_{\dot{B}^{\frac{d}{2}+1}_{2,1}}dt<\infty.
\end{aligned}
\end{equation}
The NS-Euler model $\eqref{m1}$ (i.e., $\eqref{m1n}_{1}$-$\eqref{m1n}_{4}$) can be viewed as a coupled system of the compressible Navier-Stokes equations $\eqref{m1}_{1}$-$\eqref{m1}_{2}$ and the compressible Euler equations $\eqref{m1}_{3}$-$\eqref{m1}_{4}$ through the drag force source terms $n(w-u)$ and $n(u-w)$, respectively. However, in order to derive the regularity $L^1(\mathbb{R}_{+},\dot{B}_{2,1}^{\frac{d}{2}+1})$ for the velocity $u$ in $\eqref{m1}_{1}$-$\eqref{m1}_{2}$ as that made in \cite{danchin1} for the compressible Navier-Stokes equations, we require the regularity $L^1(\mathbb{R}_{+},\dot{B}_{2,1}^{\frac{d}{2}-1})$ for $w$ in the source term $n(w-u)$. Meanwhile, to get the regularity $L^1(\mathbb{R}_{+},\dot{B}^{\frac{d}{2}}_{2,1}\cap \dot{B}_{2,1}^{\frac{d}{2}+1})$ for the velocity $w$ in $\eqref{m1}_{3}$-$\eqref{m1}_{4}$ by similar arguments used in \cite{c2} for the compressible Euler equations with damping, we need the regularity  $L^1(\mathbb{R}_{+},\dot{B}_{2,1}^{\frac{d}{2}-1}\cap\dot{B}^{\frac{d}{2}+1}_{2,1})$ for $u$ in the source term $n(u-w)$. Unfortunately, the regularity $L^1(\mathbb{R}_{+},\dot{B}_{2,1}^{\frac{d}{2}-1})$ required for the velocity $w$ in $\eqref{m1}_{1}$-$\eqref{m1}_{2}$ is not consistent with the critical regularity $L^1(\mathbb{R}_{+},\dot{B}_{2,1}^{\frac{d}{2}}\cap\dot{B}^{\frac{d}{2}+1}_{2,1})$ for $w$ in $\eqref{m1}_{3}$-$\eqref{m1}_{4}$, and neither is for the velocity $u$. In addition, since the NS-Euler  model \eqref{m1} does not satisfy the well-known ``Shizuta-Kawashima'' condition in the study of hyperbolic-parabolic composite systems (cf. \cite{kaw1,shi1,c2,xu1}) due to the drag force term, it is not obvious to analyze the dissipative structures of $w$ for \eqref{m1}. These cause essential difficulties to enclose the uniform-in-time a-priori estimates of the local solution in the critical Besov space and extend it globally in time.

To overcome these difficulties, we construct some suitable Lyapunov functionals in the spirit of hypocoercivity. More precisely, it is observed that for any $j\in\mathbb{Z}$, there exists a Lyapunov functional $\mathcal{E}_{j}(t)\sim \|\dot{\Delta}_{j}(a,\nabla a,u,b,w)\|_{L^2}^2$ such that
\begin{equation}\label{L1}
\begin{aligned}
\frac{d}{dt}\mathcal{E}_{j}(t)+\min\{1,2^{2j}\}\|\dot{\Delta}_{j}(a,b)\|_{L^2}^2+2^{2j}\|\dot{\Delta}_{j}u\|_{L^2}^2
+\|\dot{\Delta}_{j}(u-w)\|_{L^2}^2\lesssim {\text{nonlinear terms}} \sqrt{\mathcal{E}_{j}(t)}.
\end{aligned}
\end{equation}
In order to establish the estimates of the velocity $w$, we make full use of the dissipative properties of the relative velocity $u-w$ and the velocity $u$ as follows:
\begin{eqnarray}\label{dissipationw}
2^{2j}\|\dot{\Delta}_{j}u\|_{L^2}^2+\|\dot{\Delta}_{j}(u-w)\|_{L^2}^2\gtrsim 
\begin{cases}
2^{2j}\|\dot{\Delta}_{j}w\|_{L^2}^2,\quad
& \mbox{if $j\leq 0,$ } \\
\|\dot{\Delta}_{j}w\|_{L^2}^2,\quad
& \mbox{if $j\geq-1.$}
\end{cases}
\end{eqnarray}
With the help of \eqref{L1}-\eqref{dissipationw}, we derive both the $\widetilde{L}^{\infty}_{t}(\dot{B}^{\frac{d}{2}-1}_{2,1})\cap L^1_{t}(\dot{B}^{\frac{d}{2}+1}_{2,1})$-estimate of $(a, u,b,w)$ for low frequencies and the $\widetilde{L}^{\infty}_{t}(\dot{B}^{\frac{d}{2}-1}_{2,1})\cap L^1_{t}(\dot{B}^{\frac{d}{2}-1}_{2,1})$-estimate of $(\nabla a, u,b,w)$ for high frequencies. With these bounds, we then are able to have  $L^1_{t}(\dot{B}^{\frac{d}{2}+1}_{2,1})$-bound of $u$ for high frequencies treating $\eqref{m1n}_2$ as a heat equation with some given sources. In addition, in order to the higher order $\widetilde{L}^{\infty}_{t}(\dot{B}^{\frac{d}{2}+1}_{2,1})\cap L_{t}^1(\dot{B}^{\frac{d}{2}+1}_{2,1})$-estimates of $(b,w)$, we construct another Lyapunov functional $\mathcal{E}_{3,j}\sim  \|\dot{\Delta}_{j}(b,w)\|_{L^2}^2$ such that for 
\begin{equation}\label{L2}
\begin{aligned}
\frac{d}{dt}\mathcal{E}_{3,j}(t)+\|\dot{\Delta}_{j}(b,w)\|_{L^2}^2\lesssim \Big(\|\dot{\Delta}_{j}u\|_{L^2}+{\text{nonlinear terms}}\Big) \sqrt{\mathcal{E}_{3,j}(t)}.
\end{aligned}
\end{equation}
This gives the expected estimates of $(b,w)$ employing the $L^1_{t}(\dot{B}^{\frac{d}{2}+1}_{2,1})$-estimate of $u$ for high frequencies obtained by the viscosity term.

%Therefore, we are able to estimate $u$ and $w$ in $L^1(\mathbb{R}_{+};\dot{B}^{\frac{d}{2}+1}_{2,1})$ instead of $L^1(\mathbb{R}_{+};\dot{B}^{\frac{d}{2}-1}_{2,1})$ and estimate the linearized equations uniformly in time.

% to match the critical regularities for the velocities and the source terms.

%The combination of the low order and higher order estimates gives rise to the consistency of the different critical regularities of solutions for the compressible Navier-Stokes equations $\eqref{m1n}_{1}$-$\eqref{m1n}_{2}$ and the compressible Euler equations $\eqref{m1n}_{3}$-$\eqref{m1n}_{4}$.

However, due to the difficulty caused by the nonlinear term $h(a,b)(u-w)$ in $\eqref{m1n}_{2}$, the regularities estimates of the solution $(a,u,b,w)$ are not enough to enclose the a-priori estimates. To overcome this difficulty, we observe that the relative velocity $u-w$ satisfies
\begin{equation}\label{relativedamp1}
\begin{aligned}
\partial_{t}(u-w)+2(u-w)=-\nabla a+ \Delta u+\nabla b-u\cdot \nabla u+w\cdot \nabla w+G.
\end{aligned}
\end{equation}
Employing the estimates of $(a,u,b,w)$ and \eqref{relativedamp1}, we further obtain the $L^1_{t}(\dot{B}^{\frac{d}{2}}_{2,1})\cap \widetilde{L}^2_{t}(\dot{B}^{\frac{d}{2}-1}_{2,1})$-estimate of the relative velocity $u-w$, which is stronger than the $L^1_{t}(\dot{B}^{\frac{d}{2}+1}_{2,1})\cap \widetilde{L}^2_{t}(\dot{B}^{\frac{d}{2}}_{2,1})$-estimate of $(a,u,b,w)$ for low frequencies. Combing the above estimates of $(a,u,b,w)$ and $u-w$ together, we enclose the a-priori estimates of the solution $(a,u,b,w)$ to the Cauchy problem \eqref{m1n} (refer to Lemmas \ref{prop41} and \ref{prop42}).

%which is the key to handle different critical initial regularities in high frequencies (i.e., $\dot{B}^{\frac{d}{2}}_{2,1}\times\dot{B}^{\frac{d}{2}-1}_{2,1}$ for the compressible Navier-Stokes equations $\eqref{m1}_{1}$-$\eqref{m1}_{2}$ and $\dot{B}^{\frac{d}{2}+1}_{2,1}$ for the compressible Euler equations $\eqref{m1}_{3}$-$\eqref{m1}_{4}$).

If $(a_{0},u_{0},b_{0},w_{0})^{\ell}$ is further bounded in $\dot{B}^{\sigma_{0}}_{2,\infty}$ for $\sigma_{0}\in [-\frac{d}{2},\frac{d}{2}-1)$, motivated by the interesting works \cite{guo1,xin1}, we establish different time-weighted energy estimates to derive the optimal time-decay rates of the solution $(a,u,b,w)$ in $\eqref{decay1}_{1}$-$\eqref{decay1}_{2}$ (refer to Lemma \ref{prop51}-\ref{lemma52}), and furthermore take advantage of the damped equation \eqref{relativedamp1} to get the faster time-decay rates in $\eqref{decay1}_{3}$ and \eqref{decay11}. When $\|(a_{0},u_{0},b_{0},w_{0})^{\ell}\|_{\dot{B}^{\sigma_{0}}_{2,\infty}}$ is sufficiently small, we also show more time-decay rates of $(a,u,b,w)$ and $u-w$ in \eqref{r4} (refer to Lemmas \ref{lemma61}-\ref{lemma63}) in the spirit of \cite{danchin5,xu3}. It should be emphasized that the rate $(1+t)^{-\frac{1}{2}}$ of the relative velocity $u-w$ in $\dot{B}^{\sigma_{0}}_{2,\infty}$ is the key point to derive the rate $(1+t)^{-\frac{1}{2}(\frac{d}{2}+1-\sigma_{0})}$ of the nonlinear term $\|h(a,b)(u-w)\|_{\dot{B}^{\sigma_{0}}_{2,\infty}}$ in \eqref{factddd1} and enclose the energy estimates.

\vspace{2ex}

The rest of the paper is organized as follows. In Section \ref{sectionglobal}, we prove Theorems \ref{theorem11}-\ref{theorem12} and establish the a-priori estimates of the solution to the Cauchy problem \eqref{m1n}. In Section \ref{sectionlarge}, we carry out the proofs of Theorems \ref{theorem13}-\ref{theorem14} on the optimal time-decay rates of the global solution. In Section \ref{sectionnotation}, we present some notations of Besov spaces and recall related analysis tools used in this paper.

% In Section \ref{sectionresult}, we state the main results of this paper.

\section{Local existence and uniqueness}

To show the local existence, we first construct the local Friedrichs approximation (cf. \cite{danchin1,hcc1}), establish the uniform a-priori estimates of the approximate sequence, and then show their convergence to the expected strong solution of the original Cauchy problem (\ref{m1n}).

\textbf{Step 1: Construction of approximate sequence}

Let $u_{L}$ be the unique global solution for the linear problem 
\begin{equation}\label{uL}
\left\{
\begin{aligned}
&\partial_{t}u_{L}-\mathcal{A}u_{L}=0,\quad\quad x\in\mathbb{R}^{d},\quad t>0,\\
& u_{L}(x,0)=u_{0}(x),\quad\quad x\in\mathbb{R}^{d}.
\end{aligned}
\right.
\end{equation}
Define
\begin{equation}\label{tildea}
\begin{aligned}
\tilde{a}:=\frac{1}{1+a}-1,\quad \tilde{u}:=u-u_{L},\quad \Pi(\tilde{a}):=\int_{0}^{\tilde{a}} P'(\frac{1}{1+s})\frac{1}{1+s}ds,
\end{aligned}
\end{equation}
so that the system (\ref{m1n}) can be re-written as
\begin{equation}\label{m11}
\left\{
\begin{aligned}
&\partial_{t}\tilde{a}+(\tilde{u}+u_{L})\cdot \nabla \tilde{a}=(\tilde{a}+1)\div (\tilde{u}+u_{L}),\\
&\partial_{t}\tilde{u}+(\tilde{u}+u_{L})\cdot\nabla\tilde{u}+\tilde{u}\cdot\nabla u_{L}-(1+\tilde{a})\mathcal{A}\tilde{u}\\
&\quad~=\tilde{a}\mathcal{A}u_{L}-u_{L}\cdot\nabla u_{L}+\nabla \Pi(\tilde{a})+ e^{b}(\tilde{a}+1)(w-\tilde{u}-u_{L}),\\
&\partial_{t}b+w\cdot \nabla b+\div w=0,\\
&\partial_{t}w+w\cdot \nabla w+\nabla b+ w= (\tilde{u}+u_{L}),\\
&(\tilde{a},\tilde{u},b,w)(x,0)=(\frac{1}{1+a_{0}}-1,0,b_{0},w_{0})(x).
\end{aligned}
\right.
\end{equation}
Let $L^2_{n}$ be the set of $L^2$ functions spectrally supported in the annulus $\mathcal{C}_{n}:=\{\xi\in\mathbb{R}^{d}~|~\frac{1}{n}\leq|\xi|\leq n\}$ endowed with the standard $L^2$ topology, and $\dot{\mathbb{E}}_{n}$ be the Friedrichs projectors defined by
\begin{equation}\nonumber
\begin{aligned}
&\dot{\mathbb{E}}_{n}f:=\mathcal{F}^{-1}(\mathbf{1}_{\mathcal{C}_{n}}\mathcal{F}f),\quad\forall f\in L^2(\mathbb{R}^{d}),
\end{aligned}
\end{equation}
where $\mathbf{1}_{\mathcal{C}_{n}}$ is the characteristic function on the annulus $\mathcal{C}_{n}$. Then the approximate sequence $(\widetilde{a}^{n},\widetilde{u}^{n},b^{n},w^{n})$ is defined by solving the following equations:
\begin{equation}\label{appm1}
\left\{
\begin{aligned}
&\partial_{t}\tilde{a}^{n}+ \dot{\mathbb{E}}_{n}\big{(}(\tilde{u}^{n}+u_{L}^{n})\cdot\nabla \tilde{a}^{n}\big{)}=-\dot{\mathbb{E}}_{n}\big{(}(1+\tilde{a}^{n})\div (\tilde{u}^{n}+u_{L}^{n})\big{)},\\
&\partial_{t}\tilde{u}^{n}+\dot{\mathbb{E}}_{n}\big{(}(\tilde{u}^{n}+u_{L}^{n})\cdot\nabla \tilde{u}^{n}+\tilde{u}^{n}\cdot \nabla u_{L}^{n}-(1+\tilde{a}^{n})\mathcal{A}\tilde{u}^{n}\big{)}\\
&\quad=\dot{\mathbb{E}}_{n}\big{(}-\tilde{a}^{n}\mathcal{A}u_{L}^{n}-u_{L}^{n}\cdot\nabla u_{L}^{n}+\nabla \Pi(\tilde{a}^{n}) + e^{b^{n}}(1+\tilde{a}^{n})(w^{n}-\tilde{u}^{n}-u_{L}^{n})  \big{)},\\
&\partial_{t}b^{n}+\dot{\mathbb{E}}_{n}\big{(}w^{n}\cdot\nabla b^{n}\big{)}+\div w^{n}=0,\\
&\partial_{t}w^{n}+\dot{\mathbb{E}}_{n}\big{(}w^{n}\cdot\nabla w^{n}\big{)}+\nabla b^{n}+ w^{n}= (\tilde{u}^{n}+u_{L}^{n}),\\
&(\tilde{a}^{n},\tilde{u}^{n},b^{n},w^{n})(x,0)=(\tilde{a}_{0}^{n},\tilde{u}_{0}^{n},b_{0}^{n},w_{0}^{n})(x):=(\dot{\mathbb{E}}_{n}(\frac{1}{1+a_{0}}-1),0,\dot{\mathbb{E}}_{n}b_{0},\dot{\mathbb{E}}_{n}w_{0})(x),
\end{aligned}
\right.
\end{equation}
with $u_{L}^{n}:=\dot{\mathbb{E}}_{n}u_{L}$. One can show under the assumptions of Theorem \ref{theorem11} that $(\tilde{a}_{0}^{n},b^{n}_{0},w^{n}_{0})$ converges to $(\frac{1}{1+a_{0}}-1,b_{0},w_{0})$ strongly in $\dot{B}^{\frac{d}{2}}_{2,1}\times(\dot{B}^{\frac{d}{2}+1}_{2,1}\times\dot{B}^{\frac{d}{2}-1}_{2,1})\times(\dot{B}^{\frac{d}{2}+1}_{2,1}\times\dot{B}^{\frac{d}{2}-1}_{2,1})$ as $n\rightarrow\infty$, and thus one can find a sufficiently large constant $n_{0}>0$ such that 
\begin{equation}\label{rho0lower}
\left\{
\begin{aligned}
&\quad\frac{3}{4\rho_{0+}}\leq 1+\tilde{a}^{n}_{0}\leq \frac{4}{3\rho_{0-}},\\
&\sum_{j\geq m-1}2^{\frac{d}{2}j}\|\tilde{a}^{n}_{0}\|_{L^2}\leq \frac{3}{2}\sum_{j\geq m}2^{\frac{d}{2}j}\|\frac{1}{1+a^{0}}-1\|_{L^2},\quad  n\geq n_{0},\quad  m\in\mathbb{Z},
\end{aligned}
\right.
\end{equation}
with
\begin{equation}\label{rho0}
\begin{aligned}
\rho_{0+}:=\sup_{x\in\mathbb{R}^{d}}(1+a_{0})(x)>0,\quad\quad \rho_{0-}:=\inf_{x\in\mathbb{R}^{d}}(1+a_{0})(x)>0.
\end{aligned}
\end{equation}
Since all the Sobolev norms are equivalent in (\ref{appm1}) due to the Bernstein inequality, it is easy to verify that (\ref{appm1}) is an ordinary differential system in $L^2_{n}$ and locally Lipschitz with respect to the variable $(a^{n},u^{n},b^{n},w^{n})$ for any $n\geq n_{0}$. Hence by the Cauchy-Lipschitz theorem (cf. \cite{bahouri1}, Page 124), there is a maximal time $T^{n}_{*}>0$ such that unique solution $(\widetilde{a}^{n},\widetilde{u}^{n},b^{n},w^{n})\in \mathcal{C}([0,T^{n}_{*});L^2_{n}\times L^2_{n}\times L^2_{n}\times L^2_{n})$ to the problem $(\ref{appm1})$ exists, and $1+\tilde{a}^{n}$ is strictly bounded away from zero for any $n\geq n_{0}$.

\textbf{Step 2: Uniform estimates}

For some constants $T\in(0,T_{*})$, $\eta>0$ and $\beta>0$ to be determined, let $m=m(\beta,a_{0})>0$ be the constant
\begin{equation}\label{m}
\begin{aligned}
& m:=\inf\Big{\{}q\in \mathbb{Z}~\big{|}~\sum_{j\geq q-1}2^{\frac{d}{2}j}\|\dot{\Delta}_{j}(\frac{1}{1+a_{0}}-1)\|_{L^2}\leq \frac{\beta}{3}\Big{\}},
\end{aligned}
\end{equation}
and $I^{n}$ be the time set
\begin{equation}\label{In}
\begin{aligned}
&I^{n}:=\Big{\{}t\in [0,T]~\big{|}~ \|\tilde{u}^{n}\|_{\widetilde{L}^{\infty}_{t}(\dot{B}^{\frac{d}{2}-1}_{2,1})}+\|\tilde{u}^{n}\|_{L^1_{t}(\dot{B}^{\frac{d}{2}+1}_{2,1})}\leq \eta,\\
 &\quad\quad\quad\quad\quad\quad\quad\quad~\|\tilde{a}^{n}\|_{\widetilde{L}^{\infty}_{t}(\dot{B}^{\frac{d}{2}}_{2,1})}\leq 2\|\frac{1}{1+a_{0}}-1\|_{\dot{B}^{\frac{d}{2}}_{2,1}}+1,\\
&\quad\quad\quad\quad\quad\quad\quad\quad~    \underset{j\geq m-1}\sum 2^{\frac{d}{2}j}\|\dot{\Delta}_{j}\tilde{a}^{n}\|_{L^{\infty}_{t}(L^2)}\leq \beta,\\
&\quad\quad\quad\quad\quad\quad\quad\quad~\frac{1}{2\rho_{0+}}\leq1+\tilde{a}^{n}\leq\frac{2}{\rho_{0-}},\\
&\quad\quad\quad\quad\quad\quad\quad\quad~\|(b^{n},w^{n})\|_{\widetilde{L}^{\infty}_{t}(\dot{B}^{\frac{d}{2}-1}_{2,1}\cap\dot{B}^{\frac{d}{2}+1}_{2,1})}\leq2\|(b_{0},w_{0})\|_{\dot{B}^{\frac{d}{2}-1}_{2,1}\cap\dot{B}^{\frac{d}{2}+1}_{2,1}}\Big{\}},
\end{aligned}
\end{equation}
where $\rho_{0+}$ and $\rho_{0-}$ are given by (\ref{rho0}). It follows from the time continuity of $(\tilde{a}^{n},\tilde{u}^{n},b^{n},w^{n})$ that $I^{n}$ is a nonempty closed subset of $[0,T]$ for every $n\geq n_{0}$. For any $t\in I^{n}$ and $n\geq n_{0}$, we claim that
\begin{equation}\label{claim1}
\left\{
\begin{aligned}
&\|\tilde{u}^{n}\|_{\widetilde{L}^{\infty}_{t}(\dot{B}^{\frac{d}{2}-1}_{2,1})}+\|\tilde{u}^{n}\|_{L^1_{t}(\dot{B}^{\frac{d}{2}+1}_{2,1})}< \eta,\\
&\|\tilde{a}^{n}\|_{\widetilde{L}^{\infty}_{t}(\dot{B}^{\frac{d}{2}}_{2,1})}< 2\|\frac{1}{1+a_{0}}-1\|_{\dot{B}^{\frac{d}{2}}_{2,1}}+1,\\
&\underset{j\geq m-1}\sum 2^{\frac{d}{2}j}\|\dot{\Delta}_{j}\tilde{a}^{n}\|_{L^{\infty}_{t}(L^2)}< \beta,\\
&\frac{1}{2\rho_{0+}}<1+\tilde{a}^{n}<\frac{2}{\rho_{0-}},\\
&\|(b^{n},w^{n})\|_{\widetilde{L}^{\infty}_{t}(\dot{B}^{\frac{d}{2}-1}_{2,1}\cap\dot{B}^{\frac{d}{2}+1}_{2,1})}<2\|(b_{0},w_{0})\|_{\dot{B}^{\frac{d}{2}-1}_{2,1}\cap\dot{B}^{\frac{d}{2}+1}_{2,1}}.
\end{aligned}
\right.
\end{equation}
Since all the estimates in (\ref{claim1}) are strictly at the time $t\in I^{n}$, again using the time continuity of the solution $(\tilde{a}^{n},\tilde{u}^{n},b^{n},w^{n})$, one can show that there exists a ball $B(t,\varepsilon)$ for a suitably small constant $\varepsilon>0$ satisfying $[0,T]\cap B(t,\varepsilon)\subset I^{n}$, which implies that $I^{n}$ is also an open subset of $[0,T]$. Hence, we have $I^{n}=[0,T]$, and the approximate sequence $(\tilde{a}^{n},\tilde{u}^{n},b^{n},w^{n})$ satisfies the estimates in (\ref{In}) uniformly with respect to $n\geq n_{0}$.

We turn to choose suitable small $T>0$, $\eta>0$ and $\alpha>0$ such that (\ref{claim1}) is valid for any $t\in I^{n}$ and $n\geq n_{0}$. Since $\dot{\mathbb{E}}_{n}$ is a $L^2$ orthogonal projector, there is no effect
on usual energy estimates in Lemmas \ref{lemma31}-\ref{lemma35} for $(\tilde{a}^{n},\tilde{u}^{n},b^{n},w^{n})$. Applying Lemma \ref{lemma33} to (\ref{uL}) and the dominated convergence theorem, for any constant $\eta>0$, there exists a sufficiently small time $T_{1}=T_{1}(\eta,u_{0})>0$ such that
\begin{equation}\label{uLeta}
\begin{aligned}
\|u_{L}^{n}\|_{L^1_{t}(\dot{B}^{\frac{d}{2}+1}_{2,1})}<\eta^2,\quad \|u_{L}^{n}\|_{\widetilde{L}^{\infty}_{t}(\dot{B}^{\frac{d}{2}-1}_{2,1})}\leq C\|u_{0}\|_{\dot{B}^{\frac{d}{2}-1}_{2,1}},\quad t\in[0,T_{1}].
\end{aligned}
\end{equation}
Due to the embedding $\dot{B}^{\frac{d}{2}}_{2,1}\hookrightarrow L^{\infty}$ and
\begin{equation}\nonumber
\left\{
\begin{aligned}
&\|\sum_{j\geq m}\dot{\Delta}_{j}\tilde{a}^{n}\|_{\widetilde{L}^{\infty}_{t}(\dot{B}^{\frac{d}{2}}_{2,1})}\leq C\underset{j\geq m-1}\sum 2^{\frac{d}{2}j}\|\dot{\Delta}_{j}\tilde{a}^{n}\|_{L^{\infty}_{t}(L^2)}\leq C\beta,\\
&1+\sum_{j\leq m-1}\dot{\Delta}_{j}\tilde{a}^{n}\geq\frac{1}{2\rho_{0+}}-\|\sum_{j\geq m}\dot{\Delta}_{j}\tilde{a}^{n}\|_{\widetilde{L}^{\infty}_{t}(\dot{B}^{\frac{d}{2}}_{2,1})}\geq \frac{1}{2\rho_{0+}}-C\beta,
\end{aligned}
\right.
\end{equation}
derived from $(\ref{In})_{3}$-$(\ref{In})_{4}$, we get
\begin{equation}\nonumber
\left\{
\begin{aligned}
&1+\sum_{j\leq m-1}\dot{\Delta}_{j}\tilde{a}^{n}\geq \frac{1}{4\rho_{0+}},\\
&\|\sum_{j\geq m}\dot{\Delta}_{j}\tilde{a}^{n}\|_{\widetilde{L}^{\infty}_{t}(\dot{B}^{\frac{d}{2}}_{2,1})}\leq \frac{c_{2} }{2\rho_{0+}},
\end{aligned}
\right.
\end{equation}
provided
\begin{equation}
\begin{aligned}
0<\beta< \Big{\{}\frac{1}{4\rho_{0+}C}, \frac{c_{2} }{2\rho_{0+}C }\Big{\}}, \label{T1}
\end{aligned}
\end{equation}
with $c_{2}>0$ given by Lemma \ref{lemma34}. Thus, all the conditions in Lemma \ref{lemma34} hold, and we have for any $t\in I_{n}\cap[0,T_{1}]$ that
\begin{equation}\label{tildeuestimate}
\begin{aligned}
&\|\tilde{u}^{n}\|_{\widetilde{L}^{\infty}_{t}(\dot{B}^{\frac{d}{2}-1}_{2,1})}+\|\tilde{u}^{n}\|_{L^1_{t}(\dot{B}^{\frac{d}{2}+1}_{2,1})}\\
&\quad\leq C_{a_{0}}e^{C_{a_{0}}\big{(}\|\tilde{a}^{n}\|_{L^2_{t}(\dot{B}^{\frac{d}{2}}_{2,1})}^2+\|(\tilde{u}^{n},u^{n}_{L})\|_{L^1_{t}(\dot{B}^{\frac{d}{2}+1}_{2,1})}\big{)}}\Big{(}\|\tilde{a}^{n}\mathcal{A}u_{L}^{n}\|_{\widetilde{L}^1_{t}(\dot{B}^{\frac{d}{2}-1}_{2,1})}+\|u_{L}^{n}\cdot\nabla u_{L}^{n}\|_{\widetilde{L}^1_{t}(\dot{B}^{\frac{d}{2}-1}_{2,1})}\\
&\quad\quad +\|\nabla \Pi(\tilde{a}^{n})\|_{\widetilde{L}^1_{t}(\dot{B}^{\frac{d}{2}-1}_{2,1})}+\| e^{b^{n}}(1+\tilde{a}^{n})(w^{n}-\tilde{u}^{n}-u_{L}^{n})  \|_{\widetilde{L}^1_{t}(\dot{B}^{\frac{d}{2}-1}_{2,1})}     \Big{)},\\
\end{aligned}
\end{equation}
where $C_{a_{0}}>0$ denotes a suitably large constant depending on $d$, $m$ and $a_{0}$. The terms on the right-hand side of (\ref{tildeuestimate}) can be controlled as follows. By $(\ref{In})_{1}$-$(\ref{In})_{2}$ and (\ref{uLeta}), we have
\begin{equation}\label{tildeuestimate1}
\begin{aligned}
&C_{a_{0}}\big{(}\|\tilde{a}^{n}\|_{L^2_{t}(\dot{B}^{\frac{d}{2}}_{2,1})}^2+\|(\tilde{u}^{n},u^{n}_{L})\|_{L^1_{t}(\dot{B}^{\frac{d}{2}+1}_{2,1})}\big{)}\leq  C_{a_{0}}(t+\eta+\eta^2).
\end{aligned}
\end{equation}
Since the product of functions maps $\dot{B}^{\frac{d}{2}-1}_{2,1}\times\dot{B}^{\frac{d}{2}}_{2,1}$ to $\dot{B}^{\frac{d}{2}-1}_{2,1}$ in (\ref{uv2}), we derive by $(\ref{In})_{2}$ and (\ref{uLeta}) that
\begin{equation}\label{tildeuestimate2}
\begin{aligned}
&\|\tilde{a}^{n}\mathcal{A}u_{L}\|_{\widetilde{L}^1_{t}(\dot{B}^{\frac{d}{2}-1}_{2,1})}\leq C\|\tilde{a}^{n}\|_{\widetilde{L}^{\infty}_{t}(\dot{B}^{\frac{d}{2}}_{2,1})}\|\mathcal{A}u_{L}\|_{\widetilde{L}^1_{t}(\dot{B}^{\frac{d}{2}-1}_{2,1})}\leq C_{a_{0}}\eta^2,
\end{aligned}
\end{equation}
and
\begin{equation}\label{tildeuestimate3}
\begin{aligned}
&\|u^{n}_{L}\cdot \nabla u^{n}_{L}\|_{\widetilde{L}^1_{t}(\dot{B}^{\frac{d}{2}-1}_{2,1})}\leq C\|u^{n}_{L}\|_{\widetilde{L}^{\infty}_{t}(\dot{B}^{\frac{d}{2}-1}_{2,1})}\|\nabla u^{n}_{L}\|_{L^1_{t}(\dot{B}^{\frac{d}{2}}_{2,1})}\leq C\eta^2.
\end{aligned}
\end{equation}
Applying the estimate (\ref{F1}) for the composition function $\Pi(\tilde{a}^{n})$ and $(\ref{In})_{2}$, we also have
\begin{equation}\label{tildeuestimate4}
\begin{aligned}
&\|\nabla \Pi(\tilde{a}^{n})\|_{\widetilde{L}^1_{t}(\dot{B}^{\frac{d}{2}-1}_{2,1})}\leq Ct\|\Pi(\tilde{a}^{n})\|_{\widetilde{L}^{\infty}_{t}(\dot{B}^{\frac{d}{2}}_{2,1})}\leq C_{a_{0}}t.
\end{aligned}
\end{equation}
Similarly, it holds that
\begin{equation}\label{tildeuestimate5}
\begin{aligned}
&\| e^{b^{n}}(1+\tilde{a}^{n})(w^{n}-\tilde{u}^{n}-u_{L}^{n})\|_{\widetilde{L}^1_{t}(\dot{B}^{\frac{d}{2}-1}_{2,1})}\\
&\quad\leq C_{a_{0},b_{0}}\big{(}\|b^{n}\|_{\widetilde{L}^{\infty}_{t}(\dot{B}^{\frac{d}{2}}_{2,1})}+1\big{)}\big{(}\|a^{n}\|_{\widetilde{L}^{\infty}_{t}(\dot{B}^{\frac{d}{2}}_{2,1})}+1\big{)}\|(w^{n},\widetilde{u}^{n},u^{n}_{L})\|_{\widetilde{L}^{\infty}_{t}(\dot{B}^{\frac{d}{2}-1}_{2,1})}t\\
&\quad\leq C_{a_{0},b_{0}}\big{(}\eta+\|u_{0}\|_{\dot{B}^{\frac{d}{2}-1}_{2,1}}+\|(b_{0},w_{0})\|_{\dot{B}^{\frac{d}{2}-1}_{2,1}\cap\dot{B}^{\frac{d}{2}+1}_{2,1}}\big{)}t,
\end{aligned}
\end{equation}
where $C_{a_{0},b_{0}}>0$ denotes a sufficiently large constant only depending on $d,, a_{0}$ and $b_{0}$, and we have used (\ref{In}), (\ref{uLeta}),  (\ref{uv1})-(\ref{uv2}), (\ref{F1}) and the embedding $\dot{B}^{\frac{d}{2}-1}_{2,1}\cap\dot{B}^{\frac{d}{2}+1}_{2,1}\hookrightarrow \dot{B}^{\frac{d}{2}}_{2,1}$. Combining (\ref{tildeuestimate})-(\ref{tildeuestimate5}) together, we obtain for any $t\in I^{n}\cap[0,T_{1}]$ that
\begin{equation}\nonumber
\begin{aligned}
&\|\tilde{u}^{n}\|_{\widetilde{L}^{\infty}_{t}(\dot{B}^{\frac{d}{2}-1}_{2,1})}+\|\tilde{u}^{n}\|_{L^1_{t}(\dot{B}^{\frac{d}{2}+1}_{2,1})}\\
&\quad\leq C_{a_{0},b_{0}}e^{C_{a_{0}}(t+\eta+\eta^2)}\Big{(}\eta^2+\eta^4+\big{(}1+\eta+\|u_{0}\|_{\dot{B}^{\frac{d}{2}-1}_{2,1}}+\|(b_{0},w_{0})\|_{\dot{B}^{\frac{d}{2}-1}_{2,1}\cap\dot{B}^{\frac{d}{2}+1}_{2,1}}\big{)}t\Big{)}.
\end{aligned}
\end{equation}
Thus, the condition $(\ref{claim1})_{1}$ holds provided
\begin{equation}\label{T}
\begin{aligned}
T=\min\{T_{1},\eta^2\},
\end{aligned}
\end{equation}
 and
\begin{equation}\label{T2}
\begin{aligned}
&C_{a_{0},b_{0}}e^{C_{a_{0}}(t+\eta+\eta^2)}\Big{(}\eta^2+\eta^{4}+\big{(}1+\eta+\|u_{0}\|_{\dot{B}^{\frac{d}{2}-1}_{2,1}}+\|(b_{0},w_{0})\|_{\dot{B}^{\frac{d}{2}-1}_{2,1}\cap\dot{B}^{\frac{d}{2}+1}_{2,1}}\big{)}\eta^2\Big{)}<\eta.
\end{aligned}
\end{equation}

Next, we show the conditions $(\ref{claim1})_{2}$-$(\ref{claim1})_{4}$ for $\tilde{a}^{n}$. Using (\ref{lemma321}) to the equation $(\ref{m11})_{1}$, one deduces for any $t\in I^{n}\cap [0,T_{1}]$ that
\begin{equation}\nonumber
\begin{aligned}
&\|\tilde{a}^{n}\|_{\widetilde{L}^{\infty}_{t}(\dot{B}^{\frac{d}{2}}_{2,1})}\leq e^{C\|(\tilde{u}^{n},u^{n}_{L})\|_{L^1_{t}(\dot{B}^{\frac{d}{2}+1}_{2,1})}}\|\tilde{a}^{n}_{0}\|_{\dot{B}^{\frac{d}{2}}_{2,1}}+e^{C\|(\tilde{u}^{n},u^{n}_{L})\|_{L^1_{t}(\dot{B}^{\frac{d}{2}+1}_{2,1})}}-1\\
&\quad\quad\quad\quad\quad~~\leq \frac{3}{2}e^{C(\eta+\eta^2)}\|\frac{1}{1+a_{0}}-1\|_{\dot{B}^{\frac{d}{2}}_{2,1}}+e^{C(\eta+\eta^2)}-1,
\end{aligned}
\end{equation}
where in the last inequality we have used $(\ref{rho0lower})_{2}$, $(\ref{In})_{1}$-$(\ref{In})_{2}$ and (\ref{uLeta}). To derive the condition $(\ref{claim1})_{2}$, one may choose $\eta$ so small that
\begin{equation}\label{T3}
\begin{aligned}
&\frac{3}{2}e^{C(\eta+\eta^2)}<2.
\end{aligned}
\end{equation}
Similarly, it follows from $(\ref{rho0lower})_{2}$, (\ref{m}), $(\ref{In})_{1}$, $(\ref{uLeta})$ and (\ref{lemma322}) for any $t\in I^{n}\cap [0,T]$ and $n\geq n_{0}$ that
\begin{equation}\nonumber
\begin{aligned}
&\sum_{j\geq m-1}2^{\frac{d}{2}j}\|\tilde{a}^{n}\|_{L^{\infty}_{t}(L^2)}\leq \sum_{j\geq m-1}2^{\frac{d}{2}j}\|\tilde{a}^{n}_{0}\|_{L^2}+(1+\|\tilde{a}^{n}_{0}\|_{\dot{B}^{\frac{d}{2}}_{2,1}})\Big{(}e^{C\|(\tilde{u}^{n},u_{L}^{n})\|_{L^1_{t}(\dot{B}^{\frac{d}{2}+1}_{2,1})}}-1\Big{)}\\
&\quad\quad\quad\quad\quad\quad~~\quad\quad\quad\leq \frac{\alpha}{2}+C_{a_{0}}(e^{C(\eta+\eta^2)}-1).
\end{aligned}
\end{equation}
Thus one concludes $(\ref{claim1})_{2}$ if
\begin{equation}\label{T4}
\begin{aligned}
&C_{a_{0}}(e^{C(\eta+\eta^2)}-1)<\frac{\beta}{2}.
\end{aligned}
\end{equation}
To prove $(\ref{claim1})_{4}$, we re-write $1+\tilde{a}^{n}$ by
\begin{equation}\label{1plusa}
\begin{aligned}
1+\tilde{a}^{n}=1+\tilde{a}_{0}^{n}+\sum_{j\geq m}(\tilde{a}^{n}-\tilde{a}_{0}^{n})+\sum_{j\leq m-1}(\tilde{a}^{n}-\tilde{a}^{n}_{0}).
\end{aligned}
\end{equation}
By $(\ref{In})_{3}$ and the embedding $\dot{B}^{\frac{d}{2}}_{2,1}\hookrightarrow L^{\infty}$, we have
\begin{equation}\label{1plusa1}
\begin{aligned}
&\|\sum_{j\geq m}(\tilde{a}^{n}-\tilde{a}^{n}_{0})\|_{L^{\infty}_{t}(L^{\infty})}\leq C\sum_{j\geq m-1} 2^{\frac{d}{2}j}\|\tilde{a}^{n}-\tilde{a}^{n}_{0}\|_{L^{\infty}_{t}(L^2)}\leq C \beta.
\end{aligned}
\end{equation}
For the third term on the right-hand side of (\ref{1plusa}), we deduce by (\ref{lemma323}), $(\ref{In})_{1}$ and (\ref{uLeta}) for any $t\in I^{n}\cap [0,T_{1}]$ that
\begin{equation}\label{1plusa2}
\begin{aligned}
&\|\sum_{j\leq m-1}(\tilde{a}^{n}-\tilde{a}^{n}_{0})\|_{L^{\infty}_{t}(L^{\infty})}\leq C\sum_{j\leq m} 2^{\frac{d}{2}j}\|\tilde{a}^{n}-\tilde{a}^{n}_{0}\|_{L^{\infty}_{t}(L^2)}\leq C_{a_{0}}\Big{(}e^{C(\eta+\eta^2)}-1+(\eta+\eta^2)t^{\frac{1}{2}}\Big{)}.
\end{aligned}
\end{equation}
Thence, for any $t\in I^{n}\cap [0,T_{1}]$ and $n\geq n_{0}$, we substitute (\ref{1plusa1})-(\ref{1plusa2}) into (\ref{1plusa}) and use $(\ref{rho0lower})_{1}$ to have
\begin{equation}\nonumber
\left\{
\begin{aligned}
&1+\tilde{a}^{n}\leq \frac{4}{3\rho_{0-}}+C\alpha+C_{a_{0}}\Big{(}e^{C(\eta+\eta^2)}-1+(\eta+\eta^2)t^{\frac{1}{2}}\Big{)},\\
&1+\tilde{a}^{n}\geq \frac{3}{4\rho_{0+}}-C\alpha-C_{a_{0}}\Big{(}e^{C(\eta+\eta^2)}-1+(\eta+\eta^2)t^{\frac{1}{2}}\Big{)}.
\end{aligned}
\right.
\end{equation}
If the time $T$ is given by (\ref{T}), and the small constants $\beta$, $\eta$ satisfy
\begin{equation}\label{T5}
\left\{
\begin{aligned}
&C\beta<\min\{\frac{1}{3\rho_{0-}},\frac{1}{8\rho_{0+}}\},\\
&C_{a_{0}}\Big{(}e^{C(\eta+\eta^2)}-1+(\eta+\eta^2)\eta\Big{)}<\min\{\frac{1}{3\rho_{0-}},\frac{1}{8\rho_{0+}}\},
\end{aligned}
\right.
\end{equation}
then $(\ref{claim1})_{4}$ is proved.

Finally, we show $(\ref{claim1})_{5}$ for $(b^{n},w^{n})$. Applying the operator $\dot{\Delta}_{j}$ for $j\in\mathbb{Z}$ to $(\ref{appm1})_{3}$ and taking the $L^2$ inner product of the resulting equation with $\dot{\Delta}_{j}a^{n}$, we have
\begin{equation}\label{2111}
\begin{aligned}
&\frac{1}{2}\frac{d}{dt}\|\dot{\Delta}_{j}b^{n}\|_{L^2}^2+\big{(}\div\dot{\Delta}_{j}w^{n}~|~\dot{\Delta}_{j} b^{n}\big{)}_{L^2}=-\big{(}\dot{\Delta}_{j}(w^{n}\cdot\nabla b^{n})~|~\dot{\Delta}_{j}b^{n}\big{)}_{L^2}.
\end{aligned}
\end{equation}
Similarly, one can obtain
\begin{equation}\label{2112}
\begin{aligned}
&\frac{1}{2}\frac{d}{dt}\|\dot{\Delta}_{j}w^{n}\|_{L^2}^2+\kappa\|\dot{\Delta}_{j}w^{n}\|_{L^2}^2+\big{(}\nabla\dot{\Delta}_{j}b^{n}~|~\dot{\Delta}_{j}w^{n}\big{)}_{L^2}\\
&~~=-\big{(}\dot{\Delta}_{j}(w^{n}\cdot\nabla w^{n})~|~\dot{\Delta}_{j}w^{n}\big{)}_{L^2}+\kappa\big{(}\dot{\Delta}_{j}(\tilde{u}^{n}+u_{L}^{n})~|~\dot{\Delta}_{j}w^{n}\big{)}_{L^2}.
\end{aligned}
\end{equation}
Adding (\ref{2111})-(\ref{2112}) together, we get
\begin{equation}\nonumber
\begin{aligned}
&\frac{1}{2}\frac{d}{dt}\|\dot{\Delta}_{j}(b^{n},w^{n})\|_{L^2}^2+\kappa\|\dot{\Delta}_{j}w^{n}\|_{L^2}^2\\
&~~\leq \big{(}\kappa\|\dot{\Delta}_{j}(\tilde{u}^{n},u^{n}_{L})\|_{L^2}+\frac{1}{2}\|\div w^{n}\|_{L^{\infty}}\|\dot{\Delta}_{j}(b^{n},w^{n})\|_{L^2}+\|[w^{n}\cdot\nabla ,\dot{\Delta}_{j}](b^{n},w^{n})\|_{L^2}\big{)}\|\dot{\Delta}_{j}(b^{n},w^{n})\|_{L^2},
\end{aligned}
\end{equation}
from which it follows that
\begin{equation}\label{2113}
\begin{aligned}
&\|\dot{\Delta}_{j}(b^{n},w^{n})\|_{L^2}\leq \|\dot{\Delta}_{j}(b_{0},w_{0})\|_{L^2}+\kappa\|\dot{\Delta}_{j}(\tilde{u}^{n},u^{n}_{L})\|_{L^2}+\frac{1}{2}\|\div w^{n}\|_{L^{\infty}}\|\dot{\Delta}_{j}(b^{n},w^{n})\|_{L^2}\\
&\quad\quad\quad\quad\quad\quad\quad~~+\|[w^{n}\cdot\nabla ,\dot{\Delta}_{j}](b^{n},w^{n})\|_{L^2}.
\end{aligned}
\end{equation}
Multiplying (\ref{2113}) by $2^{j(\frac{d}{2}-1)}$ and summing up over $j\in\mathbb{Z}$, we have for any $t\in I^{n}$ that
\begin{equation}\label{2114}
\begin{aligned}
&\|(b^{n},w^{n})\|_{\widetilde{L}^{\infty}_{t}(\dot{B}^{\frac{d}{2}-1}_{2,1})}\\
&\quad\leq \|(b_{0},w_{0})\|_{\dot{B}^{\frac{d}{2}-1}_{2,1}}+\Big{(} \|(\tilde{u}^{n},u^{n}_{L})\|_{\widetilde{L}^{\infty}_{t}(\dot{B}^{\frac{d}{2}-1}_{2,1})}+\frac{1}{2}\|\div w^{n}\|_{L^{\infty}_{t}(L^{\infty})}\|(b^{n},w^{n})\|_{\widetilde{L}^{\infty}_{t}(\dot{B}^{\frac{d}{2}-1}_{2,1})}\\
&\quad\quad+\sum_{j\in\mathbb{Z}}2^{j(\frac{d}{2}-1)}\|[w^{n}\cdot\nabla ,\dot{\Delta}_{j}](b^{n},w^{n})\|_{L^{\infty}_{t}(L^2)}\Big{)}t.
\end{aligned}
\end{equation}
Due to $(\ref{In})_{1}$ and  (\ref{uLeta}), it holds that
\begin{equation}\label{21141}
\begin{aligned}
&\|(\tilde{u}^{n},u^{n}_{L})\|_{\widetilde{L}^{\infty}_{t}(\dot{B}^{\frac{d}{2}-1}_{2,1})}\leq C(\|u_{0}\|_{\dot{B}^{\frac{d}{2}-1}_{2,1}}+\eta).
\end{aligned}
\end{equation}
By $(\ref{In})_{5}$ and the embedding $\dot{B}^{\frac{d}{2}}_{2,1}\hookrightarrow L^{\infty}$, we also have
\begin{equation}\label{21142}
\begin{aligned}
&\|\div w^{n}\|_{L^{\infty}_{t}(L^{\infty})}\|(b^{n},w^{n})\|_{\widetilde{L}^{\infty}_{t}(\dot{B}^{\frac{d}{2}-1}_{2,1})}\leq C\|(b_{0},w_{0})\|_{\dot{B}^{\frac{d}{2}+1}_{2,1}\cap\dot{B}^{\frac{d}{2}-1}_{2,1}}^2.
\end{aligned}
\end{equation}
One concludes from the commutator estimate (\ref{commutator}) and the property $(\ref{In})_{5}$ that
\begin{equation}\label{21143}
\begin{aligned}
&\sum_{j\in\mathbb{Z}}2^{j(\frac{d}{2}-1)}\|[w^{n}\cdot\nabla ,\dot{\Delta}_{j}] (b^{n},w^{n})\|_{L^{\infty}_{t}(L^2)}\leq C\|(b_{0},w_{0})\|_{\dot{B}^{\frac{d}{2}+1}_{2,1}\cap\dot{B}^{\frac{d}{2}-1}_{2,1}}^2.
\end{aligned}
\end{equation}
We combine (\ref{2114})-(\ref{21143}) to obtain
\begin{equation}
\begin{aligned}
&\|(b^{n},w^{n})\|_{\widetilde{L}^{\infty}_{t}(\dot{B}^{\frac{d}{2}-1}_{2,1})}\leq \|(b_{0},w_{0})\|_{\dot{B}^{\frac{d}{2}-1}_{2,1}}+ C\big{(}\eta+\|u_{0}\|_{\dot{B}^{\frac{d}{2}-1}_{2,1}}+\|(b_{0},w_{0})\|_{\dot{B}^{\frac{d}{2}+1}_{2,1}\cap\dot{B}^{\frac{d}{2}-1}_{2,1}}^2\big{)}t.
\end{aligned}
\end{equation}
Similarly, one may get after a direct computation that
\begin{equation}\label{2115}
\begin{aligned}
&\|(b^{n},w^{n})\|_{\widetilde{L}^{\infty}_{t}(\dot{B}^{\frac{d}{2}+1}_{2,1})}\\
&\quad\leq \|(b_{0},w_{0})\|_{\dot{B}^{\frac{d}{2}+1}_{2,1}}+ C \|(\widetilde{u}^{n},u^{n}_{L})\|_{L^1_{t}(\dot{B}^{\frac{d}{2}+1}_{2,1})}+C\|w^{n}\|_{\widetilde{L}^{\infty}_{t}(\dot{B}^{\frac{d}{2}+1}_{2,1})}\|(b^{n},w^{n})\|_{\widetilde{L}^{\infty}_{t}(\dot{B}^{\frac{d}{2}+1}_{2,1})}t\\
&\quad\leq \|(b_{0},w_{0})\|_{\dot{B}^{\frac{d}{2}+1}_{2,1}}+C\big{(}\eta+\eta^2+\|(b_{0},w_{0})\|_{\dot{B}^{\frac{d}{2}+1}_{2,1}\cap\dot{B}^{\frac{d}{2}-1}_{2,1}}^2t\big{)},\quad t\in I^{n}.
\end{aligned}
\end{equation}
According to (\ref{2114})-(\ref{2115}), the condition $(\ref{claim1})_{5}$ holds provided that $T$ is given by (\ref{T}) and $\eta$ satisfies
\begin{equation}\label{T6}
\begin{aligned}
&C(\eta+\eta^2)+C\big{(}\eta+\|u_{0}\|_{\dot{B}^{\frac{d}{2}-1}_{2,1}}+\|(b_{0},w_{0})\|_{\dot{B}^{\frac{d}{2}-1}_{2,1}\cap\dot{B}^{\frac{d}{2}+1}_{2,1}}^2\big{)}\eta^2<\|(b_{0},w_{0})\|_{\dot{B}^{\frac{d}{2}-1}_{2,1}\cap\dot{B}^{\frac{d}{2}+1}_{2,1}}.
\end{aligned}
\end{equation}
If we first choose $\beta$ so that (\ref{T1}) and $(\ref{T5})_{1}$ hold, thence fix $\eta$ to have $(\ref{T2})$,-(\ref{T4}), $(\ref{T5})_{2}$ and $(\ref{T6})$, and finally choose the time $T$ in $(\ref{T})$, then we conclude all the condtions in (\ref{claim1}) and obtain the estimates of the approximate sequence $(\tilde{a}^{n},\tilde{u}^{n},b^{n},w^{n})$ uniformly with respect to $n\geq n_{0}$ on the lifespan $[0,T]$.

\textbf{Srep 3: Compactness and convergence}

To obtain the strong convergence of $(\tilde{a}^{n},\tilde{u}^{n},b^{n},w^{n})$ in suitable senses such that every nonlinearity converge to the corresponding term of (\ref{m11}) in the sense of distributions, we need to estimate the time derivatives of $(a^{n},u^{n},b^{n},w^{n})$ uniformly with respect to $n\geq n_{0}$. By the uniform estimates obtained in Step 2, (\ref{uv2}), $(\ref{appm1})_{1}$ and (\ref{In}), we have
\begin{equation}\nonumber
\begin{aligned}
&\sup_{n\geq n_{0}}\|\partial_{t}\tilde{a}^{n}\|_{\widetilde{L}^{2}_{T}(\dot{B}^{\frac{d}{2}-1}_{2,1})}\lesssim \sup_{n\geq n_{0}}\big{(}\|(\tilde{u}^{n},u^{n}_{L})\|_{\widetilde{L}^{2}_{T}(\dot{B}_{2,1}^{\frac{d}{2}})}\|\tilde{a}^{n}\|_{\widetilde{L}^{\infty}_{T}(\dot{B}^{\frac{d}{2}}_{2,1})}+\|(\tilde{u}^{n},u^{n}_{L})\|_{\widetilde{L}^2_{T}(\dot{B}^{\frac{d}{2}}_{2,1})}\big{)}<\infty.
\end{aligned}
\end{equation}
Similarly, one can show
\begin{equation}\nonumber
\begin{aligned}
&\sup_{n\geq n_{0}}\big{(}\|\partial_{t}\tilde{u}^{n}\|_{L^1_{T}(\dot{B}^{\frac{d}{2}-1}_{2,1})}+\|\partial_{t}b^{n}\|_{\widetilde{L}^{\infty}_{T}(\dot{B}^{\frac{d}{2}}_{2,1})}+\|\partial_{t}w^{n}\|_{\widetilde{L}^2_{T}(\dot{B}^{\frac{d}{2}}_{2,1})}\big{)}<\infty.
\end{aligned}
\end{equation}
Then by virtue of the Aubin-Lions lemma and the cantor diagonal process, there is a limit $(\tilde{a},\tilde{u},b,w)$ such that up to a subsequence $(\tilde{a}^{n_{i}},\tilde{u}^{n_{i}},b^{n_{i}},w^{n_{i}})$ for $n_{i}\geq n_{0}$, it holds as $n_{i}\rightarrow \infty$ that
\begin{equation}\nonumber
\left\{
\begin{aligned}
&(\chi\tilde{a}^{n_{i}},\chi\tilde{u}^{n_{i}})\rightarrow (\chi\tilde{a},\chi\tilde{u})\quad\quad\text{in}\quad L^2(0,T;\dot{B}^{\frac{d}{2}-1}_{2,1})\times L^2(0,T;\dot{B}^{\frac{d}{2}}_{2,1}),\\
&(\chi b^{n_{i}},\chi w^{n_{i}})\rightarrow (\chi b, \chi w)\quad\quad\text{in}\quad L^2(0,T;\dot{B}^{\frac{d}{2}}_{2,1})\times L^2(0,T;\dot{B}^{\frac{d}{2}}_{2,1}).
\end{aligned}
\right.
\end{equation}
 Defining
\begin{equation}\nonumber
\begin{aligned}
&a:=\frac{1}{1+\tilde{a}}-1,\quad\quad u:=\tilde{u}+u_{L},
\end{aligned}
\end{equation}
we can show by the uniform estimates obtained in Step 2 and the Fatou property (cf. \cite[Page 443]{bahouri1}) that the regularities (\ref{r1})-(\ref{XX0}) hold for $(a,u,b,w)$, and therefore the limit $(a,u,b,w)$ is indeed a strong solution to the Cauchy problem (\ref{m1n}) on $[0,T]$.

\textbf{Step 4: Additional regularity}

Let the assumptions (\ref{a1}) be satisfied, and additionally  $a_{0}\in\dot{B}^{\frac{d}{2}-1}_{2,1}$. It follows from $(\ref{m1n})_{1}$, $(\ref{r1})$, (\ref{uv2}) and Lemma \ref{lemma31} that
\begin{equation}\nonumber
\begin{aligned}
&\|a\|_{\widetilde{L}^{\infty}_{T}(\dot{B}^{\frac{d}{2}-1}_{2,1})}\leq e^{C\|u\|_{L^1_{T}(\dot{B}^{\frac{d}{2}+1}_{2,1})}}\Big{(}\|a_{0}\|_{\dot{B}^{\frac{d}{2}-1}_{2,1}}+CT^{\frac{1}{2}}\big{(}\|u\|_{\widetilde{L}^{2}_{T}(\dot{B}^{\frac{d}{2}}_{2,1})}+\|a\|_{\widetilde{L}^{\infty}_{T}(\dot{B}^{\frac{d}{2}}_{2,1})}\|u\|_{\widetilde{L}^2_{T}(\dot{B}^{\frac{d}{2}}_{2,1})}\big{)}\Big{)},
\end{aligned}
\end{equation}
which yields $a\in \mathcal{C}([0,T];\dot{B}^{\frac{d}{2}-1}_{2,1})$. The proof of the local existence is complete.

\subsection{Uniqueness}

Assume that $(a_{0},u_{0},b_{0},w_{0})$ satisfies (\ref{a1}) of Theorem \ref{theorem11}. Let $(a_i,u_i,b_i,w_{i})$ $(i=1,2)$ be two solutions to the Cauchy problem $(\ref{m1n})$ satisfying (\ref{r1}) on $[0,T_{0}]$ for some time $T_{0}>0$ (which is independent of the lifespan for local existence) with the same initial data $(a_{0},u_{0},b_{0},w_{0})$. To prove uniqueness, one needs to estimate the discrepancy
$$
(\delta a,\delta u,\delta b,\delta w):=(a_2-a_1,u_2-u_1,b_{2}-b_{1},w_{2}-w_{1})
$$
with respect to a suitable norm. The system for $(\delta a,\delta u,\delta b,\delta w)$  reads
\begin{equation}\label{deltam1}
\left\{
\begin{aligned}
&\partial_{t}\delta a+u_2\cdot\nabla \delta a=\delta G_1,\\
&\partial_{t}\delta u+u_2\cdot\nabla \delta u+\delta u\cdot \nabla u_1-(1+a_2)^{-1}\mathcal{A}\delta u=\delta G_2,\\
&\partial_{t}\delta b+w_2\cdot\nabla \delta b+\div \delta b=\delta G_{3},\\
&\partial_{t}\delta w+w_2\cdot\nabla \delta w+\nabla \delta b+\kappa\delta w=\delta G_{4},
\end{aligned}
\right.
\end{equation}
 where $\delta G_{i}$ $(i=1,2,3,4)$ are given by
\begin{equation}\nonumber
\left\{
\begin{aligned}
&\delta G_1:=-\delta u \cdot \nabla a_1-\delta a \div u_2-(1+a_1)\div \delta u,\\
&\delta G_2:=-\frac{\delta a}{(a_{1}+1)(a_{2}+1)}\mathcal{A}u_1+\nabla \big{(}\pi(a_{2})-\pi(a_{1})\big{)}+\kappa(\frac{e^{b_{2}}}{a_{2}+1}-\frac{e^{b_{1}}}{a_{1}+1})(w_{2}-u_{2})\\
&\quad\quad\quad+\frac{ e^{b_{1}}}{a_{1}+1}(\delta w-\delta u),\\
&\delta G_3:=- \delta w \cdot \nabla b_1,\quad\quad \delta G_4:=-\delta w\cdot \nabla w_1+\kappa\delta u.
\end{aligned}
\right.
\end{equation}

\textbf{Case 1: $d\geq 3$.}

First, we estimate $\delta a$. It holds by the product estimate (\ref{uv2}) that
\begin{equation}\nonumber
\begin{aligned}
&\|\delta G_1\|_{\dot{B}^{\frac{d}{2}-1}_{2,1}}\leq C\big{(} (1+\|a_1\|_{\dot{B}^{\frac{d}{2}}_{2,1}})\| \delta u\|_{\dot{B}^{\frac{d}{2}}_{2,1}} +\| u_2\|_{\dot{B}^{\frac{d}{2}+1}_{2,1}}\|\delta a\|_{\dot{B}^{\frac{d}{2}-1}_{2,1}}\big{)},
\end{aligned}
\end{equation}
and then we apply Lemma \ref{lemma31} to get
\begin{equation}\nonumber
\begin{aligned}
\|\delta a\|_{\widetilde{L}^{\infty}_{t}(\dot{B}^{\frac{d}{2}-1}_{2,1})}&\leq Ce^{C\|u_2\|_{L^1_{t}(\dot{B}^{\frac{d}{2}+1}_{2,1})}}\|\delta G_1\|_{L^1_{T}(\dot{B}^{\frac{d}{2}-1}_{2,1})}\\
&\leq Ce^{C\|u_2\|_{L^1_{t}(\dot{B}^{\frac{d}{2}+1}_{2,1})}}\big{(}  (1+\| a_1\|_{L^{\infty}_{t}(\dot{B}^{\frac{d}{2}}_{2,1})})\int_{0}^{t}\|\delta u\|_{\dot{B}^{\frac{d}{2}}_{2,1}}d\tau +\int_{0}^{t} \| u_2\|_{\dot{B}^{\frac{d}{2}+1}_{2,1}}\|\delta a\|_{\dot{B}^{\frac{d}{2}-1}_{2,1}}d\tau \big{)}.
\end{aligned}
\end{equation}
This together with the Gr${\rm{\ddot{o}}}$nwall inequality gives
\begin{equation}
\begin{aligned}
\|\delta a\|_{\widetilde{L}^{\infty}_{t}(\dot{B}^{\frac{d}{2}-1}_{2,1})}\leq \widetilde{C}_{T_{0}}\int_{0}^{t}\|\delta u\|_{\dot{B}^{\frac{d}{2}}_{2,1}}d\tau,\quad t\in[0,T_{0}],\label{deltaa}
\end{aligned}
\end{equation}
where $\widetilde{C}_{T_{0}}>0$ stands for a constant depending on $T_{0}$ and the regularities of $(a_{i},u_{i},b_{i}.w_{i})$ $(i=1,2)$.

Next, since the product
of functions maps $\dot{B}^{\frac{d}{2}-2}_{2,1}\times\dot{B}^{\frac{d}{2}}_{2,1}$ to $\dot{B}^{\frac{d}{2}-2}_{2,1}$ in (\ref{uv2}) for any $d\geq 3$, we have
\begin{equation}\label{nonlineardeltawn}
\begin{aligned}
&\|(\delta G_{3},\delta G_{4})\|_{\dot{B}^{\frac{d}{2}-2}_{2,1}}\leq C\|(b_{1},w_{1})\|_{\dot{B}^{\frac{d}{2}+1}_{2,1}}\|\delta w\|_{\dot{B}^{\frac{d}{2}-2}_{2,1}}+C \|\delta u\|_{\dot{B}^{\frac{d}{2}-2}_{2,1}}.
\end{aligned}
\end{equation}
Applying the operator $\dot{\Delta}_{j}$ to $(\ref{deltam1})_{3}$-$(\ref{deltam1})_{4}$ and taking the $L^2$ inner product of the resulting equations with $\dot{\Delta}_{j}\delta b$ and $\dot{\Delta}_{j}\delta w$, respectively, we obtain
\begin{equation}\nonumber
\begin{aligned}
&\frac{1}{2}\frac{d}{dt}\|\dot{\Delta}_{j}(\delta b,\delta w)\|_{L^2}^2+\kappa\|\dot{\Delta}_{j}\delta w\|_{L^2}^2\\
&\quad\leq \big{(}\frac{1}{2}\|\div w_{2}\|_{L^{\infty}}\|\dot{\Delta}_{j}(\delta b,\delta w)\|_{L^2}+\|[w_{2}\cdot \nabla, \dot{\Delta}_{j}](\delta b,\delta w)\|_{L^2}+\|\dot{\Delta}_{j} (\delta G_{3},\delta G_{4})\|_{L^2}\big{)}\|\dot{\Delta}_{j}(b,w)\|_{L^2},
\end{aligned}
\end{equation}
which together with (\ref{r1}), $(\ref{commutator})$, (\ref{nonlineardeltawn}), and the Gr${\rm{\ddot{o}}}$nwall inequality yields
\begin{equation}\label{deltabw}
\begin{aligned}
&\|(\delta b,\delta w)\|_{\widetilde{L}^{\infty}_{t}(\dot{B}^{\frac{d}{2}-2}_{2,1})}\leq \widetilde{C}_{T_{0}}\int_{0}^{t}\|\delta u\|_{\dot{B}^{\frac{d}{2}-2}_{2,1}}d\tau,\quad  t\in[0,T_{0}].
\end{aligned}
\end{equation}

Then, $\delta u$ can be estimated by taking advantage of Lemma \ref{lemma35} for $s=\frac{d}{2}-2$. By the maps $\dot{B}^{\frac{d}{2}-2}_{2,1}\times\dot{B}^{\frac{d}{2}}_{2,1}\rightarrow \dot{B}^{\frac{d}{2}-2}_{2,1}$ and $\dot{B}^{\frac{d}{2}-1}_{2,1}\times\dot{B}^{\frac{d}{2}-1}_{2,1}\rightarrow \dot{B}^{\frac{d}{2}-2}_{2,1}$ in $(\ref{uv2})$ and the composition estimates (\ref{F1})-(\ref{F2}),  we have
\begin{equation}\label{deltaunolinear1}
\begin{aligned}
&\|(\frac{e^{b_{2}}}{a_{2}+1}-\frac{e^{b_{1}}}{a_{1}+1})(w_{2}-u_{2})\|_{\dot{B}^{\frac{d}{2}-2}_{2,1}}\\
&\quad\leq C\|\frac{e^{b_{1}}(w_{2}-u_{2})}{a_{2}+1}\|_{\dot{B}^{\frac{d}{2}}_{2,1}}\|e^{\delta b}-1\|_{\dot{B}^{\frac{d}{2}-2}_{2,1}}+C\|\frac{e^{b_{1}}}{(a_{1}+1)(a_{2}+1)}(w_{2}-u_{2})\|_{\dot{B}^{\frac{d}{2}-1}_{2,1}}\|\delta a\|_{\dot{B}^{\frac{d}{2}-1}_{2,1}}\\
&\quad\leq \widetilde{C}_{T_{0}}\big{(}1+\|(a_{2},b_{1},b_{2})\|_{\dot{B}^{\frac{d}{2}}_{2,1}}\big{)}^3\|(w_{2},u_{2})\|_{\dot{B}^{\frac{d}{2}}_{2,1}}\|\delta b\|_{\dot{B}^{\frac{d}{2}-2}_{2,1}}\\
&\quad\quad+\widetilde{C}_{T_{0}}\big{(}1+\|(a_{1},a_{2},b_{1})\|_{\dot{B}^{\frac{d}{2}}_{2,1}}\big{)}^3\|(w_{2},u_{2})\|_{\dot{B}^{\frac{d}{2}-1}_{2,1}}\|\delta a\|_{\dot{B}^{\frac{d}{2}-1}_{2,1}}.
\end{aligned}
\end{equation}
Similarly, one can show after direct computations that
\begin{equation}\label{deltaunolinear2}
\left\{
\begin{aligned}
&\|\frac{\delta a}{(a_{1}+1)(a_{2}+1)}\mathcal{A}u_1\|_{\dot{B}^{\frac{d}{2}-2}_{2,1}}\leq \widetilde{C}_{T_{0}}\big{(}1+\|(a_{1},a_{2})\|_{\dot{B}^{\frac{d}{2}}_{2,1}}\big{)}^{2}\|u_{1}\|_{\dot{B}^{\frac{d}{2}+1}_{2,1}}\|\delta a\|_{\dot{B}^{\frac{d}{2}-1}_{2,1}},\\
&\|\nabla\big{(}\pi(a_{2})-\pi(a_{1})\big{)}\|_{\dot{B}^{\frac{d}{2}-2}_{2,1}}\leq \widetilde{C}_{T_{0}}\big{(}1+\|(a_{1},a_{2})\|_{\dot{B}^{\frac{d}{2}}_{2,1}}\big{)}\|\delta a\|_{\dot{B}^{\frac{d}{2}-1}_{2,1}},\\
&\|\frac{e^{b_{1}}}{a_{1}+1}(\delta w-\delta u)\|_{\dot{B}^{\frac{d}{2}-2}_{2,1}}\leq \widetilde{C}_{T_{0}}\big{(}1+\|(a_{1},b_{1})\|_{\dot{B}^{\frac{d}{2}}_{2,1}}\big{)}^2\|(\delta w,\delta u)\|_{\dot{B}^{\frac{d}{2}-2}_{2,1}}.
\end{aligned}
\right.
\end{equation}
Notice that the conditions in Lemma \ref{lemma32} are satisfied due to the uniform estimates obtained in Subsection 3.2, and thence one has by Lemma \ref{lemma32}, (\ref{deltaa}) and (\ref{deltabw})-(\ref{deltaunolinear2}) for any $t\in[0,T_{0}]$ that
\begin{equation}\nonumber
\begin{aligned}
&\|\delta u\|_{\widetilde{L}^{\infty}_{t}(\dot{B}^{\frac{d}{2}-2}_{2,1})}+\|\delta u\|_{L^1_{t}(\dot{B}^{\frac{d}{2}}_{2,1})}\\
&\quad\leq Ce^{\widetilde{C}_{T_{0}}\big{(}\|a_2\|_{\widetilde{L}^{\infty}_{t}(\dot{B}^{\frac{d}{2}}_{2,1})}^2+\|(u_1,u_2)\|_{L^1_{t}(\dot{B}^{\frac{d}{2}+1}_{2,1})}\big{)}}\int_{0}^{t}\|\delta G_2\|_{\dot{B}^{\frac{d}{2}-2}_{2,1}}d\tau\\
&\quad\leq \widetilde{C}_{T_{0}}\int_{0}^{t}\big{(}1+\|u_{1}\|_{\dot{B}^{\frac{d}{2}+1}_{2,1}}+\|u_{2}\|_{\dot{B}^{\frac{d}{2}}_{2,1}}\big{)}\big{(}\|\delta u\|_{\dot{B}^{\frac{d}{2}-1}_{2,1}}+\|\delta u\|_{L^1_{\tau}(\dot{B}^{\frac{d}{2}}_{2,1})}\big{)}d\tau,
\end{aligned}
\end{equation}
which with the Gr${\rm{\ddot{o}}}$nwall inequality, (\ref{deltaa}) and (\ref{deltabw}) gives rise to
\begin{equation}\label{deltau}
\begin{aligned}
&\|\delta u\|_{\widetilde{L}^{\infty}_{t}(\dot{B}^{\frac{d}{2}-2}_{2,1})}+\|\delta u\|_{L^1_{t}(\dot{B}^{\frac{d}{2}}_{2,1})}+\|a\|_{\widetilde{L}^{\infty}_{t}(\dot{B}^{\frac{d}{2}-1}_{2,1})}+\|(b,w)\|_{\widetilde{L}^{\infty}_{t}(\dot{B}^{\frac{d}{2}-2}_{2,1})}=0, \quad t\in[0,T_{0}].
\end{aligned}
\end{equation}
Thus, we prove the uniqueness in the case of $d\geq3$.

\textbf{Case 2: $d=2$.}

In this case, we need to use the product estimates (\ref{uv3}) instead of (\ref{uv2}). According to Lemma \ref{lemma31} and $\dot{B}^{0}_{2,\infty}\times \dot{B}^{1}_{2,1}\rightarrow\dot{B}^{0}_{2,\infty}$ in (\ref{uv3}), it holds
\begin{equation}\label{deltaacase2}
\begin{aligned}
&\|\delta a\|_{\widetilde{L}^{\infty}_{t}(\dot{B}^{0}_{2,\infty})}\leq \widetilde{C}_{T_{0}}\int_{0}^{t}\|\delta u\|_{\dot{B}^{1}_{2,1}}d\tau,\quad  t\in[0,T_{0}].
\end{aligned}
\end{equation}
And similarly to (\ref{nonlineardeltawn})-(\ref{deltabw}) in Case 1, one can show
\begin{equation}\label{deltabwcase2}
\begin{aligned}
&\|(\delta b,\delta w)\|_{\widetilde{L}^{\infty}_{t}(\dot{B}^{-1}_{2,1})}\leq \widetilde{C}_{T_{0}}\int_{0}^{t}\|\delta u\|_{\dot{B}^{-1}_{2,\infty}}d\tau,\quad t\in[0,T_{0}].
\end{aligned}
\end{equation}

Then, we fix a large integer $m$ and a small time $t_{*}\in (0,T_{0}]$ satisfying the conditions in Lemma \ref{lemma35} and use (\ref{F1}), (\ref{F3}) and (\ref{deltaacase2})-(\ref{deltabwcase2}) to have
\begin{equation}\label{deltaucase211}
\begin{aligned}
&\|\delta u\|_{\widetilde{L}^{\infty}_{t_{*}}(\dot{B}^{-1}_{2,\infty})}+\|\delta u\|_{L^1_{t_{*}}(\dot{B}^{1}_{2,\infty})}\\
&\quad\leq \widetilde{C}_{T_{0}}\int_{0}^{t_{*}}\big{(} (1+\|u_{2}\|_{\dot{B}^{1}_{2,1}})\|\delta u\|_{\dot{B}^{-1}_{2,\infty}}+(1+\|u_{1}\|_{\dot{B}^{2}_{2,1}})\|\delta u\|_{L^1_{\tau}(\dot{B}^{1}_{2,1})}\big{)}d\tau.
\end{aligned}
\end{equation}
Employing the Gr${\rm{\ddot{o}}}$nwall inequality to (\ref{deltaucase211}) and using the log type inequality (\ref{log}), we obtain
\begin{equation}\nonumber
\begin{aligned}
&\|\delta u\|_{L^1_{t_{*}}(\dot{B}^{1}_{2,\infty})}\leq \widetilde{C}_{T_{0}}\int_{0}^{t_{*}}(1+\|u_1\|_{\dot{B}^{2}_{2,1}})\|\delta u\|_{L^1_{\tau}(\dot{B}^1_{2,\infty})}\log{\Big{\{} e+\frac{\widetilde{C}_{T_{0}}}{\|\delta u\|_{L^1_{\tau}(\dot{B}^1_{2,\infty})}}\Big{\}}}d\tau.
\end{aligned}
\end{equation}
Thus, we apply the the Osgood lemma (cf. \cite{bahouri1}[Page 125]) to show $\|\delta u\|_{L^1_{t_{*}}(\dot{B}^{1}_{2,\infty})}=0$. This together with (\ref{deltaacase2})-(\ref{deltabwcase2}) leads to the uniqueness on the time interval $[0,t_{*}]\subset [0,T_{0}]$. Since we have $(a_{i},u_{i},b_{i},w_{i})(t_{*})=0$, one can show $(a_{i},u_{i},b_{i},w_{i})(t)=0$ for any $t\in[t_{*},2t_{*}]$. Taking a finite connectivity argument, we prove the uniqueness on the time interval $[0,T_{0}]$.

\section{Global existence}\label{sectionglobal}

\subsection{The a-priori estimates}

%By the continuation argument for the local solution to the Cauchy problem (\ref{m1n}) given in Theorem \ref{theorem11}, to prove Theorem \ref{theorem12}, it is sufficient to establish the uniform-in-time a-priori estimates below.
In order to prove the global existence of the solution to the Cauchy problem \eqref{m1n}, we need to establish the uniform-in-time a-priori estimates below.

\begin{prop}\label{apriorie}
For given time $T>0$, suppose that the strong solution $(a,u,b,w)$ to the Cauchy problem \eqref{m1n} satisfies for $t\in(0,T)$ that
\begin{equation}\label{XX00}
\begin{aligned}
\mathcal{X}(t)&:=\|(a,u,b,w)\|_{\widetilde{L}^{\infty}_{t}(\dot{B}^{\frac{d}{2}-1}_{2,1})}^{\ell}+\|(a,u,b,w)\|_{L^1_{t}(\dot{B}^{\frac{d}{2}+1}_{2,1})}^{\ell}+\|u-w\|_{L^1_{t}(\dot{B}^{\frac{d}{2}}_{2,1})}^{\ell}+\|u-w\|_{\widetilde{L}^2_{t}(\dot{B}^{\frac{d}{2}-1}_{2,1})}^{\ell}\\
&\quad+\| a\|_{\widetilde{L}^{\infty}_{t}(\dot{B}^{\frac{d}{2}}_{2,1})}^{h}+\|u\|_{\widetilde{L}^{\infty}_{t}(\dot{B}^{\frac{d}{2}-1}_{2,1})}^{h}+\|(b,w)\|_{\widetilde{L}^{\infty}_{t}(\dot{B}^{\frac{d}{2}+1}_{2,1})}^{h}+\|a\|_{L^1_{t}(\dot{B}^{\frac{d}{2}}_{2,1})}^{h}+\|(u,b,w)\|_{L^1_{t}(\dot{B}^{\frac{d}{2}+1}_{2,1})}^{h}\\
&\leq 2C_{0}\mathcal{X}_{0},\quad\quad t\in(0,T),
\end{aligned}
\end{equation}
where $C_{0}>1$ is a constant independent of the time $T>0$, and $\mathcal{X}_{0}$ is defined by \eqref{a2}. There exist a small constant $\var_{0}>0$ such that if $\mathcal{X}_{0}\leq \var_{0}$, then it holds
\begin{equation}
\begin{aligned}
\mathcal{X}(t)\leq C_{0}\mathcal{X}_{0},\quad \quad t\in(0,T).\label{X2}
\end{aligned}
\end{equation}
\end{prop}

The proof of Proposition \ref{apriorie} consists of Lemmas \ref{lemma41}-\ref{prop42} below.

\vspace{2ex}

\underline{\it\textbf{Proof of Theorem \ref{theorem12}:}}~Let the assumptions of Theorem \ref{theorem12} hold. According to Theorem \ref{theorem11}, there exists a time $T_{0}>0$ such that the Cauchy problem \eqref{m1n} has a unique strong solution $(a,u,b,w)(x,t)$ for $t\in(0,T_{0}]$ satisfying \eqref{r1}. By virtue of Proposition \ref{apriorie}, the solution $(a,u,b,w)$ indeed satisfies $\sup_{t\in(0,T_{0}]}\mathcal{X}(t)\leq C_{0}\mathcal{X}_{0}$, and therefore by the standard continuity arguments, we can extend the solution $(a,u,b,w)$ globally in time and verify that $(a,u,b,w)$ satisfies the properties \eqref{rglobal}-\eqref{XX0}.

%\underline{\it\textbf{Proof of Theorem \ref{theorem12}:}}~Let the assumptions of Theorem \ref{theorem12} hold. According to Theorem \ref{theorem11}, there exists a time $T_{0}>0$ such that the Cauchy problem \eqref{m1n} has a unique strong solution $(a,u,b,w)(x,t)$ for $t\in(0,T_{0}]$ satisfying \eqref{r1} and $\sup_{t\in(0,T_{0}]}\mathcal{X}(t)\leq 2C_{0}\mathcal{X}_{0}$. By virtue of Proposition \ref{apriorie}, the solution $(a,u,b,w)$ indeed satisfies $\sup_{t\in(0,T_{0}]}\mathcal{X}(t)\leq C_{0}\mathcal{X}_{0}$, and therefore we can take $(a,u,b,w)(x,T_{0})$ as a new initial data and employ Theorem \ref{theorem11} to get a strong solution $(a,u,b,w)(x,t)$ to the Cauchy problem \eqref{m1n} for $t\in[T_{0},T_{0}+T_{1}]$ with a suitably small time $T_{1}>0$. Continuing the same process for $0\leq t\leq T_{0}+NT_{1}$, $N=2,3,...$, we can obtain a unique global solution  $(a,u,b,w)$ to the Cauchy problem \eqref{m1n} satisfying the properties \eqref{rglobal}-\eqref{XX0}.

%there exists a time $T_{1}>0$ such that we can employ Theorem \ref{theorem11} again to extend the solution $(a,u,b,w)$ from 

 %we are able to extend the solution $(a,u,b,w)$ globally in time and prove that $(a,u,b,w)$ satisfies the properties \eqref{rglobal}-\eqref{XX0}.

\subsection{Low-frequency analysis}

In this subsection, we establish the a-priori estimates of solutions to the Cauchy problem $(\ref{m1n})$ in the low-frequency region $\{\xi\in\mathbb{R}^{d}~|~ |\xi|\leq \frac{8}{3}\}$. Note that \eqref{XX00} implies
\begin{equation}
\begin{aligned}
\sup_{(x,t)\in\mathbb{R}^{d}\times(0,T)}|a(x,t)|\leq \frac{1}{2}\Rightarrow \frac{1}{2}\leq \rho=1+a\leq \frac{3}{2},\quad \text{if}~\mathcal{X}_{0}<<1.\label{simsim2}
\end{aligned}
\end{equation}
The property \eqref{simsim2} will be used to handle the nonlinear terms $f(a)$, $g(a)$ and $h(a,b)$ in \eqref{G} by virtue of the composition estimates \eqref{F1}.  For any $j\in\mathbb{Z}$, applying the operator $\dot{\Delta}_{j}$ to $(\ref{m1n})_{1}$-$(\ref{m1n})_{2}$, we get
\begin{equation}\label{delta12}
\left\{
\begin{aligned}
&\partial_{t}\dot{\Delta}_{j}a+\div \dot{\Delta}_{j}u=-\div \dot{\Delta}_{j}(au),\\
&\partial_{t}\dot{\Delta}_{j}u+\nabla \dot{\Delta}_{j}a- \Delta\dot{\Delta}_{j}u+\dot{\Delta}_{j}(u-w)=-\dot{\Delta}_{j}(u\cdot\nabla u)+\dot{\Delta}_{j}G,\\
&\partial_{t}\dot{\Delta}_{j}b+\div \dot{\Delta}_{j}w=-\dot{\Delta}_{j}(w\cdot \nabla b),\\
&\partial_{t}\dot{\Delta}_{j}w+\nabla \dot{\Delta}_{j}b+\dot{\Delta}_j(w-u)=-\dot{\Delta}_{j}(w\cdot\nabla w).
\end{aligned}
\right.
\end{equation}

First, we derive a low-frequency Lyapunov type inequality of \eqref{delta12}.

\begin{lemma}\label{lemma41}
Let $(a,u,b,w)$ be any strong solution to the Cauchy problem $(\ref{m1n})$. Then, it holds for any $j\leq 0$ that
\begin{equation}\label{Lowinequality}
\begin{aligned}
&\frac{d}{dt}\mathcal{E}_{1,j}(t)+\mathcal{D}_{1,j}(t)\\
&\lesssim \big{(}\|\dot{\Delta}_{j}(2^{j}au, u\cdot \nabla u,w\cdot\nabla b,w\cdot\nabla w)\|_{L^2}+\|\dot{\Delta}_{j}G\|_{L^2} \big{)}\|\dot{\Delta}_{j}(a,u,b,w)\|_{L^2},
\end{aligned}
\end{equation}
where $\mathcal{E}_{1,j}(t)$ and $\mathcal{D}_{1,j}(t)$ are defined by
\begin{equation}\label{EL}
\left\{
\begin{aligned}
&\mathcal{E}_{1,j}(t):=\frac{1}{2}\|\dot{\Delta}_{j}(a,u,b,w) \|_{L^2}^2+\eta_{1}\big{(} \dot{\Delta}_{j}u~|~ \nabla\dot{\Delta}_{j}a\big{)}_{L^2}+\eta_{1}\big{(}\dot{\Delta}_{j}w~|~\nabla \dot{\Delta}_{j} b\big{)}_{L^2},\\
&\mathcal{D}_{1,j}(t):=\|\nabla\dot{\Delta}_{j}u \|_{L^2}^2+\|\dot{\Delta}_{j}(u-w) \|_{L^2}^2\\
&\quad\quad\quad\quad+\eta_{1}\Big{(}\|\nabla \dot{\Delta}_{j}a\|_{L^2}^2-\|\div \dot{\Delta}_{j}u\|_{L^2}^2+\big{(}\dot{\Delta}_{j}(u-w)-\Delta\dot{\Delta}_{j}u~|~\nabla\dot{\Delta}_{j}a\big{)}_{L^2}\Big{)}\\
&\quad\quad\quad\quad+\eta_{1}\Big{(}\|\nabla \dot{\Delta}_{j}b\|_{L^2}^2-\|\div \dot{\Delta}_{j}w\|_{L^2}^2+ \big{(}\dot{\Delta}_{j}(w-u)~|~ \nabla \dot{\Delta}_{j}b\big{)}_{L^2}\Big{)},
\end{aligned}
\right.
\end{equation}
with constant $\eta_{1}\in(0,1)$ to be determined later. 
\end{lemma}

\begin{proof}
Taking the $L^2$ inner product of $(\ref{delta12})_{1}$ and $(\ref{delta12})_{2}$ with $\dot{\Delta}_{j}a$ and $\dot{\Delta}_{j}u$, respectively, we have
\begin{equation}\label{auj}
\begin{aligned}
&\frac{1}{2}\frac{d}{dt}\|\dot{\Delta}_{j}(a,u)\|_{L^2}^2+ \|\nabla \dot{\Delta}_{j}u\|_{L^2}^2+\big{(}\dot{\Delta}_{j}(u-w)~|~\dot{\Delta}_{j}u\big{)}_{L^2}\\
&~~\leq\big{(}\dot{\Delta}_{j}(a u)~|~\nabla \dot{\Delta}_{j}a\big{)}_{L^2}-\big{(}\dot{\Delta}_{j}(u\cdot\nabla u+G)~|~\dot{\Delta}_{j}u\big{)}_{L^2}.
\end{aligned}
\end{equation}
By $(\ref{delta12})_{3}$-$\eqref{delta12}_{4}$, one deduces after a direct computation that
\begin{equation}\label{bwj}
\begin{aligned}
&\frac{1}{2}\frac{d}{dt}\|\dot{\Delta}_{j}(b,w) \|_{L^2}^2+\big{(}\dot{\Delta}_{j}(w-u)~|~\dot{\Delta}_{j}w\big{)}_{L^2}=-\big{(}\dot{\Delta}_{j}(w\cdot\nabla b)~|~\dot{\Delta}_{j}b\big{)}_{L^2}-\big{(} \dot{\Delta}_{j}(w\cdot \nabla w)~|~\dot{\Delta}_{j}w\big{)}_{L^2}.
\end{aligned}
\end{equation}
The combination of (\ref{auj})-(\ref{bwj}) leads to
\begin{equation}\label{E1}
\begin{aligned}
&\frac{1}{2}\frac{d}{dt}\|\dot{\Delta}_{j}(a,u,b,w) \|_{L^2}^2+\|\nabla\dot{\Delta}_{j}u \|_{L^2}^2+\|\dot{\Delta}_{j}(u-w) \|_{L^2}^2\\
&~~\leq \|\dot{\Delta}_{j}(au) \|_{L^2}\|\nabla\dot{\Delta}_{j}a\|_{L^2}+\big{(}\|\dot{\Delta}_{j}(u\cdot \nabla u,w\cdot\nabla b,w\cdot\nabla w)\|_{L^2}+\|\dot{\Delta}_{j}G\|_{L^2}\big{)}\|\dot{\Delta}_{j}(u,b,w)\|_{L^2},
\end{aligned}
\end{equation}

In order to obtain the dissipation of $a$ and $b$, we make use of $(\ref{delta12})_{1}$-$(\ref{delta12})_{2}$ to have
\begin{equation}\label{E2}
\begin{aligned}
&\frac{d}{dt}\big{(} \dot{\Delta}_{j}u~|~ \nabla\dot{\Delta}_{j}a\big{)}_{L^2}+\|\nabla \dot{\Delta}_{j}a\|_{L^2}^2-\|\div \dot{\Delta}_{j}u\|_{L^2}^2+\big{(}\dot{\Delta}_{j}(u-w)-\Delta\dot{\Delta}_{j}u~|~\nabla\dot{\Delta}_{j}a\big{)}_{L^2}\\
&\leq \big{(}\|\div\dot{\Delta}_{j}(u\cdot \nabla u)\|_{L^2}+\|\nabla\div\dot{\Delta}_{j}(au)\|_{L^2}+\|\div \dot{\Delta}_{j}G\|_{L^2}\big{)}\|\dot{\Delta}_{j}(a,u)\|_{L^2},
\end{aligned}
\end{equation}
and
\begin{equation}\label{E3}
\begin{aligned}
&\frac{d}{dt}\big{(}\dot{\Delta}_{j}w~|~\nabla \dot{\Delta}_{j} b\big{)}_{L^2}+\|\nabla\dot{\Delta}_{j}b\|_{L^2}^2-\|\div\dot{\Delta}_{j} w\|_{L^2}^2+\big{(}\dot{\Delta}_{j}(w-u)~|~\dot{\Delta}_{j}b\big{)}_{L^2}\\
&\leq \|\div \dot{\Delta}_{j}(w\cdot\nabla b,w\cdot \nabla w)\|_{L^2} \|\dot{\Delta}_{j}(b,w)\|_{L^2}.
\end{aligned}
\end{equation}
According to \eqref{E1}-\eqref{E3} and the Bernstein inequality, \eqref{Lowinequality} follows.
\end{proof}

Then, we have the following low-frequency estimates of solutions to the Cauchy problem $(\ref{m1n})$.
\begin{lemma}\label{prop41} 
Let $T>0$ be any given time, and $(a,u,b,w)$ be any strong solution to the Cauchy problem $(\ref{m1n})$ for $t\in(0,T)$. Then it holds
\begin{equation}\label{Low}
\begin{aligned}
&\|(a,u,b,w)\|_{\widetilde{L}^{\infty}_{t}(\dot{B}^{\frac{d}{2}-1}_{2,1})}^{\ell}+\|(a,u,b,w)\|_{L^1_{t}(\dot{B}^{\frac{d}{2}+1}_{2,1})}^{\ell}\\
&\quad\quad+\|u-w\|_{L^1_{t}(\dot{B}^{\frac{d}{2}}_{2,1})}^{\ell}+\|u-w\|_{\widetilde{L}^2_{t}(\dot{B}^{\frac{d}{2}-1}_{2,1})}^{\ell}\\
&\quad\lesssim \mathcal{X}_{0}+\mathcal{X}^2(t),\quad t\in(0,T),
\end{aligned}
\end{equation}
where $\mathcal{X}_{0}$ and $\mathcal{X}(t)$ are defined though \eqref{a2} and $(\ref{XX00})$, respectively.
\end{lemma}
\begin{proof}
Recall that $\mathcal{E}_{1,j}(t)$ and $\mathcal{D}_{1,j}(t)$ given by \eqref{EL} satisfy the Lyapunov type inequality \eqref{Lowinequality}. One can show for any $j\leq0$ that
\begin{equation}\label{ELsim}
\begin{aligned}
&(\frac{1}{2}-\frac{4}{3}\eta_{1})\|\dot{\Delta}_{j}(a,u,b,w)\|_{L^2}^2\leq \mathcal{E}_{1,j}(t)\leq       (\frac{1}{2}+\frac{4}{3}\eta_{1})\|\dot{\Delta}_{j}(a,u,b,w)\|_{L^2}^2,
\end{aligned}
\end{equation}
and
\begin{equation}\label{DLsim}
\begin{aligned}
\mathcal{D}_{1,j}(t)&\geq \frac{9}{16}2^{2j}\|\dot{\Delta}_{j}u\|_{L^2}^2+\|\dot{\Delta}_{j}(u-w)\|_{L^2}^2+\eta_{1}\big( \frac{9}{32}2^{2j}\|\dot{\Delta}_{j}(a,b)\|_{L^2}^2\\
&\quad\quad-C2^{2j}\|\dot{\Delta}_{j}u\|_{L^2}^2-C2^{2j}\|\dot{\Delta}_{j} w\|_{L^2}^2-C\|\dot{\Delta}_{j}(u-w)\|_{L^2}^2)\\
&\geq (\frac{9}{16}-C\eta_{1})2^{2j}\|\dot{\Delta}_{j}u\|_{L^2}^2+(1-C\eta_{1})\|\dot{\Delta}_{j}(u-w)\|_{L^2}^2\\
&\quad-C\eta_{1}2^{2j}\|\dot{\Delta}_{j}w\|_{L^2}^2+\frac{9\eta_{1}}{32}2^{2j}\|\dot{\Delta}_{j}(a,b)\|_{L^2}^2,
\end{aligned}
\end{equation}
where $C>1$ denotes a sufficiently large constant independent of time. Choosing a sufficiently small constant $\eta_{1}\in(0,1)$, we deduce by (\ref{ELsim})-(\ref{DLsim}) for any $j\leq 0$ that
\begin{equation}\label{ELsim1}
\begin{aligned}
&\mathcal{E}_{1,j}(t)\sim\|\dot{\Delta}_{j}(a,u,b,w)\|_{L^2}^2,
\end{aligned}
\end{equation}
and
\begin{equation}\label{DLsim1}
\begin{aligned}
\mathcal{D}_{1,j}(t)&\gtrsim 2^{2j} \|\dot{\Delta}_{j}u\|_{L^2}^2+\|\dot{\Delta}_{j}(u-w)\|_{L^2}^2-\eta_{1}2^{2j}\|\dot{\Delta}_{j}w\|_{L^2}^2+\eta_{1}2^{2j}\|\dot{\Delta}_{j}(a,b)\|_{L^2}^2\\
&\gtrsim 2^{2j} \|\dot{\Delta}_{j}(a,u,b,w)\|_{L^2}^2,
\end{aligned}
\end{equation}
where in the last inequality one has used the key fact
\begin{equation}\label{tri}
\begin{aligned}
&2^{2j}\|\dot{\Delta}_{j}u\|_{L^2}^2+\|\dot{\Delta}_{j}(u-w)\|_{L^2}^2\geq  \frac{1}{2}2^{2j}\|\dot{\Delta}_{j}w\|_{L^2}^2.
\end{aligned}
\end{equation}
By (\ref{Lowinequality}) and (\ref{ELsim1})-(\ref{DLsim1}), the following inequality holds:
\begin{equation}\label{Lowinequality1}
\begin{aligned}
&\frac{d}{dt}\mathcal{E}_{1,j}(t)+2^{2j}\mathcal{E}_{1,j}(t)\\
&\lesssim  \big{(}\|\dot{\Delta}_{j}(2^{j}au, u\cdot \nabla u,w\cdot\nabla b,w\cdot\nabla w)\|_{L^2}+\|\dot{\Delta}_{j}G\|_{L^2}\big{)}\sqrt{\mathcal{E}_{1,j}(t)},\quad j\leq 0.
\end{aligned}
\end{equation}
Thence we divide (\ref{Lowinequality1}) by $\big{(}\mathcal{E}_{1,j}(t)+\var_{*}^2\big{)}^{\frac{1}{2}}$ for $\var_{*}>0$, integrate the resulting inequality over $[0,t]$ and then take the limit as $\var_{*}\rightarrow 0$ to obtain
\begin{equation}\label{Lowinequality3}
\begin{aligned}
&\|\dot{\Delta}_{j}(a,u,b,w)\|_{L^2}+2^{2j}\int_{0}^{t}\|\dot{\Delta}_{j}(a,u,b,w)\|_{L^2}d\tau\\
&\quad\lesssim \|\dot{\Delta}_{j}(a_{0},u_{0},b_{0},w_{0})\|_{L^2}\\
&\quad\quad+\int_{0}^{t}\big{(}\|\dot{\Delta}_{j}(2^{j}au,u\cdot\nabla u,w\cdot\nabla b,w\cdot\nabla w)\|_{L^2}+\|\dot{\Delta}_{j}G\|_{L^2}\big{)}d\tau.
\end{aligned}
\end{equation}
Multiplying (\ref{Lowinequality3}) by $2^{j(\frac{d}{2}-1)}$, taking the supremum on $[0,t]$ and then summing over $j\leq 0$, we have
\begin{equation}\label{Lowinequality4}
\begin{aligned}
&\|(a,u,b,w)\|_{\widetilde{L}^{\infty}_{t}(\dot{B}^{\frac{d}{2}-1}_{2,1})}^{\ell}+\|(a,u,b,w)\|_{L^1_{t}(\dot{B}^{\frac{d}{2}+1}_{2,1})}^{\ell}\\
&\quad\lesssim \|(a_{0},u_{0},b_{0},w_{0})\|_{\dot{B}^{\frac{d}{2}-1}_{2,1}}^{\ell}+\|au\|_{L^1_{t}(\dot{B}^{\frac{d}{2}}_{2,1})}^{\ell}+\|(u\cdot \nabla u,w\cdot\nabla b,w\cdot\nabla w)\|_{L^1_{t}(\dot{B}^{\frac{d}{2}-1}_{2,1})}^{\ell}+\|G\|_{L^1_{t}(\dot{B}^{\frac{d}{2}-1}_{2,1})}^{\ell}.
\end{aligned}
\end{equation}
From \eqref{lh} and the definition of $\mathcal{X}(t)$, it is easy to check that
\begin{equation}\label{nonlinearsim}
\left\{
\begin{aligned}
&\|a\|_{\widetilde{L}_{t}^{\infty}(\dot{B}^{\frac{d}{2}-1}_{2,1}\cap \dot{B}^{\frac{d}{2}}_{2,1})}+\|u\|_{\widetilde{L}_{t}^{\infty}(\dot{B}^{\frac{d}{2}-1}_{2,1})}+\|(b,w)\|_{\widetilde{L}_{t}^{\infty}(\dot{B}^{\frac{d}{2}-1}_{2,1}\cap\dot{B}^{\frac{d}{2}+1}_{2,1})}\lesssim \mathcal{X}(t),\\
&\|(u,b,w)\|_{L_{t}^{1}(\dot{B}^{\frac{d}{2}+1}_{2,1})}+\|(a,u,b,w)\|_{\widetilde{L}_{t}^{2}(\dot{B}^{\frac{d}{2}}_{2,1})}+\|u-w\|_{\widetilde{L}_{t}^{2}(\dot{B}^{\frac{d}{2}-1}_{2,1})}\lesssim \mathcal{X}(t).
\end{aligned}
\right.
\end{equation}
To simplify calculations, we will employ the estimates (\ref{nonlinearsim}) frequently to control the nonlinear terms on the right-hand side of (\ref{Lowinequality4}). It follows by $(\ref{nonlinearsim})_{2}$ and (\ref{uv1}) that
\begin{equation}\label{Lownonlinear1}
\begin{aligned}
&\|au\|_{L^1_{t}(\dot{B}^{\frac{d}{2}}_{2,1})}\lesssim \|a\|_{\widetilde{L}^2_{t}(\dot{B}^{\frac{d}{2}}_{2,1})}\|u\|_{\widetilde{L}^2_{t}(\dot{B}^{\frac{d}{2}}_{2,1})}\lesssim \mathcal{X}^2(t).
\end{aligned}
\end{equation}
Due to $(\ref{nonlinearsim})_{2}$ and  (\ref{uv2}), it also holds
\begin{equation}
\begin{aligned}
&\|(u\cdot \nabla u,w\cdot\nabla b,w\cdot\nabla w)\|_{L^1_{t}(\dot{B}^{\frac{d}{2}-1}_{2,1})}\lesssim\|(u,b,w)\|_{\widetilde{L}^2_{t}(\dot{B}^{\frac{d}{2}}_{2,1})}^2\lesssim\mathcal{X}^2(t).\label{Lownonlinear2}
\end{aligned}
\end{equation}
To handle the nonlinear term $G$, we obtain from (\ref{XX00}), (\ref{simsim2}) and the continuity of composition functions in Lemma \ref{lemma24} that
\begin{equation}\nonumber
\begin{aligned}
&\|(g(a),f(a))\|_{\dot{B}^{\frac{d}{2}}_{2,1}}\lesssim \|a\|_{\dot{B}^{\frac{d}{2}}_{2,1}},\quad \|h(a,b)\|_{\dot{B}^{\frac{d}{2}}_{2,1}}\lesssim \|a\|_{\dot{B}^{\frac{d}{2}}_{2,1}}\|b\|_{\dot{B}^{\frac{d}{2}}_{2,1}}+\|a\|_{\dot{B}^{\frac{d}{2}}_{2,1}}+\|b\|_{\dot{B}^{\frac{d}{2}}_{2,1}}\lesssim \|(a,b)\|_{\dot{B}^{\frac{d}{2}}_{2,1}},
\end{aligned}
\end{equation}
which together with (\ref{nonlinearsim}) and  (\ref{uv1})-(\ref{uv2}) yields
\begin{equation}\label{Lownonlinear4}
\begin{aligned}
\|G\|_{L^1_{t}(\dot{B}^{\frac{d}{2}-1}_{2,1})}&\lesssim \|a\|_{\widetilde{L}^2_{t}(\dot{B}^{\frac{d}{2}}_{2,1})}^2+\|a\|_{\widetilde{L}^{\infty}_{t}(\dot{B}^{\frac{d}{2}}_{2,1})}\|u\|_{L^1_{t}(\dot{B}^{\frac{d}{2}+1}_{2,1})}\\
&\quad+\|(a,b)\|_{\widetilde{L}^2_{t}(\dot{B}^{\frac{d}{2}}_{2,1})}\|u-w\|_{\widetilde{L}^{2}_{t}(\dot{B}^{\frac{d}{2}-1}_{2,1})}\\
&\lesssim\mathcal{X}^2(t).
\end{aligned}
\end{equation}
We substitute (\ref{Lownonlinear1})-(\ref{Lownonlinear4}) into (\ref{Lowinequality4}) to get
\begin{equation}\label{Low1}
\begin{aligned}
&\|(a,u,b,w)\|_{\widetilde{L}^{\infty}_{t}(\dot{B}^{\frac{d}{2}-1}_{2,1})}^{\ell}+\|(a,u,b,w)\|_{L^1_{t}(\dot{B}^{\frac{d}{2}+1}_{2,1})}^{\ell}\\
&\quad\lesssim\|(a_{0},u_{0},b_{0},w_{0})\|_{\dot{B}^{\frac{d}{2}-1}_{2,1}}^{\ell}+\mathcal{X}^2(t).
\end{aligned}
\end{equation}

In addition, applying the operator $\dot{\Delta}_{j}$ to \eqref{relativedamp1}, taking the $L^2$-inner product of the resulting equation with $\dot{\Delta}_{j}(u-w)$ and then using the Bernstein inequality, we derive
\begin{equation}\label{Low2p}
\begin{aligned}
&\frac{d}{dt}\|\dot{\Delta}_{j}(u-w)\|_{L^2}^2+\|\dot{\Delta}_{j}(u-w)\|_{L^2}^2\\
&\lesssim \big( 2^{j}\|\dot{\Delta}_{j}(a,u,b)\|_{L^2}+\|\dot{\Delta}_{j}(u\cdot\nabla u,w\cdot \nabla w)\|_{L^2}+\|\dot{\Delta}_{j}G\|_{L^2}\big) \|\dot{\Delta}_{j}(u-w)\|_{L^2},\quad j\leq0,
\end{aligned}
\end{equation}
which gives rise to
\begin{equation}\label{low22pp}
\begin{aligned}
&\|\dot{\Delta}_{j}(u-w)\|_{L^2}^2+\int_{0}^{t}\|\dot{\Delta}_{j}(u-w)\|_{L^2}^2d\tau\\
&\quad\lesssim \|\dot{\Delta}_{j}(u_{0},w_{0})\|_{L^2}^2+ 2^{j}\|\dot{\Delta}_{j}(a,u,b)\|_{L^2_t(L^2)} \|\dot{\Delta}_{j}(u-w)\|_{L^2_t(L^2)}\\
&\quad\quad+ \big(\|\dot{\Delta}_{j}(u\cdot\nabla u,w\cdot \nabla w)\|_{L^1_t(L^2)}+\|\dot{\Delta}_{j}G\|_{L^1_t(L^2)}\big) \|\dot{\Delta}_{j}(u,w)\|_{L^{\infty}_t(L^2)}\\
&\quad\leq C\big( \|\dot{\Delta}_{j}(u_{0},w_{0})\|_{L^2}^2+\|\dot{\Delta}_{j}(u,w)\|_{L^{\infty}_t(L^2)}^2+ 2^{2j}\|\dot{\Delta}_{j}(a,u,b)\|_{L^2_t(L^2)}^2\big)\\
&\quad\quad+C\big( \|\dot{\Delta}_{j}(u\cdot\nabla u,w\cdot \nabla w)\|_{L^1_t(L^2)}^2+\|\dot{\Delta}_{j}G\|_{L^1_t(L^2)}^2\big)+\frac{1}{2} \int_{0}^{t}\|\dot{\Delta}_{j}(u-w)\|_{L^2}^2d\tau,\quad j\leq0.
\end{aligned}
\end{equation}
Multiplying \eqref{low22pp} by $2^{j(\frac{d}{2}-1)}$ and thence summing it over $j\leq0$, we have by (\ref{Lownonlinear2})-\eqref{Low1} and \eqref{inter} that
\begin{equation}\label{Low2}
\begin{aligned}
\|u-w\|_{\widetilde{L}^2_{t}(\dot{B}^{\frac{d}{2}-1}_{2,1})}^{\ell}&\lesssim \|(u_{0},w_{0})\|_{\dot{B}^{\frac{d}{2}-1}_{2,1}}^{\ell}+\|(u,w)\|_{\widetilde{L}^{\infty}_{t}(\dot{B}^{\frac{d}{2}-1}_{2,1})}^{\ell}+\|(a,u,b)\|_{\widetilde{L}^2_{t}(\dot{B}^{\frac{d}{2}}_{2,1})}^{\ell}\\
&\quad+\|(u\cdot\nabla u,w\cdot \nabla w)\|_{L^1_{t}(\dot{B}^{\frac{d}{2}-1}_{2,1})}^{\ell}+\|G\|_{L^1_{t}(\dot{B}^{\frac{d}{2}-1}_{2,1})}^{\ell}\\
&\lesssim\|(a_{0},u_{0},b_{0},w_{0})\|_{\dot{B}^{\frac{d}{2}-1}_{2,1}}^{\ell}+\mathcal{X}^2(t).
\end{aligned}
\end{equation}
Note that the inequality \eqref{Low2p} also implies for $j\leq0$ that
\begin{equation}\nonumber
\begin{aligned}
&\|\dot{\Delta}_{j}(u-w)\|_{L^2}+\int_{0}^{t}\|\dot{\Delta}_{j}(u-w)\|_{L^2}d\tau\\
&\quad\lesssim 2^{-j}\|\dot{\Delta}_{j}(u_{0},w_{0})\|_{L^2}+\int_{0}^{t} \big( 2^{j}\|\dot{\Delta}_{j}(a,u,b)\|_{L^2}+2^{-j}\|\dot{\Delta}_{j}(u\cdot\nabla u,w\cdot \nabla w)\|_{L^2}+2^{-j}\|\dot{\Delta}_{j}G\|_{L^2}\big) d\tau.
\end{aligned}
\end{equation}
Therefore, it holds that
\begin{equation}\label{Low3}
\begin{aligned}
&\|u-w\|_{L^1_{t}(\dot{B}^{\frac{d}{2}}_{2,1})}^{\ell}\\
&\quad\lesssim \|(u_{0},w_{0})\|_{\dot{B}^{\frac{d}{2}-1}_{2,1}}^{\ell}+\|(a,u,b)\|_{L^1_{t}(\dot{B}^{\frac{d}{2}+1}_{2,1})}^{\ell}+\|(u\cdot\nabla u,w\cdot\nabla w)\|_{L^1_{t}(\dot{B}^{\frac{d}{2}-1}_{2,1})}^{\ell}+\|G\|_{L^1_{t}(\dot{B}^{\frac{d}{2}-1}_{2,1})}^{\ell}\\
&\quad\lesssim\|(a_{0},u_{0},b_{0},w_{0})\|_{\dot{B}^{\frac{d}{2}-1}_{2,1}}^{\ell}+\mathcal{X}^2(t).
\end{aligned}
\end{equation}
By (\ref{Low1}) and (\ref{Low2})-\eqref{Low3}, we prove (\ref{Low}). 
\end{proof}

\subsection{High-frequency analysis}

In this subsection, we estimate solutions to the Cauchy problem \eqref{m1n} in the high-frequency region $\{\xi\in\mathbb{R}^{d}~|~ |\xi|\geq \frac{3}{8}\}$. To this end, we show a high-frequency Lyapunov type inequality of \eqref{delta12}.

\begin{lemma}\label{lemma45}
Let $(a,u,b,w)$ be any strong solution to the Cauchy problem $(\ref{m1n})$. Then, it holds for any $j\geq-1$ that
\begin{equation}\label{Highinequality}
\begin{aligned}
&\frac{d}{dt}\mathcal{E}_{2,j}(t)+\mathcal{D}_{2,j}(t)\\
&\lesssim \big{(}2^{j}\|\dot{\Delta}_{j}(au) \|_{L^2}+\|\dot{\Delta}_{j}(u\cdot \nabla u,w\cdot\nabla b,w\cdot\nabla w)\|_{L^2}+\|\dot{\Delta}_{j}G\|_{L^2}+\|\div u\|_{L^{\infty}}\|\nabla\dot{\Delta}_{j}a\|_{L^2}\\
&\quad+\|\nabla\dot{\Delta}_{j}(a\div u)\|_{L^2} +\|[u\cdot \nabla ,\dot{\Delta}_{j}]u\|_{L^2}+\sum_{k=1}^{d}\|[u\cdot \nabla ,\partial_{k} \dot{\Delta}_{j}]a\|_{L^2}\big{)}\|\dot{\Delta}_{j}(a,\nabla a,u,b,w)\|_{L^2},
\end{aligned}
\end{equation}
where $\mathcal{E}_{2,j}(t)$ and $\mathcal{D}_{2,j}(t)$ are defined by
\begin{equation}\label{EH}
\left\{
\begin{aligned}
&\mathcal{E}_{2,j}(t):=\frac{1}{2}\|\dot{\Delta}_{j}(a,u,b,w) \|_{L^2}^2+\eta_{2}\Big{(}\frac{1}{2} \|\nabla \dot{\Delta}_{j}a\|_{L^2}^2+\big{(} \dot{\Delta}_{j}u~|~ \nabla\dot{\Delta}_{j}a\big{)}_{L^2}\Big{)}\\
&\quad\quad\quad~~+\eta_{2}2^{-2j}\Big{(}\dot{\Delta}_{j}w~|~\nabla \dot{\Delta}_{j} b\big{)}_{L^2},\\
&\mathcal{D}_{2,j}(t):=\|\nabla\dot{\Delta}_{j}u \|_{L^2}^2+\|\dot{\Delta}_{j}(u-w) \|_{L^2}^2\\
&\quad\quad\quad~~+\eta_{2}\big{(}\|\nabla \dot{\Delta}_{j}a\|_{L^2}^2-\|\div \dot{\Delta}_{j}u\|_{L^2}^2+\big{(}\dot{\Delta}_{j}(u-w)~|~ \nabla \dot{\Delta}_{j}a\big{)}_{L^2}\Big{)}\\
&\quad\quad\quad~~+ \eta_{2}2^{-2j}\Big{(}\|\nabla \dot{\Delta}_{j}b\|_{L^2}^2-\|\div \dot{\Delta}_{j}w\|_{L^2}^2+\big{(}\dot{\Delta}_{j}(w-u)~|~ \nabla \dot{\Delta}_{j}b\big{)}_{L^2}\Big{)},
\end{aligned}
\right.
\end{equation}
with $\eta_{2}\in(0,1)$ a constant to be determined.
\end{lemma}

\begin{proof}
One can show from $(\ref{delta12})_{1}$-$(\ref{delta12})_{2}$ that
 \begin{equation}\label{E62}
\begin{aligned}
&\frac{d}{dt}\Big{(}\frac{1}{2} \|\nabla \dot{\Delta}_{j}a\|_{L^2}^2+\big{(} \dot{\Delta}_{j}u~|~ \nabla\dot{\Delta}_{j}a\big{)}_{L^2}\Big{)}\\
&\quad+\|\nabla \dot{\Delta}_{j}a\|_{L^2}^2-\|\div \dot{\Delta}_{j}u\|_{L^2}^2+ \big{(}\dot{\Delta}_{j}(u-w)~|~ \nabla \dot{\Delta}_{j}a\big{)}_{L^2}\\
&=-\big{(}\nabla \div \dot{\Delta}_{j}(au)~|~\nabla \dot{\Delta}_{j}a\big{)}_{L^2}-\big{(}\dot{\Delta}_{j}(u\cdot \nabla u)~|~ \nabla \dot{\Delta}_{j}a\big{)}_{L^2}\\
&\quad+\big{(}\nabla \div \dot{\Delta}_{j}(au)~|~\dot{\Delta}_{j}u\big{)}_{L^2}+\big{(}\dot{\Delta}_{j}G~|~\nabla \dot{\Delta}_{j}a\big{)}_{L^2}.
\end{aligned}
\end{equation}
Thence we decompose the following two nonlinearities as 
 \begin{equation}\nonumber
 \left\{
 \begin{aligned}
&\partial_{k} \div \dot{\Delta}_{j}(au)=-[u\cdot \nabla ,\partial_{k} \dot{\Delta}_{j}]a+u\cdot \nabla \partial_{k}\dot{\Delta}_{j}a+\partial_{k}\dot{\Delta}_{j}(a\div u),\quad k=1,...,d,\\
&\dot{\Delta}_{j}(u\cdot \nabla u)=-[u\cdot\nabla , \dot{\Delta}_{j}]u+u\cdot \nabla \dot{\Delta}_{j}u,
 \end{aligned}
 \right.
 \end{equation}
so that it holds
\begin{equation}\label{E63}
\begin{aligned}
&\big{|}\big{(}\nabla \div \dot{\Delta}_{j}(au)~|~\nabla \dot{\Delta}_{j}a\big{)}_{L^2}\big{|}\\
&\quad=\big{|}-\sum_{k=1}^{d}\big{(}[u\cdot \nabla ,\partial_{k} \dot{\Delta}_{j}]a~|~\nabla \dot{\Delta}_{j}a\big{)}_{L^2}-\frac{1}{2}\big{(}\div u\nabla{\Delta}_{j}a~|~\nabla \dot{\Delta}_{j}a\big{)}_{L^2}+\big{(}\nabla \dot{\Delta}_{j}(a\div u)~|~\nabla \dot{\Delta}_{j}a\big{)}_{L^2}\big{|}\\
&\quad\leq \big{(}\sum_{k=1}^{d}\|[u\cdot \nabla ,\partial_{k}\dot{\Delta}_{j}]a\|_{L^2}+\frac{1}{2}\|\div u\|_{L^{\infty}}\|\nabla \dot{\Delta}_{j}a\|_{L^2}+\|\nabla \dot{\Delta}_{j}(a\div u)\|_{L^2}\big{)}\|\nabla \dot{\Delta}_{j}a\|_{L^2},
\end{aligned}
\end{equation}
and
\begin{equation}\label{E64}
\begin{aligned}
&\big{|}-\big{(}\dot{\Delta}_{j}(u\cdot \nabla u)~|~ \nabla \dot{\Delta}_{j}a\big{)}_{L^2}+\big{(}\nabla \div \dot{\Delta}_{j}(au)~|~ \dot{\Delta}_{j}u\big{)}_{L^2}\big{|}\\
&\quad=\big{|}\big{(}u\cdot\nabla \dot{\Delta}_{j}u~|~ \nabla \dot{\Delta}_{j}a\big{)}_{L^2}+\big{(}u\cdot \nabla \dot{\Delta}_{j}\nabla a~|~ \dot{\Delta}_{j} u\big{)}_{L^2}+\big{(}\nabla \dot{\Delta}_{j}(a\div u)~|~ \dot{\Delta}_{j}u\big{)}_{L^2}\\
&\quad\quad-\big{(}[u\cdot \nabla ,\dot{\Delta}_{j}]u~|~\nabla \dot{\Delta}_{j}a\big{)}_{L^2}-\sum_{k=1}^{d}\big{(}[u\cdot \nabla ,\partial_{k} \dot{\Delta}_{j}]a~|~ \dot{\Delta}_{j}u\big{)}_{L^2}\big{|}\\
&\quad\leq \big{(}\|\div u\|_{L^{\infty}}\|\dot{\Delta}_{j}u\|_{L^2}+\|\nabla \dot{\Delta}_{j}(a\div u)\|_{L^2}+\|[u\cdot \nabla\dot{\Delta}_{j}]u\|_{L^2}\|\nabla \dot{\Delta}_{j}a\|_{L^2}\\
&\quad\quad+\sum_{k=1}^{d}\|[u\cdot \nabla ,\partial_{k}\dot{\Delta}_{j}]a\|_{L^2}\big{)}\|\dot{\Delta}_{j}(\nabla a,u)\|_{L^2}.
\end{aligned}
\end{equation}
 The combination of (\ref{E62})-(\ref{E64}) gives rise to
\begin{equation}\label{E6}
\begin{aligned}
&\frac{d}{dt}\big{(}\frac{1}{2} \|\nabla \dot{\Delta}_{j}a\|_{L^2}^2+\big{(} \dot{\Delta}_{j}u~|~ \nabla\dot{\Delta}_{j}a\big{)}_{L^2}\big{)}+\|\nabla \dot{\Delta}_{j}a\|_{L^2}^2-\|\div \dot{\Delta}_{j}u\|_{L^2}^2+ \big{(}\dot{\Delta}_{j}(u-w)~|~ \nabla \dot{\Delta}_{j}a\big{)}_{L^2}\\
&\lesssim \big{(}\|\dot{\Delta}_{j}G\|_{L^2}+\|\div u\|_{L^{\infty}}\|\nabla\dot{\Delta}_{j}a\|_{L^2}+\|\nabla\dot{\Delta}_{j}(a\div u)\|_{L^2}+\|[u\cdot \nabla ,\dot{\Delta}_{j}]u\|_{L^2}\\
&\quad+\sum_{k=1}^{d}\|[u\cdot \nabla ,\partial_{k}\dot{\Delta}_{j}]a\|_{L^2}\big{)}\|\dot{\Delta}_{j}(u,\nabla a)\|_{L^2}.
\end{aligned}
\end{equation}
By virtue of (\ref{E1}), (\ref{E3}), (\ref{E6}), the Bernstein inequality and the fact $2^{-j}\leq 2$, \eqref{Highinequality} holds.
\end{proof}

Furthermore, the following Lyapunov-type inequality is used to derive the higher order estimates for $(b,w)$ in $\eqref{m1n}_{3}$-$\eqref{m1n}_{4}$.

\begin{lemma}\label{lemma46}
Let $(a,u,b,w)$ be any strong solution to the Cauchy problem $(\ref{m1n})$. Then, it holds for any $j\geq-1$ that
\begin{equation}\label{dEulerpp}
\begin{aligned}
&\frac{d}{dt}\mathcal{E}_{3,j}(t)+\mathcal{D}_{3,j}(t)\\
&\lesssim \|\dot{\Delta}_{j}u\|_{L^2} \sqrt{\mathcal{E}_{3,j}(t)}\\
&\quad+\big{(} \|\div w\|_{L^{\infty}}\|\dot{\Delta}_{j}(b,w)\|_{L^2}+2^{-j}\|\dot{\Delta}_{j}(w\cdot\nabla b,w\cdot\nabla w)\|_{L^2}\\
&\quad+\|[w\cdot \nabla,\dot{\Delta}_{j}](b,w)\|_{L^2}\big{)}\sqrt{\mathcal{E}_{3,j}(t)},
\end{aligned}
\end{equation}
where $\mathcal{E}_{3,j}(t)$ and $\mathcal{D}_{3,j}(t)$ are defined by
\begin{equation}\label{EHH}
\left\{
\begin{aligned}
&\mathcal{E}_{3,j}(t):=\frac{1}{2}\|\dot{\Delta}_{j}(b,w)\|_{L^2}^2+\eta_{3}2^{-2j}\big{(}\dot{\Delta}_{j}w~|~\nabla \dot{\Delta}_{j} b\big{)}_{L^2},\\
&\mathcal{D}_{3,j}(t):=\|\dot{\Delta}_{j}w\|_{L^2}^2+\eta_{3}2^{-2j}\Big{(}\|\nabla\dot{\Delta}_{j}b\|_{L^2}^2-\|\div\dot{\Delta}_{j} w\|_{L^2}^2+\big{(}\dot{\Delta}_{j}(w-u)~|~\dot{\Delta}_{j}b\big{)}_{L^2}\Big{)},
\end{aligned}
\right.
\end{equation}
for a constant $\eta_{3}\in(0,1)$ to be chosen.
\end{lemma}

\begin{proof}
According to $\eqref{bwj}_{3}$-$\eqref{bwj}_{4}$, we have
\begin{equation}\nonumber
\begin{aligned}
&\frac{1}{2}\frac{d}{dt}\|\dot{\Delta}_{j}(b,w) \|_{L^2}^2+\|\dot{\Delta}_{j}w \|_{L^2}^2\\
&~~\leq \big{(} \|\dot{\Delta}_{j}u\|_{L^2}+\frac{1}{2}\|\div w\|_{L^{\infty}}\|\dot{\Delta}_{j}(b,w)\|_{L^2}+\|[w\cdot \nabla,\dot{\Delta}_{j}](b,w)\|_{L^2}\big{)}   \|\dot{\Delta}_{j}(b,w)\|_{L^2},
\end{aligned}
\end{equation}
which together with \eqref{E3} leads to \eqref{dEulerpp}.
\end{proof}

Finally, we are ready to establish the expected high-frequency estimates of solutions to the Cauchy problem $(\ref{m1n})$.
\begin{lemma}\label{prop42}
Let $T>0$ be any given time, and $(a,u,b,w)$ be any strong solution to the Cauchy problem $(\ref{m1n})$ for $t\in(0,T)$ satisfying \eqref{XX0} and \eqref{simsim2}. Then it holds
\begin{equation}\label{High}
\begin{aligned}
&\| a\|_{\widetilde{L}^{\infty}_{t}(\dot{B}^{\frac{d}{2}}_{2,1})}^{h}+\|u\|_{\widetilde{L}^{\infty}_{t}(\dot{B}^{\frac{d}{2}-1}_{2,1})}^{h}+\|(b,w)\|_{\widetilde{L}^{\infty}_{t}(\dot{B}^{\frac{d}{2}+1}_{2,1})}^{h}\\
&\quad\quad+\|a\|_{L^1_{t}(\dot{B}^{\frac{d}{2}}_{2,1})}^{h}+\|(u,b,w)\|_{L^1_{t}(\dot{B}^{\frac{d}{2}+1}_{2,1})}^{h}\\
&\quad\lesssim \mathcal{X}_{0}+\mathcal{X}^2(t),\quad t\in(0,T),
\end{aligned}
\end{equation}
where $\mathcal{X}_{0}$ and $\mathcal{X}(t)$ are defined though \eqref{a2} and $(\ref{XX00})$, respectively.
\end{lemma}
\begin{proof}
We recall that the Lyapunov type inequality \eqref{Highinequality} holds for $\mathcal{E}_{2,j}(t)$ and $\mathcal{D}_{2,j}(t)$ given by \eqref{EH}. It is easy to verify for any $j\geq -1$ that
\begin{equation}\label{Highsim1}
\begin{aligned}
&(\frac{1}{2}-\eta_{2})\|\dot{\Delta}_{j}u\|_{L^2}^2+\frac{\eta_{2}}{4}\|\nabla\dot{\Delta}_{j}a\|_{L^2}^2\leq \mathcal{E}_{2,j}(t)\leq (\frac{1}{2}+\eta_{2})\|\dot{\Delta}_{j}u\|_{L^2}^2+\frac{3\eta_{2}}{4}\|\nabla\dot{\Delta}_{j}a\|_{L^2}^2,
\end{aligned}
\end{equation}
and
\begin{equation}\label{Highsim2}
\begin{aligned}
\mathcal{D}_{2,j}(t)&\geq \frac{9}{16}2^{2j}\|\dot{\Delta}_{j}u\|_{L^2}^2+\|\dot{\Delta}_{j}(u-w)\|_{L^2}^2\\
&\quad+\eta_{2}\big{(}\frac{1}{2}\|\nabla\dot{\Delta}_{j}a\|_{L^2}^2-\frac{64}{9} 2^{2j}\|\dot{\Delta}_{j}u\|_{L^2}^2-\frac{1}{2}\|\dot{\Delta}_{j}(u-w)\|_{L^2}^2\big{)}\\
&\quad+\eta_{2}2^{-2j}\big{(}\frac{9}{32}2^{2j}\|\dot{\Delta}_{j}b\|_{L^2}^2-\frac{64}{9}2^{2j}\|\dot{\Delta}_{j}w\|_{L^2}^2-\frac{1}{2}\|\dot{\Delta}_{j}(w-u)\|_{L^2}^2\big{)}\\
&\geq( \frac{9}{64} -\frac{64}{9}\eta_{2}) 2^{2j}\|\dot{\Delta}_{j}u\|_{L^2}^2+(1-\frac{5}{2}\eta_{2}) \|\dot{\Delta}_{j}(u-w)\|_{L^2}^2-\frac{64}{9}\eta_{2}\|\dot{\Delta}_{j}w\|_{L^2}\\
&\quad+\frac{\eta_{2}}{2}\|\nabla\dot{\Delta}_{j}a\|_{L^2}^2+\frac{9\eta_{2}}{32}\|\dot{\Delta}_{j}b\|_{L^2}^2.
\end{aligned}
\end{equation}
 By (\ref{tri}) and (\ref{Highsim1})-(\ref{Highsim2}), for any $j\geq -1$, one can choose a sufficiently small constant $\eta_{2}\in(0,1)$ so that we have
\begin{equation}\label{Highsim11}
\begin{aligned}
&\mathcal{E}_{2,j}(t)\sim\|\dot{\Delta}_{j}(a,\nabla a, u,b,w)\|_{L^2}^2,
\end{aligned}
\end{equation}
and
\begin{equation}\label{Highsim22}
\begin{aligned}
\mathcal{D}_{2,j}(t)&\gtrsim \|\dot{\Delta}_{j}u\|_{L^2}^2+\|\dot{\Delta}_{j}(u-w)\|_{L^2}^2-\eta_{2}\|\dot{\Delta}_{j}w\|_{L^2}^2\\
&\quad+\eta_{2}\|\nabla\dot{\Delta}_{j}a\|_{L^2}^2+\eta_{2}\|\dot{\Delta}_{j}b\|_{L^2}^2\\
&\gtrsim \|\dot{\Delta}_{j}(a,\nabla a, u,b,w)\|_{L^2}^2,
\end{aligned}
\end{equation}
where in the last inequality one has used
$$
\|\dot{\Delta}_{j}u\|_{L^2}^2+\|\dot{\Delta}_{j}(u-w)\|_{L^2}^2\geq \frac{1}{2}\|\dot{\Delta}_{j}w\|_{L^2}^2.
$$
 Combining (\ref{Highinequality}) and (\ref{Highsim11})-(\ref{Highsim22}) together, we show
\begin{equation}\label{Highinequality2}
\begin{aligned}
&\frac{d}{dt}\mathcal{E}_{2,j}(t)+\mathcal{E}_{2,j}(t)\\
&\lesssim \big{(}2^{j}\|\dot{\Delta}_{j}(au) \|_{L^2}+\|\dot{\Delta}_{j}(u\cdot \nabla u,w\cdot\nabla b,w\cdot\nabla w)\|_{L^2}+\|\dot{\Delta}_{j}G\|_{L^2}+\|\div u\|_{L^{\infty}}\|\nabla\dot{\Delta}_{j}a\|_{L^2}\\
&\quad+\|\nabla\dot{\Delta}_{j}(a\div u)\|_{L^2} +\|[u\cdot \nabla ,\dot{\Delta}_{j}]u\|_{L^2}+\sum_{k=1}^{d}\|[u\cdot \nabla ,\partial_{k} \dot{\Delta}_{j}]a\|_{L^2}\big{)}\sqrt{ \mathcal{E}_{2,j}(t)},\quad\quad j\geq-1.
\end{aligned}
\end{equation}
By similar arguments as in (\ref{Lowinequality3})-(\ref{Lowinequality4}), the inequality \eqref{Highinequality2} implies
\begin{equation}\label{Highinequality3}
\begin{aligned}
&\|(\nabla a,u,b,w)\|_{\widetilde{L}^{\infty}_{t}(\dot{B}^{\frac{d}{2}-1}_{2,1})}^{h}+\|(\nabla a,u,b,w)\|_{L^1_{t}(\dot{B}^{\frac{d}{2}-1}_{2,1})}^{h}\\
&\quad\lesssim \|(\nabla a_{0},b_{0},u_{0},w_{0})\|_{\dot{B}^{\frac{d}{2}-1}_{2,1}}^{h}\\
&\quad\quad+\|au\|_{L^1_{t}(\dot{B}^{\frac{d}{2}}_{2,1})}^{h}+\|(u\cdot \nabla u,w\cdot\nabla b,w\cdot\nabla w)\|_{L^1_{t}(\dot{B}^{\frac{d}{2}-1}_{2,1})}^{h}\\
&\quad\quad+\|G\|_{L^1_{t}(\dot{B}^{\frac{d}{2}-1}_{2,1})}^{h}+\|\div u\|_{L^1_{t}(L^{\infty})}\|a\|_{\widetilde{L}^{\infty}_{t}(\dot{B}^{\frac{d}{2}}_{2,1})}^{h}+\|a\div u\|_{L_{t}^1(\dot{B}^{\frac{d}{2}}_{2,1})}^{h}\\
&\quad\quad+\sum_{j\geq-1}2^{j(\frac{d}{2}-1)}\big{(}\|[u\cdot \nabla , \dot{\Delta}_{j}]u\|_{L^1_{t}(L^2)}+\sum_{k=1}^{d}\|[u\cdot \nabla ,\partial_{k}\dot{\Delta}_{j}]a\|_{L^1_{t}(L^2)}\big{)}.
\end{aligned}
\end{equation}
To estimate the right-hand side of \eqref{Highinequality3}, it holds by $(\ref{nonlinearsim})$ that
\begin{equation}\label{Highinequality4}
\begin{aligned}
&\|a\div u\|_{L_{t}^1(\dot{B}^{\frac{d}{2}}_{2,1})}+\|\div u\|_{L^1_{t}(L^{\infty})}\|a\|_{\widetilde{L}^{\infty}_{t}(\dot{B}^{\frac{d}{2}}_{2,1})}^{h}\\
&\quad\lesssim \|a\|_{\widetilde{L}^{\infty}_{t}(\dot{B}^{\frac{d}{2}}_{2,1})}\|u\|_{L^1_{t}(\dot{B}^{\frac{d}{2}+1}_{2,1})}\lesssim \mathcal{X}^2(t).
\end{aligned}
\end{equation}
Making use of the commutator estimates $(\ref{commutator})$-$(\ref{commutator1})$, we also have
\begin{equation}\label{Highinequality5}
\left\{
\begin{aligned}
&\sum_{j\in\mathbb{Z}}2^{j(\frac{d}{2}-1)}\|[u\cdot \nabla , \dot{\Delta}_{j}]u\|_{L^1_{t}(L^2)}+\sum_{j\in\mathbb{Z}}2^{j(\frac{d}{2}-1)}\sum_{k=1}^{d}\|[u\cdot \nabla , \partial_{k}\dot{\Delta}_{j}]a\|_{L^1_{t}(L^2)}\\
&\quad\quad\quad\quad\lesssim \|u\|_{L^1_{t}(\dot{B}^{\frac{d}{2}+1}_{2,1})}\|(a,u)\|_{\widetilde{L}^{\infty}_{t}(\dot{B}^{\frac{d}{2}-1}_{2,1})}\lesssim \mathcal{X}^2(t).
\end{aligned}
\right.
\end{equation}
By (\ref{Lownonlinear1})-(\ref{Lownonlinear4}) and (\ref{Highinequality3})-(\ref{Highinequality5}), there holds
\begin{equation}\label{High1}
\begin{aligned}
&\|a\|_{\widetilde{L}^{\infty}_{t}(\dot{B}^{\frac{d}{2}}_{2,1})}^{h}+\|(u,b,w)\|_{\widetilde{L}^{\infty}_{t}(\dot{B}^{\frac{d}{2}-1}_{2,1})}^{h}+\|a\|_{L^1_{t}(\dot{B}^{\frac{d}{2}}_{2,1})}^{h}+\|(u,b,w)\|_{L^1_{t}(\dot{B}^{\frac{d}{2}-1}_{2,1})}^{h}\\
&\quad\lesssim\|a_{0}\|_{\dot{B}^{\frac{d}{2}}_{2,1}}^{h}+\|(u_{0},b_{0},w_{0})\|_{\dot{B}^{\frac{d}{2}-1}_{2,1}}^{h}+\mathcal{X}^2(t).
\end{aligned}
\end{equation}
Then applying Lemma \ref{heat} to the heat equation
\begin{equation}\label{ENS2}
\begin{aligned}
\partial_{t}u-  \Delta u=-\nabla a+ w-u-u\cdot\nabla u+G,
\end{aligned}
\end{equation}
 we gain
\begin{equation}\label{High2}
\begin{aligned}
&\|u\|_{\widetilde{L}^{\infty}_{t}(\dot{B}^{\frac{d}{2}-1}_{2,1})}^{h}+\|u\|_{L^1_{t}(\dot{B}^{\frac{d}{2}+1}_{2,1})}^{h}\\
&\quad\lesssim \|u_{0}\|_{\dot{B}^{\frac{d}{2}-1}_{2,1}}^{h}+\|a\|_{L^1_{t}(\dot{B}^{\frac{d}{2}}_{2,1})}^{h}+\|(u,w)\|_{L^1_{t}(\dot{B}^{\frac{d}{2}-1}_{2,1})}^{h}\\
&\quad\quad+\|u\cdot\nabla u\|_{L^1_{t}(\dot{B}^{\frac{d}{2}-1}_{2,1})}^{h}+\|G\|_{L^1_{t}(\dot{B}^{\frac{d}{2}-1}_{2,1})}^{h}\\
&\quad\lesssim \|a_{0}\|_{\dot{B}^{\frac{d}{2}}_{2,1}}^{h}+\|(u_{0},b_{0},w_{0})\|_{\dot{B}^{\frac{d}{2}-1}_{2,1}}^{h}+\mathcal{X}^2(t),
\end{aligned}
\end{equation}
where in the last inequality one has used (\ref{Lownonlinear2})-(\ref{Lownonlinear4}) and (\ref{High1}).

Next, we are going to obtain the $\widetilde{L}^{\infty}_{t}(\dot{B}^{\frac{d}{2}+1}_{2,1})\cap L^1_{t}(\dot{B}^{\frac{d}{2}+1}_{2,1})$-estimates of $(b,w)$. Let $\mathcal{E}_{3,j}(t)$ and $\mathcal{D}_{3,j}(t)$ be defined by \eqref{EHH}. For any $j\geq-1$, it is easy to verify that
\begin{equation}\label{eulerinequality1}
\begin{aligned}
&(\frac{1}{2}-2\eta_{3})\|\dot{\Delta}_{j}(b,w)\|_{L^2}^2\leq \mathcal{E}_{3,j}(t)\leq (\frac{1}{2}+2\eta_{3})\|\dot{\Delta}_{j}(b,w)\|_{L^2}^2,
\end{aligned}
\end{equation}
and
\begin{equation}\label{eulerinequality2}
\begin{aligned}
&\mathcal{D}_{3,j}(t)\geq (1-C\eta_{3})\|\dot{\Delta}_{j}w\|_{L^2}^2+\frac{9}{32}\eta_{3}\|\dot{\Delta}_{j}b\|_{L^2}^2-4\eta_{3}\|\dot{\Delta}_{j}u\|_{L^2}\|\dot{\Delta}_{j}b\|_{L^2},
\end{aligned}
\end{equation}
where $C>0$ is a constant independent of time. Choosing a suitably small constant $\eta_{3}\in(0,1)$, for any $j\geq-1$, we get by \eqref{dEulerpp} and (\ref{eulerinequality1})-$(\ref{eulerinequality2})$ that
\begin{equation}\label{dEuler}
\begin{aligned}
&\frac{d}{dt}\mathcal{E}_{3,j}(t)+\mathcal{E}_{3,j}(t)\\
&\lesssim \big( \|\dot{\Delta}_{j}u\|_{L^2}+\|\div w\|_{L^{\infty}}\|\dot{\Delta}_{j}(b,w)\|_{L^2}\\
&\quad\quad+2^{-j}\|\dot{\Delta}_{j}(w\cdot\nabla b,w\cdot\nabla w)\|_{L^2}+\|[w\cdot \nabla,\dot{\Delta}_{j}](b,w)\|_{L^2}\big{)}\sqrt{\mathcal{E}_{3,j}(t)}.
\end{aligned}
\end{equation}
With the help of (\ref{nonlinearsim}), (\ref{High2}), \eqref{dEuler}, $(\ref{commutator})$ and
\begin{equation}\nonumber
\begin{aligned}
&\|(w\cdot\nabla b,w\cdot\nabla w)\|_{L^1_{t}(\dot{B}^{\frac{d}{2}}_{2,1})}\lesssim \|w\|_{\widetilde{L}^{\infty}_{t}(\dot{B}^{\frac{d}{2}}_{2,1})}\|(w,b)\|_{L^1_{t}(\dot{B}^{\frac{d}{2}+1}_{2,1})}\lesssim \mathcal{X}^2(t),
\end{aligned}
\end{equation}
it holds
\begin{equation}\label{High3}
\begin{aligned}
&\|(b,w)\|_{\widetilde{L}^{\infty}_{t}(\dot{B}^{\frac{d}{2}+1}_{2,1})}^{h}+\|(b,w)\|_{L^1_{t}(\dot{B}^{\frac{d}{2}+1}_{2,1})}^{h}\\
&\quad\lesssim \|(b_{0},w_{0})\|_{\dot{B}^{\frac{d}{2}+1}_{2,1}}^{h}+ \|u\|_{L^1_{t}(\dot{B}^{\frac{d}{2}+1}_{2,1})}^{h}+\|\div w\|_{L^{\infty}_{t}(L^{\infty})}\|(b,w)\|_{L^1_{t}(\dot{B}^{\frac{d}{2}+1}_{2,1})}^{h}\\
&\quad\quad+\|(w\cdot\nabla b,w\cdot\nabla w)\|_{L^1_{t}(\dot{B}^{\frac{d}{2}}_{2,1})}^{h}+\sum_{j\geq-1}2^{j(\frac{d}{2}+1)}\|[w\cdot \nabla,\dot{\Delta}_{j}](b,w)\|_{L^2}\\
&\quad\lesssim \|a_{0}\|_{\dot{B}^{\frac{d}{2}}_{2,1}}^{h}+\|u_{0}\|_{\dot{B}^{\frac{d}{2}-1}_{2,1}}^{h}+\|(b_{0},w_{0})\|_{\dot{B}^{\frac{d}{2}+1}_{2,1}}^{h}+\mathcal{X}^2(t).
\end{aligned}
\end{equation}
The combination of (\ref{High1}), (\ref{High2}) and (\ref{High3}) leads to (\ref{High}). The proof of Lemma \ref{prop42} is completed.
\end{proof}

\section{Optimal time-decay rates}\label{sectionlarge}

\subsection{The proof of Theorem \ref{theorem13}}
%$\dot{B}^{\sigma_{0}}_{2,\infty}$-boundedness assumption on the initial data
In this subsection, we show Theorem \ref{theorem13} on the optimal time-decay rates of the strong solution to the Cauchy problem $(\ref{m1n})$ in the case that $\|(a_{0},u_{0},b_{0},w_{0})^{\ell}\|_{\dot{B}^{\sigma_{0}}_{2,\infty}}$ is bounded.

First, in addition to \eqref{XX0}, we have the following low-frequency estimates.
\begin{lemma}\label{prop51}
Let $(a,u,b,w)$ be the global solution to the Cauchy problem $(\ref{m1n})$ given by Theorem \ref{theorem12}. Then, under the assumptions of Theorem \ref{theorem13}, the following inequality holds:
\begin{equation}
\begin{aligned}
\mathcal{X}_{L,\sigma_{0}}(t):&=\|(a,u,b,w)\|_{\widetilde{L}^{\infty}_{t}(\dot{B}^{\sigma_{0}}_{2,\infty})}^{\ell}+\|(a,u,b,w)\|_{\widetilde{L}^1_{t}(\dot{B}^{\sigma_{0}+2}_{2,\infty})}^{\ell}\\
&\quad+\|u-w\|_{\widetilde{L}^1_{t}(\dot{B}^{\sigma_{0}+1}_{2,\infty})}^{\ell}+\|u-w\|_{\widetilde{L}^2_{t}(\dot{B}^{\sigma_{0}}_{2,\infty})}^{\ell}\leq C\delta_{0},\quad\quad t>0,\label{lowd}
\end{aligned}
\end{equation}
where $\delta_{0}$ is defined by \eqref{D0}, and $C>0$ is a constant independent of time.
\end{lemma}

\begin{proof}
Multiplying (\ref{Lowinequality3}) by $2^{\sigma_{0}j}$ and taking the supremum on both $[0,t]$ and $j\leq0$, we get
\begin{equation}\label{lowd1}
\begin{aligned}
&\|(a,u,b,w)\|_{\widetilde{L}^{\infty}_{t}(\dot{B}_{2,\infty}^{\sigma_{0}})}^{\ell}+\|(a,u,b,w)\|_{\widetilde{L}^1_{t}(\dot{B}_{2,\infty}^{\sigma_{0}+2})}^{\ell}\\
&\quad\lesssim \|(a_{0},u_{0},b_{0},w_{0})\|_{\dot{B}_{2,\infty}^{\sigma_{0}}}^{\ell}+\|au\|_{\widetilde{L}^1_{t}(\dot{B}^{\sigma_{0}+1}_{2,\infty})}^{\ell}+\|(u\cdot\nabla u,w\cdot\nabla b,w\cdot\nabla w)\|_{\widetilde{L}^1_{t}(\dot{B}^{\sigma_{0}}_{2,\infty})}^{\ell}+\|G\|_{\widetilde{L}^1_{t}(\dot{B}^{\sigma_{0}}_{2,\infty})}^{\ell}.
\end{aligned}
\end{equation}
Arguing similarly as in Lemma \ref{prop41}, we deduce by (\ref{simsim2}), (\ref{lh}), (\ref{uv3}) and Lemma \ref{lemma24} that
\begin{equation}\label{nonlinearsigma01}
\begin{aligned}
&\|(u\cdot\nabla u,w\cdot\nabla b,w\cdot\nabla w)\|_{\widetilde{L}^1_{t}(\dot{B}^{\sigma_{0}}_{2,\infty})}^{\ell}\\
&\quad\lesssim \|(u,w)\|_{\widetilde{L}^2_{t}(\dot{B}^{\frac{d}{2}}_{2,1})}\|(u,b,w)\|_{\widetilde{L}^2_{t}(\dot{B}^{\sigma_{0}+1}_{2,\infty})}\lesssim\mathcal{X}(t)\big{(}\mathcal{X}_{L,\sigma_{0}}(t)+\mathcal{X}(t)\big{)}.
\end{aligned}
\end{equation}
and
\begin{equation}\label{nonlinearsigma02}
\begin{aligned}
&\|G\|_{\widetilde{L}^1_{t}(\dot{B}^{\sigma_{0}}_{2,\infty})}\lesssim \|a\|_{\widetilde{L}^2_{t}(\dot{B}^{\frac{d}{2}}_{2,1})}\|a\|_{\widetilde{L}^2_{t}(\dot{B}^{\sigma_{0}+1}_{2,\infty})}+\|a\|_{\widetilde{L}^{\infty}_{t}(\dot{B}^{\frac{d}{2}}_{2,1})}\|u\|_{\widetilde{L}^1_{t}(\dot{B}^{\sigma_{0}+2}_{2,\infty})}\\
&\quad\quad\quad\quad\quad\quad+\big{(}\|a\|_{\widetilde{L}^{2}_{t}(\dot{B}^{\frac{d}{2}}_{2,1})}\|b\|_{\widetilde{L}^{2}_{t}(\dot{B}^{\frac{d}{2}}_{2,1})}+\|a\|_{\widetilde{L}^{2}_{t}(\dot{B}^{\frac{d}{2}}_{2,1})}+\|b\|_{\widetilde{L}^{2}_{t}(\dot{B}^{\frac{d}{2}}_{2,1})}\big{)}\|u-w\|_{\widetilde{L}^{2}_{t}(\dot{B}^{\sigma_{0}}_{2,\infty})}\\
&\quad\quad\quad\quad\quad~\lesssim\mathcal{X}(t)\big{(}\mathcal{X}_{L,\sigma_{0}}(t)+\mathcal{X}(t)\big{)}.
\end{aligned}
\end{equation}
Substituting the estimates (\ref{nonlinearsigma01})-(\ref{nonlinearsigma02}) into (\ref{lowd1}), we have
\begin{equation}\label{lowd2}
\begin{aligned}
&\|(a,u,b,w)\|_{\widetilde{L}^{\infty}_{t}(\dot{B}_{2,\infty}^{\sigma_{0}})}^{\ell}+\|(a,u,b,w)\|_{\widetilde{L}^1_{t}(\dot{B}_{2,\infty}^{\sigma_{0}+2})}^{\ell}\lesssim\|(a_{0},u_{0},b_{0},w_{0})\|_{\dot{B}_{2,\infty}^{\sigma_{0}}}^{\ell}+ \mathcal{X}(t)\big{(}\mathcal{X}_{L,\sigma_{0}}(t)+\mathcal{X}(t)\big{)}.
\end{aligned}
\end{equation}
Similarly to \eqref{Low2p}-\eqref{Low3}, one can have
\begin{equation}\label{lowd4}
\begin{aligned}
&\|u-w\|_{\widetilde{L}^1_{t}(\dot{B}^{\sigma_{0}+1}_{2,\infty})}^{\ell}+\|u-w\|_{\widetilde{L}^{2}_{t}(\dot{B}^{\sigma_{0}}_{2,\infty})}^{\ell}\\
&\quad\lesssim \|(u_{0},w_{0})\|_{\dot{B}^{\sigma_{0}}_{2,\infty}}^{\ell}+\|(a,u,b)\|_{\widetilde{L}^1_{t}(\dot{B}^{\sigma_{0}+2}_{2,\infty})}^{\ell}+\|(u\cdot\nabla u,w\cdot\nabla w)\|_{\widetilde{L}^1_{t}(\dot{B}^{\sigma_{0}+2}_{2,\infty})}^{\ell}+\|G\|_{\widetilde{L}^1_{t}(\dot{B}^{\sigma_{0}+2}_{2,\infty})}^{\ell}\\
&\lesssim  \|(a_{0},u_{0},b_{0},w_{0})\|_{\dot{B}^{\sigma_{0}}_{2,\infty}}^{\ell}+\mathcal{X}(t)\big{(}\mathcal{X}_{L,\sigma_{0}}(t)+\mathcal{X}(t)\big{)}.
\end{aligned}
\end{equation}
Thus, it follows by (\ref{nonlinearsigma01})-(\ref{lowd4}) that
\begin{equation}\nonumber
\begin{aligned}
&\mathcal{X}_{L,\sigma_{0}}(t)\lesssim \|(a_{0},u_{0},b_{0},w_{0})\|_{\dot{B}^{\sigma_{0}}_{2,\infty}}^{\ell}+\mathcal{X}(t)\big{(}\mathcal{X}_{L,\sigma_{0}}(t)+\mathcal{X}(t)\big{)}.
\end{aligned}
\end{equation}
Making use of (\ref{XX0}),  $\mathcal{X}(t)\lesssim\mathcal{X}_{0}<<1$ and $\|(a_{0},u_{0},b_{0},w_{0})\|_{\dot{B}^{\sigma_{0}}_{2,\infty}}^{\ell}+\mathcal{X}_{0}\sim \delta_{0}$, we prove (\ref{lowd}). The proof of Proposition \ref{prop51} is completed.
\end{proof}

\vspace{1mm}

Next, we introduce a new time-weighted energy functional
\begin{equation}\label{mathcalXtheta}
\begin{aligned}
\mathcal{X}_{\theta}(t)&:=\|\tau^{\theta}(a,u,b,w)\|_{\widetilde{L}^{\infty}_{t}(\dot{B}^{\frac{d}{2}-1}_{2,1})}^{\ell}+ \|\tau^{\theta}(a,u,b,w)\|_{L^1_{t}(\dot{B}^{\frac{d}{2}+1}_{2,1})}^{\ell}\\
&\quad+\|\tau^{\theta}(u-w)\|_{L^1_{t}(\dot{B}^{\frac{d}{2}}_{2,1})}^{\ell}+ \|\tau^{\theta}(u-w)\|_{\widetilde{L}^2_{t}(\dot{B}^{\frac{d}{2}-1}_{2,1})}^{\ell}\\
&\quad+\|\tau^{\theta}(\nabla a,u)\|_{\widetilde{L}^{\infty}_{t}(\dot{B}^{\frac{d}{2}-1}_{2,1})}^{h}+\|\tau^{\theta}(b,w)\|_{\widetilde{L}^{\infty}_{t}(\dot{B}^{\frac{d}{2}+1}_{2,1})}^{h}+\|\tau^{\theta}a\|_{L^1_{t}(\dot{B}^{\frac{d}{2}}_{2,1})}^{h}+\|\tau^{\theta}(u,b,w)\|_{L^1_{t}(\dot{B}^{\frac{d}{2}+1}_{2,1})}^{h}.
\end{aligned}
\end{equation}
We have the time-weighted estimates of the solution $(a,u,b,w)$ to the Cauchy problem \eqref{m1n} for low frequencies.
\begin{lemma}\label{lemma51}
Let $(a,u,b,w)$ be the global solution to the Cauchy problem $(\ref{m1n})$ given by Theorem \ref{theorem12}. Then, under the assumptions of Theorem \ref{theorem13}, for $-\frac{d}{2}\leq\sigma_{0}<\frac{d}{2}-1$ and $\theta>\frac{1}{2}(\frac{d}{2}+1-\sigma_{0})$, it holds
\begin{equation}
\begin{aligned}
&\|\tau^{\theta}(a,u,b,w)\|_{\widetilde{L}^{\infty}_{t}(\dot{B}^{\frac{d}{2}-1}_{2,1})}^{\ell}+ \|\tau^{\theta}(a,u,b,w)\|_{L^1_{t}(\dot{B}^{\frac{d}{2}+1}_{2,1})}^{\ell}+\|\tau^{\theta}(u-w)\|_{L^1_{t}(\dot{B}^{\frac{d}{2}}_{2,1})}^{\ell}+ \|\tau^{\theta}(u-w)\|_{\widetilde{L}^2_{t}(\dot{B}^{\frac{d}{2}-1}_{2,1})}^{\ell}\\
&\quad\lesssim \frac{\mathcal{X}(t)+\mathcal{X}_{L,\sigma_{0}}(t)}{\zeta}t^{\theta-\frac{1}{2}(\frac{d}{2}-1-\sigma_{0})}+  \big{(}\zeta+ \mathcal{X}(t)\big{)}\mathcal{X}_{\theta}(t),\quad t>0,\label{decaylow}
\end{aligned}
\end{equation}
where $\mathcal{X}(t)$, $\mathcal{X}_{L,\sigma_{0}}(t)$ and $\mathcal{X}_{\theta}(t)$ are defined by $(\ref{XX00})$, $(\ref{lowd})$ and $(\ref{mathcalXtheta})$, respectively, and $\zeta>0$ is a constant to be determined later.
\end{lemma}

\begin{proof}
We recall that $\mathcal{E}_{1,j}(t)$ given by \eqref{EL} satisfies the Lyapunov type inequality \eqref{Lowinequality1}. Multiplying (\ref{Lowinequality1}) by $t^{\theta}$ and using the fact $t^{\theta}\frac{d}{dt}\mathcal{E}_{1,j}(t)$=$\frac{d}{dt}\big{(}t^{\theta}\mathcal{E}_{1,j}(t)\big{)}-\theta t^{\theta-1}\mathcal{E}_{1,j}(t)$, we obtain
\begin{equation}\nonumber
\begin{aligned}
&\frac{d}{dt}\big{(}t^{\theta}\mathcal{E}_{1,j}(t)\big{)}+t^{\theta}2^{2j}\mathcal{E}_{1,j}(t)\\
&\quad\lesssim t^{\theta-1}\mathcal{E}_{1,j}(t)+t^{\theta} \big{(}2^{j}\|\dot{\Delta}_{j}(au) \|_{L^2}+\|\dot{\Delta}_{j}(u\cdot \nabla u,w\cdot\nabla b,w\cdot\nabla w)\|_{L^2}+\|\dot{\Delta}_{j}G\|_{L^2} \big{)}\sqrt{\mathcal{E}_{1,j}(t)},\quad j\leq0,
\end{aligned}
\end{equation}
which together with (\ref{ELsim1})-(\ref{DLsim1}) and $t^{\theta-1}\sqrt{\mathcal{E}_{1,j}(t)}\Big{|}_{t=0}=0$ yields for any $j\leq0$ that
\begin{equation}\label{decaylow11}
\begin{aligned}
&t^{\theta}\|\dot{\Delta}_{j}(a,u,b,w)\|_{L^2}+2^{2j}\int_{0}^{t}\tau^{\theta}\|\dot{\Delta}_{j}(a,u,b,w)\|_{L^2}d\tau\\
&\quad\lesssim \int_{0}^{t}\tau^{\theta-1}\|\dot{\Delta}_{j}(a,u,b,w)\|_{L^2}d\tau\\
&\quad\quad+\int_{0}^{t}\tau^{\theta}\big{(}2^{j}\|\dot{\Delta}_{j}(au)\|_{L^2}+\|\dot{\Delta}_{j}(u\cdot\nabla u,w\cdot\nabla b,w\cdot\nabla w)\|_{L^2}+\|\dot{\Delta}_{j}G\|_{L^2}\big{)}d\tau.
\end{aligned}
\end{equation}
Then we multiply (\ref{decaylow11}) by $2^{j(\frac{d}{2}-1)}$, take the supremum on $[0,t]$ and then sum over $j\leq 0$ to have
\begin{equation}\label{decaylow1}
\begin{aligned}
&\|\tau^{\theta}(a,u,b,w)\|_{\widetilde{L}^{\infty}_{t}(\dot{B}^{\frac{d}{2}-1}_{2,1})}^{\ell}+ \|\tau^{\theta}(a,u,b,w)\|_{L^1_{t}(\dot{B}^{\frac{d}{2}+1}_{2,1})}^{\ell}\\
&\quad\lesssim\int_{0}^{t}\tau^{\theta-1}\|(a,u,b,w)\|_{\dot{B}^{\frac{d}{2}-1}_{2,1}}^{\ell}d\tau+\|\tau^{\theta}au\|_{L^1_{t}(\dot{B}^{\frac{d}{2}}_{2,1})}^{\ell}\\
&\quad\quad+\|\tau^{\theta}(u\cdot \nabla u,w\cdot\nabla b,w\cdot\nabla w)\|_{L^1_{t}(\dot{B}^{\frac{d}{2}-1}_{2,1})}^{\ell}+\|\tau^{\theta}G\|_{L^1_{t}(\dot{B}^{\frac{d}{2}-1}_{2,1})}^{\ell}.
\end{aligned}
\end{equation}
To control the first term on the right-hand side of (\ref{decaylow1}), we deduce from (\ref{lh})-(\ref{inter}) that
\begin{equation}\label{interL1}
\begin{aligned}
&\int_{0}^{t}\tau^{\theta-1}\|(a,u,b,w)^{\ell}\|_{\dot{B}^{\frac{d}{2}-1}_{2,1}}d\tau\\
&\quad\lesssim \int_{0}^{t}\tau^{\theta-1}\|(a,u,b,w)^{\ell}\|_{\dot{B}^{\sigma_{0}}_{2,\infty}}^{1-\eta_{0}}\|(a,u,b,w)^{\ell}\|_{\dot{B}^{\frac{d}{2}+1}_{2,\infty}}^{\eta_{0}}d\tau\\
&\quad\lesssim \Big{(}\int_{0}^{t}\tau^{\theta-\frac{1}{1-\eta_{0}}}d\tau\Big{)}^{1-\eta_{0}}\|(a,u,b,w)^{\ell}\|_{\widetilde{L}^{\infty}_{t}(\dot{B}^{\sigma_{0}}_{2,\infty})}^{1-\eta_{0}}\|\tau^{\theta}(a,u,b,w)^{\ell}\|_{L^1_{t}(\dot{B}^{\frac{d}{2}+1}_{2,\infty})}^{\eta_{0}}\\
&\quad\lesssim \Big{(}t^{(\theta-\theta_{0})}\|(a,u,b,w)\|_{\widetilde{L}^{\infty}_{t}(\dot{B}^{\sigma_{0}}_{2,\infty})}^{\ell}\Big{)}^{(1-\eta_{0})}\Big{(}\|\tau^{\theta}(a,u,b,w)\|_{L^1_{t}(\dot{B}^{\frac{d}{2}+1}_{2,1})}^{\ell}\Big{)}^{\eta_{0}},
\end{aligned}
\end{equation}
for the constant $\eta_{0}\in (0,1)$ given by
\begin{equation}\label{eta3}
\begin{aligned}
\frac{d}{2}-1=\eta_{0}(\frac{d}{2}+1)+\sigma_{0}(1-\eta_{0}).
\end{aligned}
\end{equation}
Taking the advantage of (\ref{lh}) and the dissipative properties of $(a,u,b,w)$ for high frequencies, it is easy to verify that
\begin{equation}\label{interL2}
\begin{aligned}
&\int_{0}^{t}\tau^{\theta-1}\|a^{h}\|_{\dot{B}^{\frac{d}{2}-1}_{2,1}}d\tau\\
&\quad\lesssim \Big{(}\int_{0}^{t}\tau^{(\theta-1-\theta\eta_{0})\frac{1}{1-\eta_{0}}}d\tau\Big{)}^{1-\eta_{0}}\Big{(}\|a\|_{\widetilde{L}^{\infty}_{t}(\dot{B}^{\frac{d}{2}-1}_{2,1})}^{h}\Big{)}^{1-\eta_{0}}\Big{(}\|\tau^{\theta}a\|_{L^1_{t}(\dot{B}^{\frac{d}{2}-1}_{2,1})}^{h}\Big{)}^{\eta_{0}}\\
&\quad\lesssim \Big{(}t^{(\theta-\theta_{0})}\|a\|_{\widetilde{L}^{\infty}_{t}(\dot{B}^{\frac{d}{2}}_{2,1})}^{h}\Big{)}^{1-\eta_{0}}\Big{(}\|\tau^{\theta}a\|_{L^1_{t}(\dot{B}^{\frac{d}{2}}_{2,1})}^{h}\Big{)}^{\eta_{0}},
\end{aligned}
\end{equation}
and
\begin{equation}\label{interL3}
\begin{aligned}
&\int_{0}^{t}\tau^{\theta-1}\|u^{h}\|_{\dot{B}^{\frac{d}{2}-1}_{2,1}}d\tau\\
&\quad\lesssim\Big{(}\int_{0}^{t}\tau^{(\theta-1-\theta\eta_{0})\frac{1}{1-\eta_{0}}}d\tau\Big{)}^{1-\eta_{0}}\Big{(}\|u\|_{\widetilde{L}^{\infty}_{t}(\dot{B}^{\frac{d}{2}-1}_{2,1})}^{h}\Big{)}^{1-\eta_{0}}\Big{(}\|\tau^{\theta}u\|_{\widetilde{L}^1_{t}(\dot{B}^{\frac{d}{2}-1}_{2,\infty})}^{h}\Big{)}^{\eta_{0}}\\
&\quad\lesssim \Big{(}t^{(\theta-\theta_{0})}\|u\|_{\widetilde{L}^{\infty}_{t}(\dot{B}^{\frac{d}{2}-1}_{2,1})}^{h}\Big{)}^{1-\eta_{0}}\Big{(}\|\tau^{\theta}u\|_{L^1_{t}(\dot{B}^{\frac{d}{2}+1}_{2,1})}^{h}\Big{)}^{\eta_{0}}.
\end{aligned}
\end{equation}
Similarly, one has
\begin{equation}\label{interL4}
\begin{aligned}
&\int_{0}^{t}\tau^{\theta-1}\|(b,w)^h\|_{\dot{B}^{\frac{d}{2}-1}_{2,1}}d\tau\\
&\quad\lesssim \Big{(}t^{(\theta-\theta_{0})}\|(b,w)\|_{\widetilde{L}^{\infty}_{t}(\dot{B}^{\frac{d}{2}+1}_{2,1})}^{h}\Big{)}^{1-\eta_{0}}\Big{(}\|\tau^{\theta}(b,w)\|_{L^1_{t}(\dot{B}^{\frac{d}{2}+1}_{2,1})}^{h}\Big{)}^{\eta_{0}}.
\end{aligned}
\end{equation}
By (\ref{interL1})-(\ref{interL4}) and Young's inequality, we have for any constant $\zeta>0$ that
\begin{equation}\label{interL}
\begin{aligned}
&\int_{0}^{t}\tau^{\theta-1}\|(a,u,b,w)\|_{\dot{B}^{\frac{d}{2}-1}_{2,1}}^{\ell}d\tau\\
&\quad\lesssim \int_{0}^{t}\tau^{\theta-1}\big{(}\|(a,u,b,w)^{\ell}\|_{\dot{B}^{\frac{d}{2}-1}_{2,1}}+\|(a,u,b,w)^h\|_{\dot{B}^{\frac{d}{2}-1}_{2,1}}\big{)}d\tau\\
&\quad\lesssim \frac{\mathcal{X}(t)+\mathcal{X}_{L,\sigma_{0}}(t)}{\zeta}t^{\theta-\frac{1}{2}(\frac{d}{2}-1-\sigma_{0})}+\zeta\mathcal{X}_{\theta}(t).
\end{aligned}
\end{equation}
By similar arguments as used in Lemma \ref{prop41}, the nonlinearities on the right-hand side of (\ref{decaylow1}) can be estimated by
\begin{equation}
\begin{aligned}
&\|\tau^{\theta}au\|_{L^1_{t}(\dot{B}^{\frac{d}{2}}_{2,1})}+\|\tau^{\theta}(u\cdot \nabla u,w\cdot\nabla b,w\cdot\nabla w)\|_{L^1_{t}(\dot{B}^{\frac{d}{2}-1}_{2,1})}\\
&\quad\lesssim \|(a,u,w)\|_{\widetilde{L}^2_{t}(\dot{B}^{\frac{d}{2}}_{2,1})}\|\tau^{\theta}(u,b,w)\|_{\widetilde{L}^2_{t}(\dot{B}^{\frac{d}{2}}_{2,1})}\lesssim \mathcal{X}(t)\mathcal{X}_{\theta}(t) ,\label{Ltimenonlinear} 
\end{aligned}
\end{equation}
and
\begin{equation}
\begin{aligned}
&\|\tau^{\theta}G\|_{L^1_{t}(\dot{B}^{\frac{d}{2}-1}_{2,1})}\\
&\quad\lesssim \|a\|_{\widetilde{L}^2_{t}(\dot{B}^{\frac{d}{2}}_{2,1})}\|\tau^{\theta}a\|_{\widetilde{L}^2_{t}(\dot{B}^{\frac{d}{2}}_{2,1})}+\|a\|_{\widetilde{L}^{\infty}_{t}(\dot{B}^{\frac{d}{2}}_{2,1})}\|\tau^{\theta}u\|_{L^1_{t}(\dot{B}^{\frac{d}{2}+1}_{2,1})}\\
&\quad\quad+\|\tau^{\theta}(a,b)\|_{\widetilde{L}^{2}_{t}(\dot{B}^{\frac{d}{2}}_{2,1})}\|u-w\|_{\widetilde{L}^{2}_{t}(\dot{B}^{\frac{d}{2}-1}_{2,1})}\lesssim \mathcal{X}(t)\mathcal{X}_{\theta}(t).\label{Ltimenonlinear1}
\end{aligned}
\end{equation}
Substituting (\ref{interL})-(\ref{Ltimenonlinear1}) into (\ref{decaylow1}) and using Young's inequality, we get for any $\zeta>0$ that
\begin{equation}
\begin{aligned}
&\|\tau^{\theta}(a,u,b,w)\|_{\widetilde{L}^{\infty}_{t}(\dot{B}^{\frac{d}{2}-1}_{2,1})}^{\ell}+ \|\tau^{\theta}(a,u,b,w)\|_{L^1_{t}(\dot{B}^{\frac{d}{2}+1}_{2,1})}^{\ell}\\
&\quad\lesssim \frac{\mathcal{X}(t)+\mathcal{X}_{L,\sigma_{0}}(t)}{\zeta}t^{\theta-\frac{1}{2}(\frac{d}{2}-1-\sigma_{0})}+  \big{(}\zeta+ \mathcal{X}(t)\big{)}\mathcal{X}_{\theta}(t).\label{decaylowff}
\end{aligned}
\end{equation}

In addition, we multiply (\ref{Low2p}) by $t^{2\theta}$ to have
\begin{equation}\label{Low2pt}
\begin{aligned}
&\frac{d}{dt}\big( t^{2\theta}\|\dot{\Delta}_{j}(u-w)\|_{L^2}^2\big)+t^{2\theta}\|\dot{\Delta}_{j}(u-w)\|_{L^2}^2\\
&\lesssim t^{\theta-1} \|\dot{\Delta}_{j}(u,w)\|_{L^2} t^{\theta} \|\dot{\Delta}_{j}(u-w)\|_{L^2}\\
&\quad+t^{\theta}\big( 2^{j}\|\dot{\Delta}_{j}(a,u,b)\|_{L^2}+\|\dot{\Delta}_{j}(u\cdot\nabla u,w\cdot \nabla w)\|_{L^2}+\|\dot{\Delta}_{j}G\|_{L^2}\big) t^{\theta}\|\dot{\Delta}_{j}(u-w)\|_{L^2},\quad j\leq0,
\end{aligned}
\end{equation}
Similarly to \eqref{low22pp}-\eqref{Low3}, one obtains by \eqref{interL} and \eqref{decaylowff}-\eqref{Low2pt} that
\begin{equation}\label{decaylowsss}
\begin{aligned}
&\|\tau^{\theta}(u-w)\|_{\widetilde{L}^2_{t}(\dot{B}^{\frac{d}{2}-1}_{2,1})}^{\ell}+\|\tau^{\theta}(u-w)\|_{L^1_{t}(\dot{B}^{\frac{d}{2}}_{2,1})}^{\ell}\\
&\quad\lesssim \int_{0}^{t}\tau^{\theta-1}\|(u,w)\|_{\dot{B}^{\frac{d}{2}-1}_{2,1}}^{\ell}d\tau\\
&\quad\quad+\|\tau^{\theta}(a,u,b,w)\|_{\widetilde{L}^{\infty}_{t}(\dot{B}^{\frac{d}{2}-1}_{2,1})}^{\ell}+\|\tau^{\theta}(a,u,b)\|_{\widetilde{L}^2_{t}(\dot{B}^{\frac{d}{2}}_{2,1})}^{\ell}\\
&\quad\quad+\|\tau^{\theta}(u\cdot \nabla u,w\cdot\nabla b,w\cdot\nabla w)\|_{L^1_{t}(\dot{B}^{\frac{d}{2}-1}_{2,1})}^{\ell}+\|\tau^{\theta}G\|_{L^1_{t}(\dot{B}^{\frac{d}{2}-1}_{2,1})}^{\ell}\\
&\quad\lesssim \frac{\mathcal{X}(t)+\mathcal{X}_{L,\sigma_{0}}(t)}{\zeta}t^{\theta-\frac{1}{2}(\frac{d}{2}-1-\sigma_{0})}+  \big{(}\zeta+ \mathcal{X}(t)\big{)}\mathcal{X}_{\theta}(t).
\end{aligned}
\end{equation}
By \eqref{decaylowff} and \eqref{decaylowsss}, we prove (\ref{decaylow}). The proof of Lemma \ref{lemma51} is completed.
\end{proof}

Then, we show the time-weighted estimates of the solution $(a,u,b,w)$ to the Cauchy problem \eqref{m1n} for high frequencies.
\begin{lemma}\label{lemma52}
Let $(a,u,b,w)$ be the global solution to the Cauchy problem $(\ref{m1n})$ given by Theorem \ref{theorem12}. Then, under the assumptions of Theorem \ref{theorem13}, for $-\frac{d}{2}\leq\sigma_{0}<\frac{d}{2}-1$ and $\theta>\frac{1}{2}(\frac{d}{2}+1-\sigma_{0})$, it holds
\begin{equation}
\begin{aligned}
&\|\tau^{\theta }a\|_{\widetilde{L}^{\infty}_{t}(\dot{B}^{\frac{d}{2}}_{2,1})}^{h}+\|\tau^{\theta}u\|_{\widetilde{L}^{\infty}_{t}(\dot{B}^{\frac{d}{2}-1}_{2,1})}^{h}+\|\tau^{\theta}(b,w)\|_{\widetilde{L}^{\infty}_{t}(\dot{B}^{\frac{d}{2}+1}_{2,1})}^{h}\\ &\quad\quad+\|\tau^{\theta}a\|_{L^1_{t}(\dot{B}^{\frac{d}{2}}_{2,1})}^{h}+\|\tau^{\theta}(u,b,w)\|_{L^1_{t}(\dot{B}^{\frac{d}{2}+1}_{2,1})}^{h}\\ &\quad\lesssim \frac{\mathcal{X}(t)}{\zeta}t^{\theta-\frac{1}{2}(\frac{d}{2}-1-\sigma_{0})}+ \big{(}\zeta+ \mathcal{X}(t)\big{)}\mathcal{X}_{\theta}(t),\quad t>0,\label{decayhigh}
\end{aligned}
\end{equation}
where $\zeta>0$ is a constant to be determined later, and $\mathcal{X}(t)$ and $\mathcal{X}_{\theta}(t)$ are defined by $(\ref{XX00})$ and $(\ref{mathcalXtheta})$, respectively.
\end{lemma}
\begin{proof}
Let $\mathcal{E}_{2,j}(t)$ be denoted by \eqref{EH}. For any $j\geq-1$, we show after multiplying the Lyapunov type inequality (\ref{Highinequality2}) by $t^{\theta}$ that
\begin{equation}\nonumber
\begin{aligned}
&\!\!\!\!\!\!\frac{d}{dt}\big{(}t^{\theta}\mathcal{E}_{2,j}(t)\big{)}+t^{\theta}\mathcal{E}_{2,j}(t)\\
&\!\!\!\!\!\!\lesssim t^{\theta-1}\mathcal{E}_{2,j}(t)+t^{\theta}\big{(}2^{j}\|\dot{\Delta}_{j}(au) \|_{L^2}+\|\dot{\Delta}_{j}(u\cdot \nabla u,w\cdot\nabla b,w\cdot\nabla w)\|_{L^2}+\|\dot{\Delta}_{j}G\|_{L^2}\\
&\!\!\!\!\quad+\|\div u\|_{L^{\infty}}\|\nabla\dot{\Delta}_{j}a\|_{L^2}+\|\nabla\dot{\Delta}_{j}(a\div u)\|_{L^2} +\|[u\cdot \nabla ,\dot{\Delta}_{j}]u\|_{L^2}+\sum_{k=1}^{d}\|[u\cdot \nabla ,\partial_{k} \dot{\Delta}_{j}]a\|_{L^2}\big{)}\sqrt{ \mathcal{E}_{2,j}(t)}.
\end{aligned}
\end{equation}
Thence performing direct computations on the above inequality, we have
\begin{equation}\label{decayhigh1}
\begin{aligned}
&\|\tau^{\theta}(\nabla a,u,b,w)\|_{\widetilde{L}^{\infty}_{t}(\dot{B}^{\frac{d}{2}-1}_{2,1})}^{h}+\|\tau^{\theta}(\nabla a,u,b,w)\|_{L^1_{t}(\dot{B}^{\frac{d}{2}-1}_{2,1})}^{h}\\
&\quad\lesssim \int_{0}^{t}\tau^{\theta-1}\|(\nabla a,u,b,w)\|_{\dot{B}^{\frac{d}{2}-1}_{2,1}}^{h}d\tau\\
&\quad\quad+\|\tau^{\theta}au\|_{L^1_{t}(\dot{B}^{\frac{d}{2}}_{2,1})}^{h}+\|\tau^{\theta}(u\cdot \nabla u,w\cdot\nabla b,w\cdot\nabla w)\|_{L^1_{t}(\dot{B}^{\frac{d}{2}-1}_{2,1})}^{h}\\
&\quad\quad+\|\tau^{\theta}G\|_{L^1_{t}(\dot{B}^{\frac{d}{2}-1}_{2,1})}^{h}+\|\div u\|_{L^1_{t}(L^{\infty})}\|\tau^{\theta}a\|_{\widetilde{L}^{\infty}_{t}(\dot{B}^{\frac{d}{2}}_{2,1})}^{h}+\|\tau^{\theta}a\div u\|_{L_{t}^1(\dot{B}^{\frac{d}{2}}_{2,1})}^{h}\\
&\quad\quad+\sum_{j\geq-1}2^{j(\frac{d}{2}-1)}\big{(}\|[u\cdot \nabla , \dot{\Delta}_{j}]\tau^{\theta}u\|_{L^1_{t}(L^2)}+\sum_{k=1}^{d}\|[u\cdot \nabla ,\partial_{k}\dot{\Delta}_{j}]\tau^{\theta}a\|_{L^1_{t}(L^2)}\big{)}.
\end{aligned}
\end{equation}
By similar arguments as used in (\ref{interL1})-(\ref{interL4}), we have for any constant $\zeta>0$ that
\begin{equation}\label{timehigh}
\begin{aligned}
&\int_{0}^{t}\tau^{\theta-1}\|(\nabla a,u,b,w)\|_{\dot{B}^{\frac{d}{2}-1}_{2,1}}^{h}d\tau\\
&\quad\lesssim \Big{(}t^{\theta-\frac{1}{2}(\frac{d}{2}-1-\sigma_{0})}\big{(}\| (\nabla a,u)\|_{\widetilde{L}^{\infty}_{t}(\dot{B}^{\frac{d}{2}-1}_{2,1})}^{h}+\|(b,w)\|_{\widetilde{L}^{\infty}_{t}(\dot{B}^{\frac{d}{2}+1}_{2,1})}^{h}\big{)}\Big{)}^{1-\eta_{0}}\\
&\quad\quad\times\Big{(}\|\tau^{\theta}a\|_{L^1_{t}(\dot{B}^{\frac{d}{2}}_{2,1})}^{h}+\|\tau^{\theta}(u,b,w)\|_{L^1_{t}(\dot{B}^{\frac{d}{2}+1}_{2,1})}^{h}\Big{)}^{\eta_{0}}\\
&\quad\lesssim \frac{\mathcal{X}(t)}{\zeta}t^{\theta-\frac{1}{2}(\frac{d}{2}-1-\sigma_{0})}+\zeta\mathcal{D}^{H}_{\theta}(t),
\end{aligned}
\end{equation}
for the constant $\eta_{0}\in(0,1)$ given by (\ref{eta3}). As in (\ref{Highinequality4})-(\ref{Highinequality5}), one can show
\begin{equation}\label{Highinequality411}
\left\{
\begin{aligned}
&\|\tau^{\theta}a\div u\|_{L_{t}^1(\dot{B}^{\frac{d}{2}}_{2,1})}\lesssim \|a\|_{\widetilde{L}^{\infty}_{t}(\dot{B}^{\frac{d}{2}}_{2,1})}\|\tau^{\theta}u\|_{L^1_{t}(\dot{B}^{\frac{d}{2}+1}_{2,1})}\lesssim \mathcal{X}(t)\mathcal{X}_{\theta}(t),\\
&\|\div u\|_{L^1_{t}(L^{\infty})}\|\tau^{\theta}a\|_{\widetilde{L}^{\infty}_{t}(\dot{B}^{\frac{d}{2}}_{2,1})}\lesssim \|u\|_{L^1_{t}(\dot{B}^{\frac{d}{2}+1}_{2,1})}\|\tau^{\theta}a\|_{\widetilde{L}^{\infty}_{t}(\dot{B}^{\frac{d}{2}}_{2,1})}\lesssim \mathcal{X}(t)\mathcal{X}_{\theta}(t),\\
&\sum_{j\in\mathbb{Z}}2^{j(\frac{d}{2}-1)}\|[u\cdot \nabla , \dot{\Delta}_{j}]\tau^{\theta}u\|_{L^1_{t}(L^2)}\lesssim \|u\|_{L^1_{t}(\dot{B}^{\frac{d}{2}+1}_{2,1})}\|\tau^{\theta}u\|_{\widetilde{L}^{\infty}_{t}(\dot{B}^{\frac{d}{2}-1}_{2,1})}\lesssim \mathcal{X}(t)\mathcal{X}_{\theta}(t),\\
&\sum_{k=1}^{d}\sum_{j\in\mathbb{Z}}2^{j(\frac{d}{2}-1)}\|[u\cdot \nabla , \partial_{k}\dot{\Delta}_{j}]\tau^{\theta}a\|_{L^1_{t}(L^2)}\lesssim \|u\|_{L^1_{t}(\dot{B}^{\frac{d}{2}+1}_{2,1})}\|\tau^{\theta}a\|_{\widetilde{L}^{\infty}_{t}(\dot{B}^{\frac{d}{2}}_{2,1})}\lesssim \mathcal{X}(t)\mathcal{X}_{\theta}(t).
\end{aligned}
\right.
\end{equation}
Substituting (\ref{Ltimenonlinear})-(\ref{Ltimenonlinear1}) and (\ref{timehigh})-(\ref{Highinequality411}) into (\ref{decayhigh1}), we have
\begin{equation}\label{decayhigh2}
\begin{aligned}
&\|\tau^{\theta}(\nabla a,u,b,w)\|_{\widetilde{L}^{\infty}_{t}(\dot{B}^{\frac{d}{2}-1}_{2,1})}^{h}+\|\tau^{\theta}(\nabla a,u,b,w)\|_{L^1_{t}(\dot{B}^{\frac{d}{2}-1}_{2,1})}^{h}\\
&\quad\lesssim \frac{\mathcal{X}(t)}{\zeta}t^{\theta-\frac{1}{2}(\frac{d}{2}-1-\sigma_{0})}+\big{(}\zeta+ \mathcal{X}(t)\big{)}\mathcal{X}_{\theta}(t).
\end{aligned}
\end{equation}
Then we multiply $(\ref{ENS2})$ by $t^{\theta}$ to get
\begin{equation}\label{tuh}
\begin{aligned}
\partial_{t}(t^{\theta}u)-\Delta (t^{\theta} u)=\theta t^{\theta-1}u-\nabla(t^{\theta}a)+ t^{\theta}w- t^{\theta} u-t^{\theta}u\cdot\nabla u+t^{\theta}G,
\end{aligned}
\end{equation}
with the initial data $t^{\theta}u|_{t=0}=0$. By virtue of Lemma \ref{heat} for (\ref{tuh}), we have
\begin{equation}\nonumber
\begin{aligned}
&\|\tau^{\theta} u\|_{L^1_{t}(\dot{B}^{\frac{d}{2}+1}_{2,1})}^{h}\lesssim \int_{0}^{t}\tau^{\theta-1}\|u\|_{\dot{B}^{\frac{d}{2}-1}_{2,1}}^{h}d\tau+\|\tau^{\theta} a\|_{L^1_{t}(\dot{B}^{\frac{d}{2}}_{2,1})}^{h}+\|\tau^{\theta} (w,u)\|_{L^1_{t}(\dot{B}^{\frac{d}{2}-1}_{2,1})}^{h}\\
&\quad\quad\quad\quad\quad\quad~\quad+\|\tau^{\theta} (u\cdot\nabla u)\|_{L^1_{t}(\dot{B}^{\frac{d}{2}-1}_{2,1})}^{h}+\|\tau^{\theta} G\|_{L^1_{t}(\dot{B}^{\frac{d}{2}-1}_{2,1})}^{h},
\end{aligned}
\end{equation}
which together with (\ref{interL2}), $(\ref{Ltimenonlinear})$-$(\ref{Ltimenonlinear1})$, (\ref{timehigh}) and (\ref{decayhigh2}) gives rise to
\begin{equation}\label{decayhigh3}
\begin{aligned}
&\|\tau^{\theta} u\|_{L^1_{t}(\dot{B}^{\frac{d}{2}+1}_{2,1})}^{h}\lesssim \frac{\mathcal{X}(t)}{\zeta}t^{\theta-\frac{1}{2}(\frac{d}{2}-1-\sigma_{0})}+\zeta \mathcal{X}_{\theta}(t)+\mathcal{X}(t)\mathcal{X}_{\theta}(t).
\end{aligned}
\end{equation}

Finally, we multiply the inequality (\eqref{dEulerpp}) by $t^{\theta}$ to obtain
\begin{equation}\label{addppt}
\begin{aligned}
&\frac{d}{dt}\big{(}t^{\theta}\mathcal{E}_{3,j}(t)\big{)}+t^{\theta}\mathcal{E}_{3,j}(t)\\
&\lesssim t^{\theta-1}\mathcal{E}_{3,j}(t)\\
&\quad+ t^{\theta}\Big{(} \|\dot{\Delta}_{j}u\|_{L^2}+\|\div w\|_{L^{\infty}}\|\dot{\Delta}_{j}(b,w)\|_{L^2}+2^{-j}\|\dot{\Delta}_{j}(w\cdot\nabla b,w\cdot\nabla w)\|_{L^2}\\
&\quad+\|[w\cdot \nabla,\dot{\Delta}_{j}](b,w)\|_{L^2}\Big{)}\sqrt{\mathcal{E}_{3,j}(t)},\quad\quad j\geq-1.
\end{aligned}
\end{equation}
Therefore, the above inequality \eqref{addppt} as well as (\ref{eulerinequality1}) implies
\begin{equation}\label{eulertimein}
\begin{aligned}
&\|\tau^{\theta}(b,w)\|_{\widetilde{L}^{\infty}_{t}(\dot{B}^{\frac{d}{2}+1}_{2,1})}^{h}+\|\tau^{\theta}(b,w)\|_{L^1_{t}(\dot{B}^{\frac{d}{2}+1}_{2,1})}^{h}\\
&\quad\lesssim \int_{0}^{t}\tau^{\theta-1}\|(b,w)\|_{\dot{B}^{\frac{d}{2}+1}_{2,1}}^{h}d\tau+ \|\tau^{\theta}u\|_{L^1_{t}(\dot{B}^{\frac{d}{2}+1}_{2,1})}^{h}\\
&\quad\quad+\|\div w\|_{L^{\infty}_{t}(L^{\infty})}\|\tau^{\theta}(b,w)\|_{L^1_{t}(\dot{B}^{\frac{d}{2}+1}_{2,1})}^{h}+\|\tau^{\theta}(w\cdot\nabla b,w\cdot\nabla w)\|_{L^1_{t}(\dot{B}^{\frac{d}{2}}_{2,1})}^{h}\\
&\quad\quad+\int_{0}^{t}\sum_{j\geq-1}2^{j(\frac{d}{2}+1)}\|[w\cdot \nabla,\dot{\Delta}_{j}]\tau^{\theta}(b,w)\|_{L^2}d\tau.
\end{aligned}
\end{equation}
The right-hand side of \eqref{eulertimein} can be controlled below. As in (\ref{interL1}), one can get
\begin{equation}\label{hightime111}
\begin{aligned}
&\int_{0}^{t}\tau^{\theta-1}\|(b,w)\|_{\dot{B}^{\frac{d}{2}+1}_{2,1}}^{h}d\tau\lesssim \frac{\mathcal{X}(t)}{\zeta}t^{\theta-\frac{1}{2}(\frac{d}{2}-1-\sigma_{0})}+\zeta \mathcal{X}_{\theta}(t).
\end{aligned}
\end{equation}
Due to (\ref{uv2}), we get
\begin{equation}\label{hightime112}
\begin{aligned}
&\|\div w\|_{L^{\infty}_{t}(L^{\infty})}\|\tau^{\theta}(b,w)\|_{L^1_{t}(\dot{B}^{\frac{d}{2}+1}_{2,1})}^{h}+\|\tau^{\theta}(w\cdot\nabla b,w\cdot\nabla w)\|_{L^1_{t}(\dot{B}^{\frac{d}{2}}_{2,1})}\\
&\quad\lesssim \|w\|_{\widetilde{L}^{\infty}_{t}(\dot{B}^{\frac{d}{2}+1}_{2,1})}\|\tau^{\theta}(b,w)\|_{L^1_{t}(\dot{B}^{\frac{d}{2}+1}_{2,1})}\lesssim\mathcal{X}(t)\mathcal{X}_{\theta}(t).
\end{aligned}
\end{equation}
By $(\ref{commutator})_{1}$, it also holds
\begin{equation}\label{hightime113}
\begin{aligned}
&\int_{0}^{t}\sum_{j\in\mathbb{Z}}2^{j(\frac{d}{2}+1)}\|[w\cdot \nabla,\dot{\Delta}_{j}]\tau^{\theta}(b,w)\|_{L^2}d\tau\lesssim \|w\|_{L^1_{t}(\dot{B}^{\frac{d}{2}+1}_{2,1})}\|\tau^{\theta}(b,w)\|_{\widetilde{L}^{\infty}_{t}(\dot{B}^{\frac{d}{2}+1}_{2,1})}\lesssim \mathcal{X}(t)\mathcal{X}_{\theta}(t).
\end{aligned}
\end{equation}
Thence it follows by (\ref{decayhigh3})-(\ref{hightime113}) that
\begin{equation}\label{eulertimeinp}
\begin{aligned}
&\|\tau^{\theta}(b,w)\|_{\widetilde{L}^{\infty}_{t}(\dot{B}^{\frac{d}{2}+1}_{2,1})}^{h}+\|\tau^{\theta}(b,w)\|_{L^1_{t}(\dot{B}^{\frac{d}{2}+1}_{2,1})}^{h}\\
&\quad\lesssim \frac{\mathcal{X}(t)}{\zeta}t^{\theta-\frac{1}{2}(\frac{d}{2}-1-\sigma_{0})}+  \big{(}\zeta+ \mathcal{X}(t)\big{)}\mathcal{X}_{\theta}(t).
\end{aligned}
\end{equation}
The combination of (\ref{decayhigh2}), (\ref{decayhigh3}) and (\ref{eulertimeinp}) leads to (\ref{decayhigh}), and the proof of Lemma \ref{lemma52} is completed.
\end{proof}

\vspace{2ex}

\underline{\it\textbf{Proof of Theorem \ref{theorem13}:}}~ Let the assumptions of Theorem \ref{theorem13} hold, and $(a,u,b,w)$ be the global solution to the Cauchy problem \eqref{m1n} given by Theorem \ref{theorem12}. By \eqref{XX0}, \eqref{lowd} and \eqref{mathcalXtheta}, there holds
\begin{equation}\label{addpp}
\left\{
\begin{aligned}
&\|(a,u,b,w)(t)\|_{\dot{B}^{\frac{d}{2}-1}_{2,1}}^{\ell}\lesssim \mathcal{X}(t)<<1,\\
&\|a(t)\|_{\dot{B}^{\frac{d}{2}}_{2,1}}^{h}+\|u(t)\|_{\dot{B}^{\frac{d}{2}-1}_{2,1}}^{h}+\|(b,w)(t)\|_{\dot{B}^{\frac{d}{2}+1}_{2,1}}^{h}\lesssim \mathcal{X}(t)<<1,\\
&\|(a,u,b,w)(t)\|_{\dot{B}^{\sigma_{0}}_{2,\infty}}^{\ell}\lesssim \mathcal{X}_{L,\sigma_{0}}(t)\lesssim \delta_{0},\\
&t^{\theta}\|(a,u,b,w)(t)\|_{\dot{B}^{\frac{d}{2}-1}_{2,1}}^{\ell}\lesssim \mathcal{X}_{\theta}(t),\\
&t^{\theta}\big(\|a(t)\|_{\dot{B}^{\frac{d}{2}}_{2,1}}^{h}+\|u(t)\|_{\dot{B}^{\frac{d}{2}-1}_{2,1}}^{h}+\|(b,w)(t)\|_{\dot{B}^{\frac{d}{2}+1}_{2,1}}^{h}\big{)}\lesssim \mathcal{X}_{\theta}(t),
\end{aligned}
\right.
\end{equation}
where $\mathcal{X}(t)$, $\mathcal{X}_{L,\sigma_{0}}(t)$ and $\mathcal{X}_{\theta}(t)$ are defined by \eqref{XX00}, \eqref{lowd} and \eqref{mathcalXtheta}, respectively.

For any $\theta>\frac{1}{2}(\frac{d}{2}+1-\sigma_{0})>1$, we obtain by Lemmas \ref{lemma51}-\ref{lemma52} that
\begin{equation}\label{addpp1}
\begin{aligned}
&\mathcal{X}_{\theta}(t)\lesssim \frac{\mathcal{X}(t)+\mathcal{X}_{L,\sigma_{0}}(t)}{\zeta}t^{\theta-\frac{1}{2}(\frac{d}{2}-1-\sigma_{0})}+  \big{(}\zeta+ \mathcal{X}(t)\big{)}\mathcal{X}_{\theta}(t),\quad t>0.
\end{aligned}
\end{equation}
Choosing a suitably small constant $\zeta>0$ in \eqref{addpp1} and employing $\eqref{addpp}_{1}$-$\eqref{addpp}_{2}$, we have
\begin{equation}\label{Xtimetime}
\begin{aligned}
&\mathcal{X}_{\theta}(t)\lesssim  \delta_{0} t^{\theta-\frac{1}{2}(\frac{d}{2}-1-\sigma_{0})},\quad t>0,
\end{aligned}
\end{equation}
which together with $\eqref{addpp}_{3}$ implies
\begin{equation}\label{d21}
\begin{aligned}
\|(a,u,b,w)(t)\|_{\dot{B}^{\frac{d}{2}-1}_{2,1}}&\lesssim \|(a,u,b,w)(t)\|_{\dot{B}^{\frac{d}{2}-1}_{2,1}}^{\ell}+\|a(t)\|_{\dot{B}^{\frac{d}{2}}_{2,1}}^{h}+\|u(t)\|_{\dot{B}^{\frac{d}{2}-1}_{2,1}}^{h}+\|(b,w)(t)\|_{\dot{B}^{\frac{d}{2}+1}_{2,1}}^{h}\\
&\lesssim \delta_{0}(1+t)^{-\frac{1}{2}(\frac{d}{2}-1-\sigma_{0})},\quad t\geq1.
\end{aligned}
\end{equation}
Owing to \eqref{XX0}, there is a time $t_{*}>0$ such that $\|u(t_{*})\|_{\dot{B}^{\frac{d}{2}+1}_{2,1}}\lesssim \delta_{0}$ follows. For simplicity, we set $t_{*}=1$. By applying Lemma \ref{heat} to \eqref{tuh} on $[1,t]$ and using the estimates \eqref{Ltimenonlinear}-\eqref{Ltimenonlinear1} and \eqref{Xtimetime}, we get
\begin{equation}\nonumber
\begin{aligned}
\|\tau^{\theta}u\|_{\widetilde{L}^{\infty}_{t}(\dot{B}^{\frac{d}{2}+1}_{2,1})}^{h}&\lesssim \|u(1)\|_{\dot{B}^{\frac{d}{2}+1}_{2,1}}^{h}+\|\tau^{\theta-1}u\|_{\widetilde{L}^{\infty}(1,t;\dot{B}^{\frac{d}{2}-1}_{2,1})}^{h}+\|\tau^{\theta}a\|_{\widetilde{L}^{\infty}(1,t;\dot{B}^{\frac{d}{2}}_{2,1})}^{h}+\|\tau^{\theta}( w,u)\|_{\widetilde{L}^{\infty}(1,t;\dot{B}^{\frac{d}{2}-1}_{2,1})}^{h}\\
&\quad+\|\tau^{\theta}u\cdot\nabla u\|_{\widetilde{L}^{\infty}(1,t;\dot{B}^{\frac{d}{2}-1}_{2,1})}^{h}+\| \tau^{\theta}G\|_{\widetilde{L}^{\infty}(1,t;\dot{B}^{\frac{d}{2}-1}_{2,1})}^{h}\\
&\lesssim  \|u(1)\|_{\dot{B}^{\frac{d}{2}+1}_{2,1}}^{h}+\mathcal{X}_{\theta}(t)+\mathcal{X}(t)\mathcal{X}_{\theta}(t)\lesssim \delta_{0} (1+t)^{\theta-\frac{1}{2}(\frac{d}{2}-1-\sigma_{0})},\quad\quad t\geq 1,
\end{aligned}
\end{equation}
from which one infers
\begin{align}
&\|u(t)\|_{\dot{B}^{\frac{d}{2}+1}_{2,1}}^{h}\lesssim \delta_{0}(1+t)^{\frac{1}{2}(\frac{d}{2}-1-\sigma_{0})},\quad\quad t\geq1. \label{d221121}
\end{align}
Then it follows from $\eqref{addpp}_{2}$,  (\ref{d21}) and the interpolation inequality (\ref{inter}) that
\begin{equation}\label{d21222}
\begin{aligned}
&\|(a,u,b,w)^{\ell}(t)\|_{\dot{B}^{\sigma}_{2,1}}\\
&\quad\lesssim \|(a,u,b,w)^{\ell}(t)\|_{\dot{B}^{\sigma_{0}}_{2,\infty}}^{\frac{\frac{d}{2}-1-\sigma}{\frac{d}{2}-1-\sigma_{0}}}\|(a,u,b,w)^{\ell}(t)\|_{\dot{B}^{\frac{d}{2}-1}_{2,1}}^{\frac{\sigma-\sigma_{0}}{\frac{d}{2}-1-\sigma_{0}}}\lesssim \delta_{0} (1+t)^{-\frac{1}{2}(\sigma-\sigma_{0})},\quad \quad \sigma\in (\sigma_{0},\frac{d}{2}-1).
\end{aligned}
\end{equation}
By (\ref{d21})-(\ref{d21222}), the optimal time-decay estimates in $(\ref{decay1})_{1}$-$\eqref{decay1}_{2}$ hold.

Furthermore, we show that the relative velocity $u-w$ satisfies the faster time-decay rate in $\eqref{decay1}_{3}$. The equation (\ref{relativedamp1}) can be re-written as
\begin{equation}\label{relativedamp}
\begin{aligned}
u-w&=e^{-2 t}(u_{0}-w_{0})\\
&\quad+\int_{0}^{t}e^{-2(\tau-t)}\big{(}-\nabla a+  \Delta u+\nabla b-u\cdot \nabla u+w\cdot \nabla w+G\big{)}d\tau.
\end{aligned}
\end{equation}
We take the low-frequency $\dot{B}^{\sigma_{0}}_{2,\infty}$-norm of \eqref{relativedamp} and make use of \eqref{decay1} to get
\begin{equation}\label{relativesmall1000}
\begin{aligned}
&\|(u-w)(t)\|_{\dot{B}^{\sigma_{0}}_{2,\infty}}^{\ell}\\
&\quad\lesssim e^{-2 t}\|(u_{0},w_{0})\|_{\dot{B}^{\sigma_{0}}_{2,\infty}}^{\ell}\\
&\quad\quad+\int_{0}^{t}e^{-2(t-\tau)}\big{(}\|(a,u,b)\|_{\dot{B}^{\sigma_{0}+1}_{2,\infty}}^{\ell}+\|(u\cdot\nabla u,w\cdot \nabla w)\|_{\dot{B}^{\sigma_{0}}_{2,\infty}}^{\ell}+\|G\|_{\dot{B}^{\sigma_{0}}_{2,\infty}}^{\ell}\big{)}d\tau,\quad t>0.
\end{aligned}
\end{equation}
Employing \eqref{decay1}, we have
\begin{equation}\label{relativesmall131111pp}
\begin{aligned}
\int_{0}^{t}e^{-2(t-\tau)}\|(a,u,b)\|_{\dot{B}^{\sigma_{0}+1}_{2,\infty}}^{\ell}d\tau\lesssim (1+t)^{-\min\{\frac{1}{2},\frac{1}{2}(\frac{d}{2}-1-\sigma_{0})\}}.
\end{aligned}
\end{equation}
In addition, it holds by (\ref{XX0}), \eqref{lowd}, (\ref{d21}) and (\ref{uv3}) that
\begin{equation}\label{relativesmall131111}
\begin{aligned}
&\int_{0}^{t}e^{-2(t-\tau)}\|(u\cdot \nabla u,w\cdot\nabla w)\|_{\dot{B}^{\sigma_{0}}_{2,\infty}}^{\ell}+\|G\|_{\dot{B}^{\sigma_{0}}_{2,\infty}}^{\ell}\big{)}d\tau\\
&\quad\lesssim \int_{0}^{t}e^{-2(t-\tau)}\big( \|(a,u,w)\|_{\dot{B}^{\frac{d}{2}}_{2,1}} \|(a,u,w)\|_{\dot{B}^{\sigma_{0}+1}_{2,1}}+\|(a,b)\|_{\dot{B}^{\frac{d}{2}}_{2,1}} \|u-w\|_{\dot{B}^{\sigma_{0}}_{2,\infty}}\big)d\tau\\
&\quad\lesssim \Big( \int_{0}^{t}e^{-4(t-\tau)}  \|(a,u,w)\|_{\dot{B}^{\sigma_{0}+1}_{2,1}}^2 d\tau \Big)^{\frac{1}{2}} \|(a,u,w)\|_{\widetilde{L}^2_{t}(\dot{B}^{\frac{d}{2}}_{2,1})}\\
&\quad\quad+\Big( \int_{0}^{t}e^{-4(t-\tau)}\big( \|(a,b)\|_{\dot{B}^{\frac{d}{2}-1}_{2,1}}^{\ell}+\|a\|_{\dot{B}^{\frac{d}{2}}_{2,1}}^{h}+\|b\|_{\dot{B}^{\frac{d}{2}+1}_{2,1}}^{h} \big)^2 d\tau \Big)^{\frac{1}{2}}\|u-w\|_{\widetilde{L}^2_{t}(\dot{B}^{\sigma_{0}}_{2,\infty})}\\
&\quad\lesssim (1+t)^{-\min\{\frac{1}{2},\frac{1}{2}(\frac{d}{2}-1-\sigma_{0})\}}.
\end{aligned}
\end{equation}
By (\ref{relativesmall1000})-(\ref{relativesmall131111}), $\eqref{decay1}_{3}$ follows. Since $\eqref{decay11}$ can be proved in a similar way, we omit the details. The proof of Theorem \ref{theorem13} is completed.

\subsection{The proof of Theorem \ref{theorem14}}\label{sectionsmall}

%$\dot{B}^{\sigma_{0}}_{2,\infty}$-smallness assumption on the initial data

In this subsection, we prove Theorem \ref{theorem14} on the time-decay estimates of the global solution to the Cauchy problem \eqref{m1n} in the case that $\|(a_{0},u_{0},b_{0},w_{0})^{\ell}\|_{\dot{B}^{\sigma_{0}}_{2,\infty}}$ is sufficiently small. In what follows, we need to use the following elementary inequality frequenctly:
\begin{equation}\label{decayw}
\begin{aligned}
&\int_{0}^{t}\langle t-\tau\rangle^{-\frac{1}{2}(\sigma-\sigma_{0})}\langle \tau\rangle^{-\sigma_{1}}d\tau\lesssim \langle t\rangle^{-\frac{1}{2}(\sigma-\sigma_{0})},\quad\quad 0\leq \frac{1}{2}(\sigma-\sigma_{0})\leq \sigma_{1},\quad \sigma_{1}>1.
\end{aligned}
\end{equation}
Define the time-weighted energy functional
\begin{equation}\label{r44}
\begin{aligned}
\mathcal{Z}(t):&=\sup_{\sigma\in [\sigma_{0}+\varepsilon,\frac{d}{2}+1]}\|\langle \tau\rangle^{\frac{1}{2}(\sigma-\sigma_{0})}(a,u,b,w)\|_{L^{\infty}_{t}(\dot{B}^{\sigma}_{2,1})}^{\ell}\\
&\quad+\|\langle \tau\rangle^{\frac{1}{2}}(u-w)\|_{L^{\infty}_{t}(\dot{B}^{\sigma_{0}}_{2,\infty})}^{\ell}+\sup_{\sigma\in [\sigma_{0}+\varepsilon,\frac{d}{2}]}\|\langle \tau\rangle^{\frac{1}{2}(1+\sigma-\sigma_{0})}(u-w)\|_{L^{\infty}_{t}(\dot{B}^{\sigma}_{2,1})}^{\ell}\\
&\quad+\|\langle \tau\rangle^{\alpha}a\|_{\widetilde{L}^{\infty}_{t}(\dot{B}^{\frac{d}{2}}_{2,1})}^{h}+\|\langle\tau\rangle^{\alpha}u\|_{\widetilde{L}^{\infty}_{t}(\dot{B}^{\frac{d}{2}-1}_{2,1})}^{h}+\|\tau^{\alpha}u\|_{\widetilde{L}^{\infty}_{t}(\dot{B}^{\frac{d}{2}+1}_{2,1})}^{h}+\|\langle \tau\rangle^{\alpha}(b,w)\|_{\widetilde{L}^{\infty}_{t}(\dot{B}^{\frac{d}{2}+1}_{2,1})}^{h},
\end{aligned}
\end{equation}
with $\langle t\rangle:=(1+t^2)^{\frac{1}{2}}$ and $\alpha:=\frac{1}{2}(d+1-2\sigma_{0}-2\varepsilon)$ for a sufficiently small constant $\varepsilon\in(0,1]$.

First, we have the time-weighted estimates of the solution $(a,u,b,w)$ to the Cauchy problem \eqref{m1n} for low frequencies.
\begin{lemma}\label{lemma61}
Let $(a,u,b,w)$ be the global solution to the Cauchy problem $(\ref{m1n})$ given by Theorem \ref{theorem12}. Then, under the assumptions of Theorem \ref{theorem14}, it holds
\begin{equation}\label{Lowsmall}
\begin{aligned}
&\sup_{\sigma\in[\sigma_{0}+\varepsilon,\frac{d}{2}+1]}\|\langle \tau\rangle^{\frac{1}{2}(\sigma-\sigma_{0})}(a,u,b,w)\|_{L^{\infty}_{t}(\dot{B}^{\sigma}_{2,1})}^{\ell}\\
&\quad\quad\lesssim \|(a_{0},u_{0},b_{0},w_{0})\|_{\dot{B}^{\sigma_{0}}_{2,\infty}}^{\ell}+\mathcal{X}^2(t)+\mathcal{Z}^2(t),\quad t>0,
\end{aligned}
\end{equation}
where $\mathcal{X}(t)$ and $\mathcal{Z}(t)$ are defined by $(\ref{XX00})$ and $(\ref{r44})$, respectively.
\end{lemma}
\begin{proof}
Applying the Gr\"{o}nwall  inequality to (\ref{Lowinequality1}), we get
\begin{equation}\label{Lowsmall1}
\begin{aligned}
&\|\dot{\Delta}_{j}(a,u,b,w)\|_{L^2}\\
&\quad\lesssim e^{-2^{2j}t}\|\dot{\Delta}_{j}(a_{0},u_{0},b_{0},w_{0})\|_{L^2}\\
&\quad\quad+\int_{0}^{t}e^{-2^{2j}(t-\tau)}\big{(}\|\dot{\Delta}_{j}(2^{j}(au), u\cdot \nabla u,w\cdot\nabla b,w\cdot\nabla w)\|_{L^2}+\|\dot{\Delta}_{j}G\|_{L^2}\big{)}d\tau,
\end{aligned}
\end{equation}
 which implies for any $\sigma>\sigma_{0}$ that
\begin{equation}\label{Lowsmall2}
\begin{aligned}
&\|(a,u,b,w)\|_{\dot{B}^{\sigma}_{2,1}}^{\ell}\\
&\quad\lesssim  \langle t\rangle^{-\frac{1}{2}(\sigma-\sigma_{0})}\|(a_{0},u_{0},b_{0},w_{0})\|_{\dot{B}^{\sigma_{0}}_{2,\infty}}^{\ell}\\
&\quad\quad+\int_{0}^{t}\langle t-\tau\rangle^{-\frac{1}{2}(\sigma-\sigma_{0})}\big{(}\|au\|_{\dot{B}^{\sigma_{0}+1}_{2,\infty}}^{\ell}+\|(u\cdot\nabla u,w\cdot\nabla b,w\cdot\nabla w)\|_{\dot{B}^{\sigma_{0}}_{2,\infty}}^{\ell}\\
&\quad\quad+\|g(a) \nabla a\|_{\dot{B}^{\sigma}_{2,\infty}}^{\ell}+\|f(a)  \Delta u    \|_{\dot{B}^{\sigma}_{2,\infty}}^{\ell}+\|h(a,b)(u-w)\|_{\dot{B}^{\sigma}_{2,\infty}}^{\ell}\big{)}d\tau.
\end{aligned}
\end{equation}
To control the first nonlinear term on the right-hand side of (\ref{Lowsmall2}), we consider the cases $t\leq 2$ and $t\geq 2$ separately. For the case $t\leq 2$, we make use of (\ref{lh}), (\ref{uv3}) and $\langle t\rangle\sim 1$ to obtain
\begin{equation}\label{lownon11}
\begin{aligned}
&\int_{0}^{t}\langle t-\tau\rangle^{-\frac{1}{2}(\sigma-\sigma_{0})}\|au\|_{\dot{B}^{\sigma_{0}+1}_{2,\infty}}d\tau\\
&\quad\lesssim\|a\|_{L^{\infty}_{t}(\dot{B}^{\frac{d}{2}}_{2,1})} \big{(} \|u\|_{L^{\infty}_{t}(\dot{B}^{\sigma_{0}+1}_{2,1})}^{\ell}+\|u\|_{L^1_{t}(\dot{B}^{\frac{d}{2}+1}_{2,1})}^{h} \big{)}\\
&\quad\lesssim\big{(}\mathcal{X}(t)\mathcal{Z}(t)+\mathcal{X}^2(t)\big{)}\langle t\rangle^{-\frac{1}{2}(\sigma-\sigma_{0})}.
\end{aligned}
\end{equation}
For the case $t\geq 2$, we split the integration into two parts:
$$
\int_{0}^{t}\langle t-\tau\rangle^{-\frac{1}{2}(\sigma-\sigma_{0})}\|au\|_{\dot{B}^{\sigma_{0}+1}_{2,\infty}}d\tau=\Big{(}\int_{0}^{1}+\int_{1}^{t}\Big{)}\langle t-\tau\rangle^{-\frac{1}{2}(\sigma-\sigma_{0})}\|au\|_{\dot{B}^{\sigma_{0}+1}_{2,\infty}}d\tau.
$$
Owing to (\ref{uv3}) and the fact $\langle t-\tau\rangle\sim \langle t\rangle$ for any $\tau\in[0,1]$, we have
\begin{equation}\label{lownon12}
\begin{aligned}
&\int_{0}^{1}\langle t-\tau\rangle^{-\frac{1}{2}(\sigma-\sigma_{0})}\|au\|_{\dot{B}^{\sigma_{0}+1}_{2,\infty}}d\tau\lesssim \langle t\rangle^{-\frac{1}{2}(\sigma-\sigma_{0})}\int_{0}^{1}\|a\|_{\dot{B}^{\frac{d}{2}}_{2,1}}\|u\|_{\dot{B}^{\sigma_{0}+1}_{2,\infty}}d\tau\\
&\quad\quad\quad\quad\quad\quad\quad\quad\quad\quad\quad\quad\quad\quad\lesssim \big{(}\mathcal{X}(t)\mathcal{Z}(t)+\mathcal{X}^2(t)\big{)}\langle t\rangle^{-\frac{1}{2}(\sigma-\sigma_{0})}.
\end{aligned}
\end{equation}
On the other hand, since it holds for $\sigma_{0}\in[-\frac{d}{2},\frac{d}{2}-1)$ and $\sigma\in (\sigma_{0},\frac{d}{2}+1]$ that
\begin{equation}\nonumber
\begin{aligned}
&\frac{1}{2}(\frac{d}{2}+1-\sigma_{0})>1,\quad\quad 0\leq \frac{1}{2}(\sigma-\sigma_{0})\leq \frac{1}{2}(\frac{d}{2}+1-\sigma_{0}),
\end{aligned}
\end{equation}
we use  (\ref{decayw}), (\ref{lh}), (\ref{uv3}), $au=a^{\ell}u^{\ell}+a^{h}u^{\ell}+a^{\ell}u^{h}+a^{h}u^{h}$ and $\tau^{-1}\lesssim \langle \tau\rangle^{-1}$ for $\tau\geq 1$ to get
\begin{equation}\label{non12}
\begin{aligned}
&\int_{1}^{t}\langle t-\tau\rangle^{-\frac{1}{2}(\sigma-\sigma_{0})}\|au\|_{\dot{B}^{\sigma_{0}+1}_{2,\infty}}d\tau\\
&\quad \lesssim\int_{1}^{t}\langle t-\tau\rangle^{-\frac{1}{2}(\sigma-\sigma_{0})} \big(\|a\|_{\dot{B}^{\frac{d}{2}}_{2,1}}^{\ell}\|u\|_{\dot{B}^{\sigma_{0}+1}_{2,1}}^{\ell}+\|a\|_{\dot{B}^{\frac{d}{2}}_{2,1}}^{h}\|u\|_{\dot{B}^{\sigma_{0}+1}_{2,1}}^{\ell}\\
&\quad\quad+\|a\|_{\dot{B}^{\frac{d}{2}}_{2,1}}^{\ell}\|u\|_{\dot{B}^{\frac{d}{2}+1}_{2,1}}^{h}+\|a\|_{\dot{B}^{\frac{d}{2}}_{2,1}}^{h}\|u\|_{\dot{B}^{\frac{d}{2}+1}_{2,1}}^{h} \big) d\tau\\
&\quad\lesssim\mathcal{Z}^2(t) \int_{0}^{t} \langle t-\tau\rangle^{-\frac{1}{2}(\sigma-\sigma_{0})} \langle \tau\rangle^{-\frac{1}{2}(\frac{d}{2}+1-\sigma_{0})} d\tau\\
&\quad\lesssim \mathcal{Z}^2(t)\langle t\rangle^{-\frac{1}{2}(\sigma-\sigma_{0})},\quad t\geq2,
\end{aligned}
\end{equation}
 which together with (\ref{lownon11})-(\ref{lownon12}) implies for any $\sigma\in (\sigma_{0}, \frac{d}{2}+1]$ that
\begin{equation}\label{non111}
\begin{aligned}
\int_{0}^{t}\langle t-\tau\rangle^{-\frac{1}{2}(\sigma-\sigma_{0})}\|au\|_{\dot{B}^{\sigma_{0}+1}_{2,\infty}}d\tau\lesssim\big{(}\mathcal{X}^2(t)+\mathcal{Z}^2(t)\big{)} \langle t\rangle^{-\frac{1}{2}(\sigma-\sigma_{0})},\quad t>0,
\end{aligned}
\end{equation}
 Similarly, one can show 
\begin{equation}\label{non112}
\begin{aligned}
&\int_{0}^{t}\langle t-\tau\rangle^{-\frac{1}{2}(\sigma-\sigma_{0})}\big{(}\|(u\cdot\nabla u,w\cdot\nabla b,w\cdot\nabla w)\|_{\dot{B}^{\sigma_{0}}_{2,\infty}}+\|g(a)\nabla a\|_{\dot{B}^{\sigma}_{2,\infty}}+\|f(a)  \Delta u   \|_{\dot{B}^{\sigma}_{2,\infty}}\big{)}d\tau\\
&\quad\lesssim \big{(}\mathcal{X}^2(t)+\mathcal{Z}^2(t)\big{)}\langle t\rangle^{-\frac{1}{2}(\sigma-\sigma_{0})},\quad  t>0.
\end{aligned}
\end{equation}
Due to the lack of one spatial derivative, the estimate of the last nonlinear term in (\ref{Lowsmall2}) is different from usual nonlinearities of compressible Navier-Stokes system \cite{danchin5,xu3}. To overcome this difficulty, according to (\ref{uv3})-\eqref{F1} and $\mathcal{X}(t)\lesssim \mathcal{X}_{0}<<1$, we have
\begin{equation}\label{keypp}
\begin{aligned}
&\int_{0}^{t}\langle t-\tau\rangle^{-\frac{1}{2}(\sigma-\sigma_{0})}\|h(a,b)(u-w)\|_{\dot{B}^{\sigma}_{2,\infty}}d\tau\\
&\quad\lesssim \int_{0}^{t}\langle t-\tau\rangle^{-\frac{1}{2}(\sigma-\sigma_{0})}\|(a,b)\|_{\dot{B}^{\frac{d}{2}}_{2,1}}\|u-w\|_{\dot{B}^{\sigma_{0}}_{2,\infty}}d\tau\lesssim \mathcal{Z}^2(t)\langle t\rangle ^{-\frac{1}{2}(\sigma-\sigma_{0})},
\end{aligned}
\end{equation}
where in the last inequality one has used the key fact
\begin{equation}\label{factddd1}
\begin{aligned}
&\|(a,b)\|_{\dot{B}^{\frac{d}{2}}_{2,1}}\|u-w\|_{\dot{B}^{\sigma_{0}}_{2,\infty}}\\
&\quad\lesssim \big{(}\|(a,b)\|_{\dot{B}^{\frac{d}{2}}_{2,1}}^{\ell}+\|a\|_{\dot{B}^{\frac{d}{2}}_{2,1}}^{h}+\|b\|_{\dot{B}^{\frac{d}{2}+1}_{2,1}}^{h}\big{)}\big{(}\|u-w\|_{\dot{B}^{\sigma_{0}}_{2,\infty}}^{\ell}+\|u\|_{\dot{B}^{\frac{d}{2}-1}_{2,1}}^{h}+\|w\|_{\dot{B}^{\frac{d}{2}+1}_{2,1}}^{h}\big{)}\\
&\quad\lesssim \mathcal{Z}^2(t)\langle t\rangle ^{-\frac{1}{2}(\frac{d}{2}+1-\sigma_{0})}.
\end{aligned}
\end{equation}
Combining (\ref{non111})-(\ref{keypp}) together, we obtain (\ref{Lowsmall}).
\end{proof}

Next, we prove the following time-weighted estimates of the solution $(a,u,b,w)$ to the Cauchy problem \eqref{m1n} for low frequencies.

\begin{lemma}\label{lemma62}
Let $(a,u,b,w)$ be the global solution to the Cauchy problem $(\ref{m1n})$ given by Theorem \ref{theorem12}. Then, under the assumptions of Theorem \ref{theorem14}, it holds
\begin{equation}\label{Highsmall}
\begin{aligned}
&\|\langle \tau\rangle^{\alpha}a\|_{\widetilde{L}^{\infty}_{t}(\dot{B}^{\frac{d}{2}}_{2,1})}^{h}+\|\langle \tau\rangle^{\alpha}u\|_{\widetilde{L}^{\infty}_{t}(\dot{B}^{\frac{d}{2}-1}_{2,1})}^{h}\\
&\quad\quad+\|\tau^{\alpha}u\|_{\widetilde{L}^{\infty}_{t}(\dot{B}^{\frac{d}{2}+1}_{2,1})}^{h}+\|\langle \tau\rangle^{\alpha}(b,w)\|_{\widetilde{L}^{\infty}_{t}(\dot{B}^{\frac{d}{2}+1}_{2,1})}^{h}\\
&\quad\lesssim \|a_{0}\|_{\dot{B}^{\frac{d}{2}}_{2,1}}^{h}+\|u_{0}\|_{\dot{B}^{\frac{d}{2}-1}_{2,1}}^{h}+\|(b_{0},w_{0})\|_{\dot{B}^{\frac{d}{2}+1}_{2,1}}^{h}+\mathcal{X}^2(t)+\mathcal{Z}^2(t),\quad t>0,
\end{aligned}
\end{equation}
where $\mathcal{X}(t)$ and $\mathcal{Z}(t)$ are defined by $(\ref{XX00})$ and $(\ref{r44})$, respectively, and $\alpha$ is given by $\alpha:=\frac{1}{2}(d+1-2\sigma_{0}-2\varepsilon)$ for a sufficiently small constant $\varepsilon\in (0,1]$.
\end{lemma}

\begin{proof}
First, it follows from (\ref{Highinequality2}) for any $j\geq -1$ that
\begin{equation}\nonumber
\begin{aligned}
&\|\dot{\Delta}_{j}(\nabla a,u,b,w)\|_{L^2}\lesssim e^{-t}\|\dot{\Delta}_{j}(\nabla a_{0},u_{0},b_{0},w_{0})\|_{L^2}+\int_{0}^{t}e^{-(t-\omega)}\sum_{i=1}^{7}I_{i,j}d\omega,
\end{aligned}
\end{equation}
where $I_{i,j}$ $(i=1,...,7)$ are given by
\begin{equation}\nonumber
\left\{
\begin{aligned}
&I_{1,j}:=2^{j}\|\dot{\Delta}_{j}(au) \|_{L^2},~~I_{2,j}:=\|\dot{\Delta}_{j}(u\cdot \nabla u)\|_{L^2}, ~~ I_{3,j}:=\|\dot{\Delta}_{j}G\|_{L^2},~~ I_{4,j}:=\|\div u\|_{L^{\infty}}\|\nabla\dot{\Delta}_{j}a\|_{L^2},\\
& I_{5,j}:=\|\nabla\dot{\Delta}_{j}(\nabla a\div u)\|_{L^2},~~I_{6,j}:=\|[u\cdot \nabla ,\dot{\Delta}_{j}]u\|_{L^2},~~ I_{7,j}=\sum_{k=1}^{d}\|[u\cdot \nabla ,\partial_{k} \dot{\Delta}_{j}]a\|_{L^2}.
\end{aligned}
\right.
\end{equation}
Therefore, one has
\begin{equation}\label{highdecaypp1}
\begin{aligned}
&\|\langle\tau\rangle^{\alpha}(\nabla a,u,b,w)\|_{\widetilde{L}^{\infty}_{t}(\dot{B}^{\frac{d}{2}-1}_{2,1})}^{h}\\
&\quad\lesssim \|(\nabla a_{0},u_{0},b_{0},w_{0})\|_{\dot{B}^{\frac{d}{2}-1}_{2,1}}^{h}+\sum_{j\geq-1} \sup_{\tau\in [0,t]}\langle \tau\rangle^{\alpha}\int_{0}^{\tau}e^{-(\tau-\omega)}2^{j(\frac{d}{2}-1)}\sum_{i=1}^{7}I_{i,j}d\omega.
\end{aligned}
\end{equation}
To control the nonlinear terms of (\ref{highdecaypp1}), we may consider two cases $t\leq 2$ and $t\geq 2$ and split the integration over $[0,t]$ into $[0,1]$ and $[1,t]$ for $t\geq 2$, respectively. One can show
\begin{equation}\nonumber
\begin{aligned}
&\sum_{j\geq-1} \sup_{\tau\in [0,t]}\langle \tau\rangle^{\alpha}\int_{0}^{\tau}e^{-(\tau-\omega)}2^{j(\frac{d}{2}-1)}(I_{1,j}+I_{2,j})d\omega\\
&\quad\lesssim \int_{0}^{t} \big{(}\|a u\|_{\dot{B}^{\frac{d}{2}}_{2,1}}^{h}+\|u\cdot\nabla u\|_{\dot{B}^{\frac{d}{2}-1}_{2,1}}^{h}\big{)}d\tau\lesssim \|(a,u)\|_{L^2_{t}(\dot{B}^{\frac{d}{2}}_{2,1})}\|u\|_{L^2_{t}(\dot{B}^{\frac{d}{2}}_{2,1})}\lesssim \mathcal{X}^2(t),\quad\quad t\leq 2.
\end{aligned}
\end{equation}
By direct computations, it also holds
\begin{equation}\nonumber
\begin{aligned}
&\sum_{j\geq-1} \sup_{\tau\in [2,t]}\langle \tau\rangle^{\alpha}\int_{0}^{1}e^{-(\tau-\omega)}(I_{1,j}+I_{2,j})d\omega\lesssim \mathcal{X}^2(1),\quad\quad t\geq 2.
\end{aligned}
\end{equation}
We turn to estimates the integration on $[1,t]$ of the first and second nonlinear terms on the right-hand side of $(\ref{highdecaypp1})$ for $t\geq 2$. Due to (\ref{lh}),  (\ref{uv1}) and
\begin{equation}\label{keyfact}
\left\{
\begin{aligned}
&\|\langle \tau\rangle^{\frac{1}{2}(\frac{d}{2}-\sigma_{0}+\varepsilon)}(a,u)^{\ell}\|_{\widetilde{L}^{\infty}_{t}(\dot{B}^{\frac{d}{2}}_{2,1})}\lesssim \|\langle \tau\rangle^{\frac{1}{2}(\frac{d}{2}-\sigma_{0}+\varepsilon)}(a,u)\|_{L^{\infty}_{t}(\dot{B}^{\frac{d}{2}-\varepsilon}_{2,1})}^{\ell}\lesssim\mathcal{Z}(t) ,\\
&\|\langle \tau\rangle^{\frac{1}{2}(\frac{d}{2}+1-\sigma_{0}+\varepsilon)}(a,u)^{\ell}\|_{\widetilde{L}^{\infty}_{t}(\dot{B}^{\frac{d}{2}+1}_{2,1})}\lesssim \|\langle \tau\rangle^{\frac{1}{2}(\frac{d}{2}+1-\sigma_{0}+\varepsilon)}(a,u)\|_{L^{\infty}_{t}(\dot{B}^{\frac{d}{2}+1-\varepsilon}_{2,1})}^{\ell}\lesssim\mathcal{Z}(t),
\end{aligned}
\right.
\end{equation}
one gets
\begin{equation}\label{keyfact1}
\begin{aligned}
&\|\tau^{\alpha} a^{\ell}u^{\ell}\|_{\widetilde{L}^{\infty}_{t}(\dot{B}^{\frac{d}{2}}_{2,1})}^{h}+\|\tau^{\alpha}u^{\ell}\cdot\nabla u^{\ell}\|_{\widetilde{L}^{\infty}_{t}(\dot{B}^{\frac{d}{2}-1}_{2,1})}^{h} \\
&\quad\lesssim \|\tau^{\alpha}a^{\ell}u^{\ell}\|_{\widetilde{L}^{\infty}_{t}(\dot{B}^{\frac{d}{2}+1}_{2,1})}^{h}+\|\tau^{\alpha}u^{\ell}\cdot\nabla u^{\ell}\|_{\widetilde{L}^{\infty}_{t}(\dot{B}^{\frac{d}{2}+1}_{2,1})}^{h}\\
&\quad\lesssim \|\langle \tau\rangle^{\frac{1}{2}(\frac{d}{2}+1-\sigma_{0}+\varepsilon)}(a,u)^{\ell}\|_{\dot{B}^{\frac{d}{2}+1}_{2,1}}\|\langle \tau\rangle^{\frac{1}{2}(\frac{d}{2}-\sigma_{0}+\varepsilon)}(a,u)^{\ell}\|_{\dot{B}^{\frac{d}{2}}_{2,1}}\lesssim \mathcal{Z}^2(t).
\end{aligned}
\end{equation}
By (\ref{uv1})-(\ref{uv2}), there holds
\begin{equation}\label{keyfact2}
\begin{aligned}
&\|\tau^{\alpha}a^{h}u^{\ell}\|_{\widetilde{L}^{\infty}_{t}(\dot{B}^{\frac{d}{2}}_{2,1})}^{h}+\|\tau^{\alpha} a u^{h}\|_{\widetilde{L}^{\infty}_{t}(\dot{B}^{\frac{d}{2}}_{2,1})}^{h}\\
&\quad\lesssim \|\langle \tau\rangle^{\alpha}a\|_{\widetilde{L}^{\infty}_{t}(\dot{B}^{\frac{d}{2}}_{2,1})}^{h}\|u\|_{\widetilde{L}^{\infty}_{t}(\dot{B}^{\frac{d}{2}-1}_{2,1})}^{\ell}+\|a\|_{\widetilde{L}^{\infty}_{t}(\dot{B}^{\frac{d}{2}}_{2,1})}\|\tau^{\alpha}u\|_{\widetilde{L}^{\infty}_{t}(\dot{B}^{\frac{d}{2}+1}_{2,1})}^{h}\\
&\quad\lesssim \mathcal{Z}(t)\mathcal{X}(t),
\end{aligned}
\end{equation}
and
\begin{equation}\label{keyfact3}
\begin{aligned}
&\|\tau^{\alpha}u^{h}\cdot\nabla u^{\ell}\|_{\widetilde{L}^{\infty}_{t}(\dot{B}^{\frac{d}{2}-1}_{2,1})}^{h} +\|\tau^{\alpha}u\cdot\nabla u^{h}\|_{\widetilde{L}^{\infty}_{t}(\dot{B}^{\frac{d}{2}-1}_{2,1})}^{h}\\
&\quad\lesssim \|\tau^{\alpha}u\|_{\widetilde{L}^{\infty}_{t}(\dot{B}^{\frac{d}{2}}_{2,1})}^{h}\|u\|_{\widetilde{L}^{\infty}_{t}(\dot{B}^{\frac{d}{2}-1}_{2,1})}^{\ell}+\|\tau^{\alpha} u\|_{\widetilde{L}^{\infty}_{t}(\dot{B}^{\frac{d}{2}+1}_{2,1})}^{h}\|u\|_{\widetilde{L}^{\infty}_{t}(\dot{B}^{\frac{d}{2}-1}_{2,1})}\\
&\quad \lesssim \mathcal{Z}(t)\mathcal{X}(t).
\end{aligned}
\end{equation}
For $t\geq 2$ and the integration on $[1,t]$, one deduces by (\ref{keyfact1})-(\ref{keyfact3}) that
\begin{equation}\nonumber
\begin{aligned}
&\sum_{j\geq-1} \sup_{\tau\in [2,t]}\langle \tau\rangle^{\alpha}\int_{1}^{\tau}e^{-(\tau-\omega)}(I_{1,j}+I_{2,j})d\omega\\
&\quad\lesssim \big{(}\|\tau^{\alpha} a u\|_{\widetilde{L}^{\infty}_{t}(\dot{B}^{\frac{d}{2}}_{2,1})}^{h}+\|\tau^{\alpha} u\cdot\nabla u\|_{\widetilde{L}^{\infty}_{t}(\dot{B}^{\frac{d}{2}}_{2,1})}^{h}\big{)}\sup_{\tau\in [2,t]}\langle \tau\rangle^{\alpha}\int_{1}^{\tau}e^{-(\tau-\omega)}\omega^{-\alpha}d\omega\\
&\quad\lesssim \mathcal{X}^2(t)+\mathcal{X}(t)\mathcal{Z}(t).
\end{aligned}
\end{equation}
Therefore, we have
\begin{equation}\label{S1S2}
\begin{aligned}
&\sum_{j\geq-1} \sup_{\tau\in [0,t]}\langle \tau\rangle^{\alpha}\int_{0}^{t}e^{-(\tau-\omega)}(I_{1,j}+I_{2,j})d\tau\lesssim  \mathcal{X}^2(t)+\mathcal{Z}^2(t),\quad t>0.
\end{aligned}
\end{equation}
By similar arguments as in (\ref{S1S2}), one can show
\begin{equation}\label{S1S21}
\begin{aligned}
&\sum_{j\geq-1} \sup_{\tau\in [0,t]}\langle \tau\rangle^{\alpha}\int_{0}^{t}e^{-(\tau-\omega)}\sum_{i=3}^{7}I_{i,j}d\tau\lesssim  \mathcal{X}^2(t)+\mathcal{Z}^2(t),\quad t>0.
\end{aligned}
\end{equation}
Substituting (\ref{S1S2})-(\ref{S1S21}) into (\ref{highdecaypp1}), we obtain 
\begin{equation}\label{highdecaypp55}
\begin{aligned}
&\|\langle\tau\rangle^{\alpha}(\nabla a,u,b,w)\|_{\widetilde{L}^{\infty}_{t}(\dot{B}^{\frac{d}{2}-1}_{2,1})}^{h}\lesssim \|(\nabla a_{0},u_{0},b_{0},w_{0})\|_{\dot{B}^{\frac{d}{2}-1}_{2,1}}^{h}+\mathcal{X}^2(t)+\mathcal{Z}^2(t).
\end{aligned}
\end{equation}

Then, we show the higher-order time-decay estimate of $u$. Employing Lemma \ref{heat} for (\ref{tuh}) with $\theta=\alpha>1$, we gain
\begin{equation}\label{tuhh}
\begin{aligned}
&\|\tau^{\alpha}u\|_{\widetilde{L}^{\infty}_{t}(\dot{B}^{\frac{d}{2}+1}_{2,1})}^{h}\\
&\quad\lesssim \|\tau^{\alpha-1}u\|_{\widetilde{L}^{\infty}_{t}(\dot{B}^{\frac{d}{2}-1}_{2,1})}^{h}+\|\tau^{\alpha}a\|_{\widetilde{L}^{\infty}_{t}(\dot{B}^{\frac{d}{2}}_{2,1})}^{h}+\|\tau^{\alpha}(u,w)\|_{\widetilde{L}^{\infty}_{t}(\dot{B}^{\frac{d}{2}-1}_{2,1})}^{h}\\
&\quad\quad+\|\tau^{\alpha}u\cdot \nabla u\|_{\widetilde{L}^{\infty}_{t}(\dot{B}^{\frac{d}{2}-1}_{2,1})}^{h}+\|\tau^{\alpha}G\|_{\widetilde{L}^{\infty}_{t}(\dot{B}^{\frac{d}{2}-1}_{2,1})}^{h}.
\end{aligned}
\end{equation}
It is easy to verify that
\begin{equation}\label{tuhh22}
\begin{aligned}
&\|\tau^{\alpha-1}u\|_{\widetilde{L}^{\infty}_{t}(\dot{B}^{\frac{d}{2}-1}_{2,1})}^{h}\lesssim \|\langle \tau\rangle^{\alpha}u\|_{\widetilde{L}^{\infty}_{t}(\dot{B}^{\frac{d}{2}-1}_{2,1})}^{h}.
\end{aligned}
\end{equation}
Similarly to (\ref{keyfact1})-(\ref{keyfact3}), one has
\begin{equation}\label{tuhh33}
\begin{aligned}
&\|\tau^{\alpha}u\cdot \nabla u\|_{\widetilde{L}^{\infty}_{t}(\dot{B}^{\frac{d}{2}-1}_{2,1})}^{h}+\|\tau^{\alpha}G\|_{\widetilde{L}^{\infty}_{t}(\dot{B}^{\frac{d}{2}-1}_{2,1})}^{h}\lesssim \mathcal{X}^2(t)+\mathcal{Z}^2(t).
\end{aligned}
\end{equation}
Combining (\ref{highdecaypp55})-(\ref{tuhh33}) together, we get
\begin{equation}\label{highdecaypp56}
\begin{aligned}
\|\tau^{\alpha}u\|_{\widetilde{L}^{\infty}_{t}(\dot{B}^{\frac{d}{2}+1}_{2,1})}^{h}\lesssim \|(\nabla a_{0},u_{0},b_{0},w_{0})\|_{\dot{B}^{\frac{d}{2}-1}_{2,1}}^{h}+\mathcal{X}^2(t)+\mathcal{Z}^2(t).
\end{aligned}
\end{equation}

Furthermore,  to establish the higher order time-weighted estimate of $(b,w)$, one has by (\ref{dEuler}) that
\begin{equation}\nonumber
\begin{aligned}
&\|\langle \tau\rangle^{\alpha}(b,w)\|_{\widetilde{L}^{\infty}_{t}(\dot{B}^{\frac{d}{2}+1}_{2,1})}^{h}\\
&\quad\lesssim \|(b_{0},w_{0})\|_{\dot{B}^{\frac{d}{2}+1}_{2,1}}+\sum_{j\geq-1}\sup_{\tau\in[0,t]}\langle\tau\rangle^{\alpha}\int_{0}^{\tau}e^{-(\tau-\omega)}2^{j(\frac{d}{2}+1)}\big{(} \|\dot{\Delta}_{j}u\|_{L^2}+\|\div w\|_{L^{\infty}}\|\dot{\Delta}_{j}(b,w)\|_{L^2}\\
&\quad\quad+2^{-j}\|\dot{\Delta}_{j}(w\cdot\nabla b,w\cdot\nabla w)\|_{L^2}+\|[w\cdot \nabla,\dot{\Delta}_{j}](b,w)\|_{L^2}\big{)}d\omega.
\end{aligned}
\end{equation}
Note that it holds
\begin{equation}\nonumber
\begin{aligned}
\sum_{j\geq-1}\sup_{\tau\in[0,t]}\langle\tau\rangle^{\alpha}\int_{0}^{\tau}e^{-(\tau-\omega)}2^{j(\frac{d}{2}+1)} \|\dot{\Delta}_{j}u\|_{L^2}d\omega\lesssim \|u\|_{L^1_{t}(\dot{B}^{\frac{d}{2}+1}_{2,1})}^{h},\quad\quad t\leq 2,
\end{aligned}
\end{equation}
and
\begin{equation}\nonumber
\begin{aligned}
&\sum_{j\geq-1}\sup_{\tau\in[0,t]}\langle\tau\rangle^{\alpha}\int_{0}^{\tau}e^{-(\tau-\omega)}2^{j(\frac{d}{2}+1)} \|\dot{\Delta}_{j}u\|_{L^2}d\omega\\
&\quad\lesssim \sum_{j\geq-1} \Big(\sup_{\tau\in[0,1]}\langle\tau\rangle^{\alpha}\int_{0}^{\tau} +\sup_{\tau\in[1,t]}\langle\tau\rangle^{\alpha}\int_{0}^{1}+\sup_{\tau\in[1,t]}\langle\tau\rangle^{\alpha}\int_{1}^{\tau}\Big) e^{-(\tau-\omega)}2^{j(\frac{d}{2}+1)} \|\dot{\Delta}_{j}u\|_{L^2}d\omega\\
&\quad\lesssim \|u\|_{L^1_{1}(\dot{B}^{\frac{d}{2}+1}_{2,1})}^{h}+\|\tau^{\alpha}u\|_{L^1_{t}(\dot{B}^{\frac{d}{2}+1}_{2,1})}^{h},\quad\quad t\geq 2,
\end{aligned}
\end{equation}
from which and (\ref{highdecaypp56}) we get
\begin{equation}\nonumber
\begin{aligned}
&\sum_{j\geq-1}\sup_{\tau\in[0,t]}\langle\tau\rangle^{\alpha}\int_{0}^{\tau}e^{-(\tau-\omega)}2^{j(\frac{d}{2}+1)} \|\dot{\Delta}_{j}u\|_{L^2}d\omega\lesssim \|(\nabla a_{0},u_{0},b_{0},w_{0})\|_{\dot{B}^{\frac{d}{2}-1}_{2,1}}^{h}+\mathcal{X}^2(t)+\mathcal{Z}^2(t).
\end{aligned}
\end{equation}
By (\ref{uv1}) and (\ref{commutator}), one can obtain after direct computations that
\begin{equation}\nonumber
\begin{aligned}
&\sum_{j\geq-1}\sup_{\tau\in[0,t]}\langle\tau\rangle^{\alpha}\int_{0}^{\tau}e^{-(\tau-\omega)}2^{j(\frac{d}{2}+1)}\big{(} \|\div w\|_{L^{\infty}}\|\dot{\Delta}_{j}(b,w)\|_{L^2}\\
&\quad\quad+2^{-j}\|\dot{\Delta}_{j}(w\cdot\nabla b,w\cdot\nabla w)\|_{L^2}+\|[w\cdot \nabla,\dot{\Delta}_{j}](b,w)\|_{L^2}\big{)}d\omega\lesssim\mathcal{Z}^2(t)+ \mathcal{X}(t)\mathcal{Z}(t).
\end{aligned}
\end{equation}
Thus, we prove
\begin{equation}\label{highdecaypp57}
\begin{aligned}
&\|\langle \tau\rangle^{\alpha}(b,w)\|_{\widetilde{L}^{\infty}_{t}(\dot{B}^{\frac{d}{2}+1}_{2,1})}^{h}\\
&\quad\lesssim \|a_{0}\|_{\dot{B}^{\frac{d}{2}}_{2,1}}^{h}+\|u_{0}\|_{\dot{B}^{\frac{d}{2}-1}_{2,1}}^{h}+\|(b_{0},w_{0})\|_{\dot{B}^{\frac{d}{2}+1}_{2,1}}^{h}+\mathcal{X}^2(t)+\mathcal{Z}^2(t).
\end{aligned}
\end{equation}
The combination of (\ref{highdecaypp55}) and (\ref{highdecaypp56})-(\ref{highdecaypp57}) leads to (\ref{Highsmall}) and completes the proof of Lemma \ref{lemma62}.
\end{proof}

Finally, we need the additional time-weighted estimates of the relative velocity $u-w$ to control the nonlinear term $h(a,b)(u-w)$ and enclose the expected time-weighted estimates.
\begin{lemma}\label{lemma63}
Let $(a,u,b,w)$ be the global solution to the Cauchy problem $(\ref{m1n})$ given by Theorem \ref{theorem12}. Then, under the assumptions of Theorem \ref{theorem14}, it holds
\begin{equation}\label{relativesmall}
\begin{aligned}
&\|\langle \tau\rangle^{\frac{1}{2}}(u-w)\|_{L^{\infty}_{t}(\dot{B}^{\sigma_{0}}_{2,\infty})}^{\ell}+\sup_{\sigma\in[\sigma_{0}+\var,\frac{d}{2}]}\|\langle \tau\rangle^{\frac{1}{2}(\sigma-\sigma_{0})}(u-w)\|_{L^{\infty}_{t}(\dot{B}^{\sigma}_{2,1})}^{\ell}\\
&\quad\lesssim \|(a_{0},u_{0},b_{0},w_{0})\|_{\dot{B}^{\sigma_{0}}_{2,\infty}}^{\ell}+\mathcal{X}^2(t)+\mathcal{Z}^2(t),\quad t>0,
\end{aligned}
\end{equation}
where $\mathcal{X}(t)$ and $\mathcal{Z}(t)$ are defined by $(\ref{XX00})$ and $(\ref{r44})$, respectively, and $\var\in(0,1)$ is any suitably small constant.
\end{lemma}

\begin{proof}
Taking the low-frequency $\dot{B}^{\sigma_{0}}_{2,\infty}$-norm of \eqref{relativedamp} for low frequnecies, we get
\begin{equation}\label{relativesmall1}
\begin{aligned}
&\|u-w\|_{\dot{B}^{\sigma_{0}}_{2,\infty}}^{\ell}&\\
&\lesssim e^{-2 t}\|(u_{0},w_{0})\|_{\dot{B}^{\sigma_{0}}_{2,\infty}}^{\ell}\\
&\quad+\int_{0}^{t}e^{-2(t-\tau)}\big{(}\|(a,b,u)\|_{\dot{B}^{\sigma_{0}+1}_{2,\infty}}^{\ell}+\|(u\cdot\nabla u,w\cdot \nabla w)\|_{\dot{B}^{\sigma_{0}}_{2,\infty}}^{\ell}+\|G\|_{\dot{B}^{\sigma_{0}}_{2,\infty}}^{\ell}\big{)}d\tau.
\end{aligned}
\end{equation}
The every term on the right-hand side of (\ref{relativesmall1}) can be estimated as follows. It holds by (\ref{Lowsmall}) that
\begin{equation}\label{relativesmall11}
\begin{aligned}
&\int_{0}^{t}e^{-2(t-\tau)}\big{(}\|(a,b)\|_{\dot{B}^{\sigma_{0}+1}_{2,\infty}}^{\ell}+\|u\|_{\dot{B}^{\sigma_{0}+2}_{2,\infty}}^{\ell}\big{)}d\tau\\
&\quad\lesssim  \|\langle \tau\rangle^{\frac{1}{2}}(a,u)\|_{L^{\infty}_{t}(\dot{B}^{\sigma_{0}+1}_{2,1})}^{\ell}\int_{0}^{t}e^{-2(t-\tau)}\langle \tau\rangle^{-\frac{1}{2}}d\tau\\
&\quad\lesssim \big{(}\|(a_{0},u_{0},b_{0},w_{0})\|_{\dot{B}^{\sigma_{0}}_{2,\infty}}^{\ell}+\mathcal{X}^2(t)+\mathcal{Z}^2(t)\big{)}\langle t\rangle^{-\frac{1}{2}}.
\end{aligned}
\end{equation}
We derive by (\ref{Lowsmall}) and (\ref{uv3}) that
\begin{equation}\label{relativesmall13}
\begin{aligned}
&\int_{0}^{t}e^{-2(t-\tau)}\big{(}\|(u\cdot \nabla u,w\cdot\nabla w)\|_{\dot{B}^{\sigma_{0}}_{2,\infty}}^{\ell}+\|G\|_{\dot{B}^{\sigma_{0}}_{2,\infty}}^{\ell}\big{)}d\tau\\
&\quad\lesssim \big{(}\mathcal{X}^2(t)+\mathcal{Z}^2(t)\big{)} \langle t\rangle^{-\frac{1}{2}(\frac{d}{2}+1-\sigma_{0})}.
\end{aligned}
\end{equation}
The combination of (\ref{relativesmall1})-(\ref{relativesmall13}) gives rise to
\begin{equation}\label{relativesmall2}
\begin{aligned}
\|\langle \tau\rangle^{\frac{1}{2}}(u-w)\|_{L^{\infty}_{t}(\dot{B}^{\sigma_{0}}_{2,\infty})}^{\ell}\lesssim\|(a_{0},u_{0},b_{0},w_{0})\|_{\dot{B}^{\sigma_{0}}_{2,\infty}}^{\ell}+\mathcal{X}^2(t)+\mathcal{Z}^2(t).
\end{aligned}
\end{equation}
By a similar argument as used in (\ref{relativesmall2}), one can have
\begin{equation}\nonumber
\begin{aligned}
\|\langle \tau\rangle^{\frac{1+\sigma-\sigma_{0}}{2}}(u-w)\|_{L^{\infty}_{t}(\dot{B}^{\sigma}_{2,1})}^{\ell}\lesssim \|(a_{0},u_{0},b_{0},w_{0})\|_{\dot{B}^{\sigma_{0}}_{2,\infty}}^{\ell}+\mathcal{X}^2(t)+\mathcal{Z}^2(t), \quad \sigma\in (\sigma_{0},\frac{d}{2}].
\end{aligned}
\end{equation}
For brevity, the details are omitted here. 
\end{proof}

\vspace{2ex}

\underline{\it\textbf{Proof of Theorem \ref{theorem14}:}}~
Assume that the initial data $(a_{0},u_{0},b_{0},w_{0})$ satisfies \eqref{a1} and \eqref{a4}, and let $(a,u,b,w)$ be the global solution to the Cauchy problem \eqref{m1n} given by Theorem \ref{theorem12}. In terms of the time-weighted estimated established in Lemmas \ref{lemma61}-\ref{lemma63}, we have
\begin{equation}\label{z2}
\begin{aligned}
&\mathcal{Z}(t)\lesssim \|(a_{0},u_{0},b_{0},w_{0})\|_{\dot{B}^{\sigma_{0}}_{2,\infty}}^{\ell}+\mathcal{X}_{0}+\mathcal{X}^2(t)+\mathcal{Z}^2(t),\quad\quad t>0,
\end{aligned}
\end{equation}
where $\mathcal{X}(t)$ and $\mathcal{Z}(t)$ are defined by \eqref{XX00} and \eqref{r44}, respectively. Due to \eqref{z2} and the fact that $\delta_{0}\sim\mathcal{X}_{0}+\|(a_{0},u_{0},b_{0},w_{0})\|_{\dot{B}^{\sigma_{0}}_{2,\infty}}^{\ell}$ is sufficiently small, we conclude $\mathcal{Z}(t)\lesssim \delta_{0}$ for any $t>0$. Thus, \eqref{r4} follows.

\vspace{2ex}

\section{Appendix: Littlewood-Paley decomposition and Besov spaces}\label{sectionnotation}

We explain the notations and technical lemmas used throughout this paper. $C>0$ and $c>0$ denote two constants independent of time. $A\lesssim B(A\gtrsim B)$ means $A\leq CB$ $(A\geq CB)$, and $A\sim B$ stands for $A\lesssim B$ and $A\gtrsim B$. For any Banach space $X$ and the functions $g,h\in X$, let $\|(g,h)\|_{X}:=\|g\|_{X}+\|h\|_{X}$. For any $T>0$ and $1\leq \varrho\leq\infty$, we denote by $L^{\varrho}(0,T;X)$ the set of measurable functions $g:[0,T]\rightarrow X$ such that $t\mapsto \|g(t)\|_{X}$ is in $L^{\varrho}(0,T)$ and write $\|\cdot\|_{L^{\varrho}(0,T;X)}:=\|\cdot\|_{L^{\varrho}_{T}(X)}$.

We recall the Littlewood-Paley decomposition, Besov spaces and related analysis tool. The reader can refer to Chapters 2-3 in \cite{bahouri1} for the details. Choose a smooth radial non-increasing function $\chi(\xi)$  compactly supported in $B(0,\frac{4}{3})$ and satisfying $\chi(\xi)=1$ in $B(0,\frac{3}{4})$. Then $\varphi(\xi):=\chi(\frac{\xi}{2})-\chi(\xi)$ satisfies
$$
\sum_{j\in \mathbb{Z}}\varphi(2^{-j}\cdot)=1,\quad \text{{\rm{Supp}}}~ \varphi\subset \{\xi\in \mathbb{R}^{d}~|~\frac{3}{4}\leq |\xi|\leq \frac{8}{3}\}.
$$
For any $j\in \mathbb{Z}$, define the homogeneous dyadic blocks $\dot{\Delta}_{j}$ by
$$
\dot{\Delta}_{j}u:=\mathcal{F}^{-1}\big{(} \varphi(2^{-j}\cdot )\mathcal{F}(u) \big{)}=2^{jd}h(2^{j}\cdot)\star u,\quad\quad h:=\mathcal{F}^{-1}\varphi,
$$
where $\mathcal{F}$ and $\mathcal{F}^{-1}$ are the Fourier transform and its inverse. Let $\mathcal{P}$ be the class of all polynomials on $\mathbb{R}^{d}$ and $\mathcal{S}_{h}'=\mathcal{S}'/\mathcal{P}$ stand for the tempered distributions on $\mathbb{R}^{d}$ modulo polynomials. One can get
\begin{equation}\nonumber
\begin{aligned}
&u=\sum_{j\in \mathbb{Z}}\dot{\Delta}_{j}u\quad\text{in}~\mathcal{S}',\quad \forall u\in \mathcal{S}_{h}',\quad \quad \dot{\Delta}_{j}\dot{\Delta}_{l}u=0,\quad\text{if}\quad|j-l|\geq2.
\end{aligned}
\end{equation}

With the help of those dyadic blocks, we give the definition of homogeneous Besov spaces as follow.

\begin{defn}\label{defnbesov}
For $s\in \mathbb{R}$ and $1\leq p,r\leq \infty$, the  homogeneous Besov space $\dot{B}^{s}_{p,r}$ is defined by
$$
\dot{B}^{s}_{p,r}:=\big{\{} u\in \mathcal{S}_{h}'~|~\|u\|_{\dot{B}^{s}_{p,r}}:=\|\{2^{js}\|\dot{\Delta}_{j}u\|_{L^{p}}\}_{j\in\mathbb{Z}}\|_{l^{r}}<\infty \big{\}} .
$$
\end{defn}

Next, we state a class of mixed space-time Besov spaces introduced by Chemin-Lerner \cite{chemin1}.
\begin{defn}\label{defntimespace}
For $T>0$, $s\in\mathbb{R}$ and $1\leq \varrho,r, q \leq \infty$, the space $\widetilde{L}^{\varrho}(0,T;\dot{B}^{s}_{p,r})$ is defined as
$$
\widetilde{L}^{\varrho}(0,T;\dot{B}^{s}_{p,r}):= \big{\{} u\in L^{\varrho}(0,T;\mathcal{S}'_{h})~|~ \|u\|_{\widetilde{L}^{\varrho}_{T}(\dot{B}^{s}_{p,r})}:=\|\{2^{js}\|\dot{\Delta}_{j}u\|_{L^{\varrho}_{T}(L^{p})}\}_{j\in\mathbb{Z}}\|_{l^{r}}<\infty \big{\}}.
$$
By the Minkowski inequality, it holds 
\begin{equation}\nonumber
\begin{aligned}
&\|u\|_{\widetilde{L}^{\varrho}_{T}(\dot{B}^{s}_{p,r})}\leq(\geq) \|u\|_{L^{\varrho}_{T}(\dot{B}^{s}_{p,r})}\quad\text{if}~r\geq(\leq)\rho,
\end{aligned}
\end{equation}
where $\|\cdot\|_{L^{\varrho}_{T}(\dot{B}^{s}_{p,r})}$ is the usual Lebesgue-Besov norm. Moreover, we denote
\begin{equation}\nonumber
\begin{aligned}
&\mathcal{C}_{b}(\mathbb{R}_{+};\dot{B}^{s}_{p,r}):=\big{\{} u\in\mathcal{C}(\mathbb{R}_{+};\dot{B}^{s}_{p,r})~|~\|f\|_{\widetilde{L}^{\infty}(\mathbb{R}_{+};\dot{B}^{s}_{p,r})}<\infty \big{\}}.
\end{aligned}
\end{equation}
\end{defn}

In order to restrict Besov norms to the low frequency part and the high-frequency part, we often use the following notations for any $s\in\mathbb{R}$ and $p\in[1,\infty]$:
\begin{equation}\nonumber
\left\{
\begin{aligned}
&\|u\|_{\dot{B}^{s}_{p,r}}^{\ell}:=\|\{2^{js}\|\dot{\Delta}_{j}u\|_{L^{p}}\}_{j\leq 0}\|_{l^r},\quad \quad\quad\quad~\|u\|_{\dot{B}^{s}_{p,r}}^{h}:=\|\{2^{js}\|\dot{\Delta}_{j}u\|_{L^{p}}\}_{j\geq-1}\|_{l^r},\\
&\|u\|_{\widetilde{L}^{\varrho}_{T}(\dot{B}^{s}_{p,r})}^{\ell}:=\|\{2^{js}\|\dot{\Delta}_{j}u\|_{L^{\varrho}_{T}(L^{p})}\}_{j\leq 0}\|_{l^r},\quad \|u\|_{\widetilde{L}^{\varrho}_{T}(\dot{B}^{s}_{p,r})}^{h}:=\|\{2^{js}\|\dot{\Delta}_{j}u\|_{L_{T}^{\varrho}(L^{p})}\}_{j\geq-1}\|_{l^r}.
\end{aligned}
\right.
\end{equation}
Define
$$
u^{\ell}:=\sum_{j\leq -1}\dot{\Delta}_{j}u,\quad\quad u^{h}:=u-u^{\ell}=\sum_{j\geq0}\dot{\Delta}_{j}u.
$$
It is easy to check for any $s'>0$ that
\begin{equation}\label{lh}
\left\{
\begin{aligned}
&\|u^{\ell}\|_{\dot{B}^{s}_{p,r}}\lesssim \|u\|_{\dot{B}^{s}_{p,r}}^{\ell}\lesssim \|u\|_{\dot{B}^{s-s'}_{p,r}}^{\ell},\quad\quad\quad\quad\quad\quad \|u^{h}\|_{\dot{B}^{s}_{p,1}}\lesssim \|u\|_{\dot{B}^{s}_{p,r}}^{h}\lesssim \|u\|_{\dot{B}^{s+s'}_{p,r}}^{h},\\
&\|u^{\ell}\|_{\widetilde{L}^{\varrho}_{T}(\dot{B}^{s}_{p,r})}\lesssim \|u\|_{\widetilde{L}^{\varrho}_{T}(\dot{B}^{s}_{p,r})}^{\ell}\lesssim \|u\|_{\widetilde{L}^{\varrho}_{T}(\dot{B}^{s-s'}_{p,r})}^{\ell},\quad\|u^{h}\|_{\widetilde{L}^{\varrho}_{T}(\dot{B}^{s}_{p,r})}\lesssim \|u\|_{\widetilde{L}^{\varrho}_{T}(\dot{B}^{s}_{p,r})}^{h}\lesssim \|u\|_{\widetilde{L}^{\varrho}_{T}(\dot{B}^{s+s'}_{p,r})}^{h}.
\end{aligned}
\right.
\end{equation}

We recall some basic properties of Besov spaces and product estimates which will be used repeatedly in this paper. Remark that all the properties remain true for the Chemin--Lerner type spaces whose time exponent has to behave according to the H${\rm{\ddot{o}}}$lder inequality for the time variable.

The first lemma is the Bernstein inequalities, which in particular implies that $\dot{\Delta}_{j}u$ is smooth for every $u$ in any Besov spaces so that we can take direct calculations on linear equations after applying the operator $\dot{\Delta}_{j}$.
\begin{lemma}\label{lemma21}
Let $0<r<R$, $1\leq p\leq q\leq \infty$ and $k\in \mathbb{N}$. For any $u\in L^p$ and $\lambda>0$, it holds
\begin{equation}\nonumber
\left\{
\begin{aligned}
&{\rm{Supp}}~ \mathcal{F}(u) \subset \{\xi\in\mathbb{R}^{d}~| ~|\xi|\leq \lambda R\}\Rightarrow \|D^{k}u\|_{L^q}\lesssim\lambda^{k+d(\frac{1}{p}-\frac{1}{q})}\|u\|_{L^p},\\
&{\rm{Supp}}~ \mathcal{F}(u) \subset \{\xi\in\mathbb{R}^{d}~|~ \lambda r\leq |\xi|\leq \lambda R\}\Rightarrow \|D^{k}u\|_{L^{p}}\sim\lambda^{k}\|u\|_{L^{p}}.
\end{aligned}
\right.
\end{equation}
\end{lemma}

Due to the Bernstein inequalities, the Besov spaces have the following properties.
\begin{lemma}\label{lemma22}
The following properties hold{\rm:}
\begin{itemize}
\item{} For $s\in\mathbb{R}$, $1\leq p_{1}\leq p_{2}\leq \infty$ and $1\leq r_{1}\leq r_{2}\leq \infty$, it holds
\begin{equation}\nonumber
\begin{aligned}
\dot{B}^{s}_{p_{1},r_{1}}\hookrightarrow \dot{B}^{s-d(\frac{1}{p_{1}}-\frac{1}{p_{2}})}_{p_{2},r_{2}};
\end{aligned}
\end{equation}
\item{} For $1\leq p\leq q\leq\infty$, we have the following chain of continuous embedding:
\begin{equation}\nonumber
\begin{aligned}
\dot{B}^{0}_{p,1}\hookrightarrow L^{p}\hookrightarrow \dot{B}^{0}_{p,\infty}\hookrightarrow \dot{B}^{\sigma}_{q,\infty},\quad \sigma=-d(\frac{1}{p}-\frac{1}{q})<0;
\end{aligned}
\end{equation}
\item{} If $p<\infty$, then $\dot{B}^{\frac{d}{p}}_{p,1}$ is continuously embedded in the set of continuous functions decaying to 0 at infinity;
\item{} The following real interpolation property is satisfied for $1\leq p\leq\infty$, $s_{1}<s_{2}$ and $\theta\in(0,1)$:
\begin{equation}
\begin{aligned}
&\|u\|_{\dot{B}^{\theta s_{1}+(1-\theta)s_{2}}_{p,1}}\lesssim \frac{1}{\theta(1-\theta)(s_{2}-s_{1})}\|u\|_{\dot{B}^{ s_{1}}_{p,\infty}}^{\theta}\|u\|_{\dot{B}^{s_{2}}_{p,\infty}}^{1-\theta},\label{inter}
\end{aligned}
\end{equation}
which in particular implies for any $\varepsilon>0$ that
\begin{equation}\nonumber
\begin{aligned}
H^{s+\varepsilon}\hookrightarrow \dot{B}^{s}_{2,1}\hookrightarrow \dot{H}^{s};
\end{aligned}
\end{equation}
%\item{}For $1\leq p,\varrho\leq \infty$, $s\in\mathbb{R}$ and $0<\varepsilon\leq 1$, we have following log-type inequality:
%\begin{equation}
%\begin{aligned}
%&\|u\|_{\widetilde{L}^{\varrho}_{T}(\dot{B}^{s}_{p,1})}\lesssim \frac{\|u\|_{\widetilde{L}^{\varrho}_{T}(\dot{B}^{s}_{p,\infty})}}{\varepsilon}\log\Big{\{}{1+\frac{\|u\|_{\widetilde{L}^{\varrho}_{T}(\dot{B}^{s-\varepsilon}_{p,\infty})}+\|u\|_{\widetilde{L}^{\varrho}_{T}(\dot{B}^{s+\varepsilon}_{p,\infty})}}{\|u\|_{\widetilde{L}^{\varrho}_{T}(\dot{B}^{s}_{p,\infty})}}}\Big{\}};\label{log}
%\end{aligned}
%\end{equation}
\item{}
Let $\Lambda^{\sigma}$ be defined by $\Lambda^{\sigma}=(-\Delta)^{\frac{\sigma}{2}}u:=\mathcal{F}^{-}\big{(} |\xi|^{\sigma}\mathcal{F}(u) \big{)}$ for $\sigma\in \mathbb{R}$ and $u\in\dot{S}^{'}_{h}$, then $\Lambda^{\sigma}$ is an isomorphism from $\dot{B}^{s}_{p,r}$ to $\dot{B}^{s-\sigma}_{p,r}$;
\item{} Let $1\leq p_{1},p_{2},r_{1},r_{2}\leq \infty$, $s_{1}\in\mathbb{R}$ and $s_{2}\in\mathbb{R}$ satisfy
    $$
    s_{2}<\frac{d}{p_{2}}\quad\text{\text{or}}\quad s_{2}=\frac{d}{p_{2}}~\text{and}~r_{2}=1.
    $$
    The space $\dot{B}^{s_{1}}_{p_{1},r_{1}}\cap \dot{B}^{s_{2}}_{p_{2},r_{2}}$ endowed with the norm $\|\cdot \|_{\dot{B}^{s_{1}}_{p_{1},r_{1}}}+\|\cdot\|_{\dot{B}^{s_{2}}_{p_{2},r_{2}}}$ is a Banach space and has the weak compact and Fatou properties$:$ If $u_{n}$ is a uniformly bounded sequence of $\dot{B}^{s_{1}}_{p_{1},r_{1}}\cap \dot{B}^{s_{2}}_{p_{2},r_{2}}$, then an element $u$ of $\dot{B}^{s_{1}}_{p_{1},r_{1}}\cap \dot{B}^{s_{2}}_{p_{2},r_{2}}$ and a subsequence $u_{n_{k}}$ exist such that $u_{n_{k}}\rightarrow u $ in $\mathcal{S}'$ and
    \begin{equation}\nonumber
    \begin{aligned}
    \|u\|_{\dot{B}^{s_{1}}_{p_{1},r_{1}}\cap \dot{B}^{s_{2}}_{p_{2},r_{2}}}\lesssim \liminf_{n_{k}\rightarrow \infty} \|u_{n_{k}}\|_{\dot{B}^{s_{1}}_{p_{1},r_{1}}\cap \dot{B}^{s_{2}}_{p_{2},r_{2}}}.
    \end{aligned}
    \end{equation}
\end{itemize}
\end{lemma}

\begin{lemma}
For $1\leq p,\varrho\leq \infty$, $s\in\mathbb{R}$ and $0<\varepsilon\leq 1$, we have following log-type inequality:
\begin{equation}
\begin{aligned}
&\|u\|_{\widetilde{L}^{\varrho}_{T}(\dot{B}^{s}_{p,1})}\lesssim \frac{\|u\|_{\widetilde{L}^{\varrho}_{T}(\dot{B}^{s}_{p,\infty})}}{\varepsilon}\log\Big{\{}{1+\frac{\|u\|_{\widetilde{L}^{\varrho}_{T}(\dot{B}^{s-\varepsilon}_{p,\infty})}+\|u\|_{\widetilde{L}^{\varrho}_{T}(\dot{B}^{s+\varepsilon}_{p,\infty})}}{\|u\|_{\widetilde{L}^{\varrho}_{T}(\dot{B}^{s}_{p,\infty})}}}\Big{\}}.\label{log}
\end{aligned}
\end{equation}
\end{lemma}

The following Morse-type product estimates in Besov spaces play a fundamental role in the analysis on nonlinear terms:
\begin{lemma}\label{lemma23}
The following statements hold:
\begin{itemize}
\item{} Let $s>0$ and $1\leq p,r\leq \infty$. Then $\dot{B}^{s}_{p,r}\cap L^{\infty}$ is a algebra and
    \begin{equation}\label{uv1}
\begin{aligned}
\|uv\|_{\dot{B}^{s}_{p,r}}\lesssim \|u\|_{L^{\infty}}\|v\|_{\dot{B}^{s}_{p,r}}+ \|v\|_{L^{\infty}}\|u\|_{\dot{B}^{s}_{p,r}};
\end{aligned}
\end{equation}
\item{}
Let $s_{1}$, $s_{2}$ and $p$ satisfy $2\leq p\leq \infty$, $s_{1}\leq \frac{d}{p}$, $s_{2}\leq \frac{d}{p}$ and $s_{1}+s_{2}>0$. Then we have
\begin{equation}\label{uv2}
\begin{aligned}
&\|uv\|_{\dot{B}^{s_{1}+s_{2}-\frac{d}{p}}_{p,1}}\lesssim \|u\|_{\dot{B}^{s_{1}}_{p,1}}\|v\|_{\dot{B}^{s_{2}}_{p,1}};
\end{aligned}
\end{equation}
\item{} Assume that $s_{1}$, $s_{2}$ and $p$ satisfy $2\leq p\leq \infty$, $s_{1}\leq \frac{d}{p}$, $s_{2}<\frac{d}{p}$ and $s_{1}+s_{2}\geq0$. Then it holds 
\begin{equation}\label{uv3}
\begin{aligned}
&\|uv\|_{\dot{B}^{s_{1}+s_{2}-\frac{d}{p}}_{p,\infty}}\lesssim \|u\|_{\dot{B}^{s_{1}}_{p,1}}\|v\|_{\dot{B}^{s_{2}}_{p,\infty}}.
\end{aligned}
\end{equation}
\end{itemize}
\end{lemma}

We state the following result about the continuity for composition functions:
\begin{lemma}\label{lemma24}
Let $G:I\rightarrow \mathbb{R}$ be a smooth function satisfying $G(0)=0$. For any $1\leq p\leq \infty$, $s>0$ and $g\in\dot{B}^{s}_{2,1}\cap L^{\infty}$, there holds $G(g)\in \dot{B}^{s}_{p,r}\cap L^{\infty}$ and
\begin{align}
\|G(g)\|_{\dot{B}^{s}_{p,r}}\leq C_{g}\|g\|_{\dot{B}^{s}_{p,r}},\label{F1}
\end{align}
where the constant $C_{g}>0$ depends only on $\|g\|_{L^{\infty}}$, $G'$, $s$ and $d$.

In addition, if $g_{1}, g_{2}\in \dot{B}^{s}_{p,1}\cap L^{\infty}$, then it holds
\begin{align}
&\|G(g_{1})-G(g_{2})\|_{\dot{B}^{s}_{p,1}}\leq C_{g_{1},g_{2}}(1+\|(g_{1},g_{2})\|_{\dot{B}^{\frac{d}{2}}_{2,1}})\|g_{1}-g_{2}\|_{\dot{B}^{s}_{p,1}},\quad ~s\in (-\frac{d}{2},\frac{d}{2}],\label{F2}\\
&\|G(g_{1})-G(g_{2})\|_{\dot{B}^{s}_{p,\infty}}\leq C_{g_{1},g_{2}}(1+\|(g_{1},g_{2})\|_{\dot{B}^{\frac{d}{2}}_{2,1}})\|g_{1}-g_{2}\|_{\dot{B}^{s}_{p,\infty}},\quad s\in (-\frac{d}{2},\frac{d}{2}),\label{F3}
\end{align}
where the constant $C_{g_{1},g_{2}}>0$ depends only on $\|(g_{1},g_{2})\|_{L^{\infty}}$, $G'$, $s$, $p$ and $d$.

\end{lemma}

Finally, the following commutator estimates will be useful to control the nonlinearities in high frequencies:
\begin{lemma}\label{lemma25}
Let $1\leq p\leq \infty$ and $-\frac{d}{p}-1\leq s\leq 1+\frac{d}{p}$. Then it holds
\begin{align}
&\sum_{j\in\mathbb{Z}}2^{js}\|[u\cdot \nabla ,\dot{\Delta}_{j}]a\|_{L^{p}}\lesssim\|u\|_{\dot{B}^{\frac{d}{p}+1}_{p,1}}\|a\|_{\dot{B}^{s}_{p,1}},\label{commutator}\\
&\sum_{j\in\mathbb{Z}}2^{j(s-1)}\|[u\cdot \nabla ,\partial_{k} \dot{\Delta}_{j}]a\|_{L^{p}}\lesssim\|u\|_{\dot{B}^{\frac{d}{p}+1}_{p,1}}\|a\|_{\dot{B}^{s}_{p,1}},\quad k=1,...,d,\label{commutator1}
\end{align}
with the commutator $[A,B]:=AB-BA$.
\end{lemma}

\subsection{Estimates for some linear equations}

We first consider the linear transport equation:
\begin{equation}\label{trans}
\left\{
\begin{aligned}
&\partial_{t}f+u\cdot \nabla f=F,\quad\quad x\in\mathbb{R}^{d},\quad t>0,\\
&f(x,0)=f_{0}(x),\quad\quad\quad x\in\mathbb{R}^{d}.
\end{aligned}
\right.
\end{equation}

\begin{lemma}[\!\!\cite{bahouri1}]\label{lemma31}
Let $T>0$, $-\frac{d}{2}<s\leq 1+\frac{d}{2}$, $1\leq r\leq \infty$, $f_{0}\in\dot{B}^{s}_{2,1}$, $u\in L^1(0,T;\dot{B}^{\frac{d}{2}+1}_{2,1})$ and $F\in L^1(0,T;\dot{B}^{s}_{2,1})$. Then, for any solution $f$ to $(\ref{trans})$ in the sense of distributions, there exists a constant $C$ depending only on $s$, $d$ and $r$ such that for any $t\in[0,T]$, it holds
\begin{equation}\nonumber
\begin{aligned}
&\|f\|_{\widetilde{L}^{\infty}_{t}(\dot{B}^{s}_{2,r})}\leq e^{C\|u\|_{L^1_{t}(\dot{B}^{\frac{d}{2}+1}_{2,1})}}\big{(}\|f_{0}\|_{\dot{B}^{s}_{2,r}}+\int_{0}^{t}\|F\|_{\dot{B}^{s}_{2,r}}d\tau\big{)}.
\end{aligned}
\end{equation}
If $r<\infty$, we have $f\in\mathcal{C}([0,T];\dot{B}^{s}_{2,r})$.
\end{lemma}

Next, for the equation
\begin{equation}\label{mass}
\left\{
\begin{aligned}
&\partial_{t}a+u\cdot \nabla a =(1+a)\div u,\quad\quad x\in\mathbb{R}^{d},\quad t>0,\\
&a(x,0)=a_{0}(x),\quad\quad\quad \quad\quad\quad\quad~~ x\in\mathbb{R}^{d},
\end{aligned}
\right.
\end{equation}
we have the following estimates:

\begin{lemma}[\!\!\cite{danchin2}]\label{lemma32}
Let $T>0$, $a_{0}\in\dot{B}^{\frac{d}{2}}_{2,1}$ and $u\in L^1(0,T;\dot{B}^{\frac{d}{2}+1}_{2,1})$. Then there exists a constant $C$ depending only on $d$ such that if $a$ is a solution to $(\ref{mass})$ in the sense of distributions, then any $t\in[0,T]$ and $m\in\mathbb{Z}$, we have
\begin{align}
&~~\|a\|_{\widetilde{L}^{\infty}_{t}(\dot{B}^{\frac{d}{2}}_{2,1})}\leq e^{C\| u\|_{L^1_{t}(\dot{B}^{\frac{d}{2}+1}_{2,1})}}\|a_{0}\|_{\dot{B}^{\frac{d}{2}}_{2,1}}+e^{C\|u\|_{L^1_{t}(\dot{B}^{\frac{d}{2}+1}_{2,1})}}-1,\label{lemma321}\\
&\sum_{j\geq m}2^{\frac{d}{2}j}\|\dot{\Delta}_{j}a\|_{L^{\infty}_{t}(L^2)}\leq \sum_{j\geq m}2^{\frac{d}{2}j}\|\dot{\Delta}_{j}a_{0}\|_{L^2}+(1+\|a_{0}\|_{\dot{B}^{\frac{d}{2}}_{2,1}})\Big{(}e^{C\| u\|_{L^1_{t}(\dot{B}^{\frac{d}{2}+1}_{2,1})}}-1\Big{)},\label{lemma322}\\
&\sum_{j\leq m}2^{\frac{d}{2}j}\|\dot{\Delta}_{j}(a-a_{0})\|_{L^{\infty}_{t}(L^2)}\leq (1+\|a_{0}\|_{\dot{B}^{\frac{d}{2}}_{2,1}})\Big{(}e^{C\| u\|_{L^1_{t}(\dot{B}^{\frac{d}{2}+1}_{2,1})}}-1\Big{)}+C2^{m}\|a_{0}\|_{\dot{B}^{\frac{d}{2}}_{2,1}}\|u\|_{L^1_{t}(\dot{B}^{\frac{d}{2}}_{2,1})}.\label{lemma323}
\end{align}

\end{lemma}

~\

Then, consider the following parabolic system with constant coefficients:
\begin{equation}\label{heat11}
\left\{
\begin{aligned}
&\partial_{t}u-\mathcal{A}u=F,\quad \quad~~x\in\mathbb{R}^{d},\quad t>0,\\
&u(x,0)=u_{0}(x),\quad \quad x\in\mathbb{R}^{d}.
\end{aligned}
\right.
\end{equation}
We have the following classical $L^2$-regularity estimates:
\begin{lemma}[\!\!\cite{danchin2}]\label{lemma33}
Let $T>0$, $s\in\mathbb{R}$, $\nu:=\min\{\mu,2\mu+\lambda\}>0$, $1\leq \varrho_{1}\leq \varrho_{2}\leq\infty$, $\varrho_{3}=(1+\frac{1}{\varrho_{1}}-\frac{1}{\varrho_{2}})^{-1}$,  $u_{0}\in \dot{B}^{s}_{2,1}$ and $F\in \widetilde{L}^{\varrho_{2}}(0,T;\dot{B}^{s}_{2,1})$. If $u$ is a solution to $(\ref{heat11})$ in the sense of distributions, then there exists two constants $C>0$ and $c_{0}>0$ depending only on $d$ such that for any $t\in[0,T]$, it holds
\begin{equation}\nonumber
\begin{aligned}
&\|u\|_{\widetilde{L}^{\varrho_{1}}_{t}(\dot{B}^{s+\frac{2}{\varrho_{1}}}_{2,1})}\leq C\sum_{j\in\mathbb{Z}}2^{js}\|\dot{\Delta}_{j}u\|_{L^2}\Big{(}\frac{1-e^{-c_{0}\nu t2^{2j}\varrho_{1}}}{c_{0}\nu\varrho_{1}}\Big{)}^{\frac{1}{\varrho_{1}}}\\
&\quad\quad\quad\quad\quad\quad\quad+C\sum_{j\in\mathbb{Z}}2^{j(s-2+\frac{2}{\varrho_{2}})}\|\dot{\Delta}_{j}F\|_{L^{\varrho_{2}}_{t}(L^2)}\Big{(}\frac{1-e^{-c_{0}\nu t2^{2j}\varrho_{3}}}{c_{0}\nu\varrho_{3}}\Big{)}^{\frac{1}{\varrho_{3}}},
\end{aligned}
\end{equation}
which in particular implies
\begin{equation}\nonumber
\begin{aligned}
\|u\|_{\widetilde{L}^{\varrho_{1}}_{t}(\dot{B}^{s+\frac{2}{\varrho_{1}}}_{2,1})}\leq \frac{C}{(c_{0}\nu)^{\frac{1}{\varrho_{1}}}}\|u_{0}\|_{\dot{B}^{s}_{2,1}}+\frac{C}{(c_{0}\nu)^{\frac{1}{\varrho_{3}}}}\|F\|_{\widetilde{L}^{\varrho_{2}}_{t}(\dot{B}^{s-2+\frac{2}{\varrho_{2}}}_{2,1})}.
\end{aligned}
\end{equation}
Moreover, we have $u\in \mathcal{C}([0,T];\dot{B}^{s}_{2,1})$.
\end{lemma}

Then, we consider the transport-diffusion system with variable coefficients:
\begin{equation}\label{heat1}
\left\{
\begin{aligned}
&\partial_{t}u+\upsilon\cdot\nabla  u+u\cdot\nabla  \zeta-\vartheta\mathcal{A}u=F,\quad\quad x\in\mathbb{R}^{d},\quad t>0,\\
&u(x,0)=u_{0}(x),\quad\quad\quad\quad\quad\quad\quad\quad\quad~~ x\in\mathbb{R}^{d}.
\end{aligned}
\right.
\end{equation}
The following two lemmas play a key role in the local well-posednss for large initial density \cite{danchin2,danchin3}.
\begin{lemma}[\!\!\cite{danchin2,danchin3}]\label{lemma34}
Let $T>0$, $\nu:=\min\{\mu,2\mu+\lambda\}>0$, $s\in(-\frac{d}{2},\frac{d}{2}]$, $u_{0}\in \dot{B}^{s}_{2,1}$, $F\in L^1(0,T;\dot{B}^{s}_{2,1})$, $\vartheta-1\in L^2(0,T;\dot{B}^{\frac{d}{2}}_{2,1})$, $\upsilon,\zeta \in L^2(0,T;\dot{B}^{\frac{d}{2}+1}_{2,1})$ and $\vartheta$ be bounded from below by a constant $\underline{\vartheta}>0$. There exist three constants $c_{i}$ $(i=1,2,3)$ (with $c_{2},c_{3}$ depending only on $d$ and $s$ and $c_{1}$ universal) and a sufficiently large constant $m\in\mathbb{Z}$ such that if
\begin{equation}\nonumber
\left\{
\begin{aligned}
&\inf_{(x,t)\in \mathbb{R}^{d}\times[0,T]}\Big{(}1+\sum_{j\leq m-1}\dot{\Delta}_{j}(\vartheta-1)\Big{)}(x,t)\geq \frac{\underline{\vartheta}}{2},\\
&\quad\|\sum_{j\geq m}\dot{\Delta}_{j}(\vartheta-1)\|_{\widetilde{L}^{\infty}_{T}(\dot{B}^{\frac{d}{2}}_{2,1})}\leq \frac{c_{2}\underline{\vartheta}\nu}{\mu+|\mu+\lambda|},
\end{aligned}
\right.
\end{equation}
then for any solution $u$ to $(\ref{heat1})$ in the sense of distributions, we have
\begin{equation}\nonumber
\begin{aligned}
&\|u\|_{\widetilde{L}^{\infty}_{T}(\dot{B}^{s}_{2,1})}+c_{1}\underline{\vartheta}\nu\|u\|_{L^1_{T}(\dot{B}^{s+2}_{2,1})}\leq e^{c_{3}(Z_{m}(T)+W(T))}\big{(}\|u_{0}\|_{\dot{B}^{s}_{2,1}}+\|F\|_{L^1_{T}(\dot{B}^{s}_{2,1})}\big{)},
\end{aligned}
\end{equation}
with
$$
Z_{m}(t):=\frac{2^{2m}(\mu+|\mu+\lambda|)^2}{\underline{\vartheta}\nu}\int_{0}^{t}\|\vartheta-1\|_{\dot{B}^{\frac{d}{2}}_{2,1}}^2d\tau,\quad W(t):=\int_{0}^{t}\|(\upsilon,\zeta)\|_{\dot{B}^{\frac{d}{2}+1}_{2,1}}d\tau.
$$
\end{lemma}

~\

To touch the limit case $s=-\frac{d}{2}$, we need the following lemma:
\begin{lemma}[\!\!\cite{danchin2}]\label{lemma35}
Let $T>0$, $\nu:=\min\{\mu,2\mu+\lambda\}>0$, $u_{0}\in \dot{B}^{-\frac{d}{2}}_{2,\infty}$, $g\in L^1(0,T;\dot{B}^{-\frac{d}{2}}_{2,\infty})$, $\vartheta-1\in L^2(0,T;\dot{B}^{\frac{d}{2}}_{2,1})$, $\upsilon,\zeta \in L^2(0,T;\dot{B}^{\frac{d}{2}+1}_{2,1})$ and $\vartheta$ be bounded from below by a constant $\underline{\vartheta}>0$. There exist three constants $c_{i}$ $(i=4,5.6)$ (with $c_{4}$ universal and $c_{5},c_{6}$ depending only on $d$) and a sufficiently large constant $m\in\mathbb{Z}$ such that if
\begin{equation}\nonumber
\left\{
\begin{aligned}
&\inf_{(x,\tau)\in \mathbb{R}^{d}\times[0,T]}\Big{(}1+\sum_{j\leq m-1}\dot{\Delta}_{j}(\vartheta-1)\Big{)}(x,\tau)\geq \frac{\underline{\vartheta}}{2},\\
&\quad\|\sum_{j\geq m}\dot{\Delta}_{j}(\vartheta-1)\|_{\widetilde{L}^{\infty}_{T}(\dot{B}^{\frac{d}{2}}_{2,1})}\leq \frac{c_{5}\underline{\vartheta}\nu}{\mu+|\mu+\lambda|},
\end{aligned}
\right.
\end{equation}
then for any solution $u$ to $(\ref{heat1})$ in the sense of distributions and suitably small time $t\in [0,T]$ satisfying
$$
t\|\vartheta -1\|_{\widetilde{L}^{\infty}_{t}(\dot{B}_{2,1}^{\frac{d}{2}})}^2\leq \frac{c_{5}2^{-2m}\underline{\vartheta}\nu}{(\mu+|\mu+\lambda|)^2},
$$
we have
\begin{equation}\nonumber
\begin{aligned}
&\|u\|_{\widetilde{L}^{\infty}_{t}(\dot{B}^{-\frac{d}{2}}_{2,\infty})}+c_{4}\underline{\vartheta}\nu\|u\|_{L^1_{t}(\dot{B}^{2-\frac{d}{2}}_{2,\infty})}\leq 2e^{c_{6}W(t)}\big{(}\|u_{0}\|_{\dot{B}_{2,\infty}^{-\frac{d}{2}}}+\|G\|_{L^1_{t}(\dot{B}^{-\frac{d}{2}}_{2,\infty})}\big{)}.
\end{aligned}
\end{equation}
\end{lemma}

~\

Finally, we need the optimal regularity estimates for the Lam\'e system with constant coefficients.
\begin{lemma}\label{heat}
Let $T>0$, $\mu>0$, $2\mu+\lambda>0$, $s\in\mathbb{R}$, $1\leq p,r\leq\infty$ and $1\leq\varrho_{2}\leq\varrho_{1}\leq\infty$. Assume that $u_{0}\in\dot{B}^{s}_{p,r}$ and $f\in\widetilde{L}^{\rho_{2}}(0,T;\dot{B}^{s-2+\frac{2}{\varrho_{2}}}_{p,r})$ hold. If $u$ is a solution of
\begin{equation}\nonumber
\left\{
\begin{aligned}
&\partial_{t}u-\mu \Delta u-(\mu+\lambda)\nabla \div u=f,\quad x\in\mathbb{R}^{d},\quad t\in(0,T),\\
&u(x,0)=u_{0}(x),~\quad\quad\quad\quad\quad\quad\quad\quad x\in\mathbb{R}^{d},
\end{aligned}
\right.
\end{equation}
then $u$ satisfies
\begin{equation}\nonumber
\begin{aligned}
&\min\{\mu,2\mu+\lambda\}^{\frac{1}{\varrho_{1}}}\|u\|_{\widetilde{L}^{\varrho_{1}}_{T}(\dot{B}^{s+\frac{2}{\varrho_{1}}}_{p,r})}\lesssim \|u_{0}\|_{\dot{B}^{s}_{p,r}}+\|f\|_{\widetilde{L}^{\rho_{2}}_{T}(\dot{B}^{s-2+\frac{2}{\varrho_{2}}}_{p,r})}.
\end{aligned}
\end{equation}
\end{lemma}

\vspace{2ex}
%\noindent
\textbf{Acknowledgments}
The authors thank the referees for their valuable suggestions and comments on the manuscript.  Part of this work was done when the second author was visiting Universit\'e Paris-Est in 2020.  The second author is grateful to Professor R. Danchin and Dr. T. Crin-Barat for their helpful discussions. The research of the paper is supported by National Natural Science Foundation of China (No.11931010 and 11871047) and by the key research project of Academy for Multidisciplinary Studies, Capital Normal University, and by the Capacity Building for Sci-Tech Innovation-Fundamental Scientific Research Funds (No.007/20530290068).

\end{document}